\documentclass{amsart}
\usepackage{amsmath}
\usepackage{amssymb}
\usepackage{graphicx}
\usepackage{latexsym}
\usepackage{color}
\usepackage{amscd}
\usepackage[all]{xy}
\usepackage{enumerate}
\usepackage{enumitem}
\usepackage{subfigure}
\usepackage{flafter}
\usepackage{float}
\IfFileExists{soul.sty}{\usepackage{soul}}{}
\usepackage{comment}
\usepackage{epsfig}
\usepackage{epstopdf}
\usepackage{listings}
\usepackage{xcolor}
\IfFileExists{ulem.sty}{\usepackage[normalem]{ulem}}{}
\usepackage{tikz}
\usetikzlibrary{arrows.meta,calc,decorations.markings,positioning}
\usepackage{multirow}
\newtheorem{thma}{Theorem}

\newtheorem{thmb}{Theorem}

\newtheorem{thm}{Theorem}[section]

\newtheorem{theorem}[thm]{Theorem}
\newtheorem{lemma}[thm]{Lemma}
\newtheorem{corollary}[thm]{Corollary}
\newtheorem{proposition}[thm]{Proposition}

\newtheorem{conjecture}[thm]{Conjecture}
\theoremstyle{definition}

\newtheorem*{definition*}{Definition}

\newtheorem{remark}[thm]{Remark}

\newtheorem{example}[thm]{Example}

\newcommand{\CP}{{\mathbb{CP}}{}^{2}}

\newcommand{\RP}{{\mathbb{RP}}{}^{2}}

\newcommand{\Q}{\mathbb{Q}}

\newcommand{\Z}{\mathbb{Z}}

\newcommand{\twprod}{\mathbin{\mathchoice{\ooalign{\raise1.15ex\hbox{$\scriptstyle\sim$}\cr\hidewidth$\times$\hidewidth\cr}}{\ooalign{\raise1.15ex\hbox{$\scriptstyle\sim$}\cr\hidewidth$\times$\hidewidth\cr}}{\ooalign{\raise.85ex\hbox{$\scriptscriptstyle\sim$}\cr\hidewidth$\scriptstyle\times$\hidewidth\cr}}{\ooalign{\raise.65ex\hbox{$\scriptscriptstyle\sim$}\cr\hidewidth$\scriptscriptstyle\times$\hidewidth\cr}}}}

\usepackage{scalerel}

\def \x {\times}
\def \eu{{\text{e}}}

\newcommand{\id}{\operatorname{id}}
\allowdisplaybreaks[4]
\definecolor{alphacolor}{RGB}{0,91,150}
\definecolor{betacolor}{RGB}{102,83,124}
\definecolor{fixedcolor}{RGB}{0,112,91}
\definecolor{blockcolor}{RGB}{166,32,105}
\definecolor{lineone}{RGB}{28,92,151}
\definecolor{linetwo}{RGB}{178,92,0}
\definecolor{linethree}{RGB}{73,132,57}
\definecolor{linefour}{RGB}{129,73,153}
\definecolor{linefive}{RGB}{160,45,55}
\definecolor{auxblue}{RGB}{36,151,219}
\tikzset{
  pagebdry/.style={draw=black!78,fill=black!5,line width=.72pt,rounded corners=8pt},
  hole/.style={draw=black!78,fill=white,line width=.58pt},
  alpha curve/.style={draw=alphacolor,line width=1.15pt,line cap=round,line join=round},
  beta curve/.style={draw=betacolor,line width=1.15pt,line cap=round,line join=round},
  fixed curve/.style={draw=fixedcolor,line width=1.15pt,line cap=round,line join=round},
  block curve/.style={draw=blockcolor,line width=1.15pt,line cap=round,line join=round},
  auxiliary blue curve/.style={draw=auxblue,line width=1.15pt,line cap=round,line join=round},
  basisstem/.style={draw=black!52,line width=.56pt,line cap=round,line join=round},
  basisloop/.style={draw=black!72,line width=.72pt,line cap=round,line join=round},
  arrpoint/.style={circle,fill=black,inner sep=1.3pt}
}
\usepackage{mathtools,array,booktabs}
\usepackage{hyperref}
\hypersetup{hidelinks}
\DeclareMathOperator{\Mod}{Mod}
\DeclareMathOperator{\Fix}{Fix}
\DeclareMathOperator{\rank}{rank}
\DeclareMathOperator{\Inn}{Inn}
\newcommand{\ZZ}{\mathbb Z}
\newcommand{\F}{\mathbb F}
\newcommand{\PP}{\mathbb P}

\newcommand{\Sig}{\Sigma}
\newcommand{\al}{\alpha}
\newcommand{\be}{\beta}
\newcommand{\ga}{\gamma}
\newcommand{\ov}[1]{\overline{#1}}
\newcommand{\file}[1]{\nolinkurl{#1}}
\numberwithin{equation}{section}

\newcommand{\eps}{\varepsilon}
\newcommand{\cN}{\mathcal N}
\newcommand{\cL}{\mathcal L}
\newcommand{\cO}{\mathcal O}

\newcommand{\fv}{\mathfrak v}
\newcommand{\CS}{\mathrm{CS}}
\begin{document}
\title[Small Lefschetz fibrations on the BMY line]
{Small Lefschetz fibrations on the BMY line}
\author[R. \.{I}. Baykur]{R. \.{I}nan\c{c} Baykur}
\address{Department of Mathematics and Statistics, University of Massachusetts,
Amherst, MA 01003, USA}
\email{inanc.baykur@umass.edu}
\author[A. I. Stipsicz]{Andr\'as I. Stipsicz}
\address{HUN-REN Alfr\'ed R\'enyi Institute of Mathematics, Budapest, Hungary}
\email{stipsicz.andras@renyi.hu}
\author[Z. Szab\'o]{Zolt\'an Szab\'o}
\address{Department of Mathematics, Princeton University, Princeton, NJ 08544, USA}
\email{szabo@math.princeton.edu}

\begin{abstract}
We give explicit positive factorizations for a genus-19 Lefschetz
fibration over the torus and a family of genus-4 Lefschetz fibrations
over a genus-2 surface. All these examples have only three nodes,
and their total spaces are symplectic $4$--manifolds of general type
on the Bogomolov--Miyaoka--Yau line. In the course of deriving the monodromy relation in genus 19, we obtain several interesting identities, including a factorization of a Dehn twist into a product of three order-$3$ mapping classes. We show that the total space of this fibration is topologically $s$-cobordant to the Cartwright--Steger surface, and we determine its real quotient up
to homeomorphism. The genus-4 family consists of variations of the Bauer--Catanese fibrations and contains 65 isomorphism classes
of Lefschetz fibrations on 45 diffeomorphism types of compact complex ball quotients. We describe their topology and that of the real quotients
of the three original surfaces. The handle monodromies in our main factorizations are pseudo-Anosov, and we show more generally that for holomorphic genus-4 Lefschetz fibrations on the BMY line, this holds for every choice of handle generators.
\end{abstract}
\maketitle

\setcounter{secnumdepth}{3}
\setcounter{section}{0}
\section{Introduction}\label{sec:intro}

By Yau's uniformization theorem \cite{Yau}, compact complex surfaces of
general type on the Bogomolov--Miyaoka--Yau (BMY) line
$\eu(X)=3\sigma(X)>0$ are precisely compact ball quotients.
We give explicit positive factorizations for Lefschetz fibrations
on this line with only three critical points.

Our first theorem presents the monodromy of a symplectic Lefschetz
fibration on the BMY line with the same Chern numbers as a fake
projective plane.

\begin{thma}\label{thm:fibration}\label{thm:genus19intro}
There are mapping classes $P,Q,R\in\Mod(\Sigma_{19})$ of order $3$,
a nonseparating curve $c$, and an orientation-reversing
involution $C$ such that
\begin{equation}\label{eq:intro-periodic}
 PQR\,t_c=1,\qquad
 CPC=R^{-1},\qquad CQC=Q^{-1},\qquad C(c)=c.
\end{equation}
These give a genus-19 Lefschetz fibration $\pi:X\to E$ over the
torus with three nodal fibers and monodromy factorization
\begin{equation}\label{eq:intro-genus19}
 \begin{gathered}
 t_c t_{P(c)}t_{P^2(c)}[A,B]=1 \quad\text{in }\Mod(\Sigma_{19}),\\
 \text{where}\qquad A=QR^{-1},\quad B=R^{-1}Q.
 \end{gathered}
\end{equation}
The three vanishing cycles are nonseparating and meet pairwise
once. The handle monodromies $A,B$ are conjugate and pseudo-Anosov.
Moreover, $P=U^{-26}$ and $R=V^{-26}$, where $U,V$ are the
positive Dehn twist products along two $38$--chains exchanged
by $C$; see Theorem~\ref{cor:chain}.

The total space is a minimal symplectic $4$--manifold with
\[
 (\eu(X),\sigma(X))=(3,1),\qquad H_1(X;\Z)=\Z^2.
\]
It is topologically $s$-cobordant to the Cartwright--Steger
surface $\CS$. Moreover, the fibration admits a real involution
$\tau$ with fixed surface $\# 3\RP$ and quotient
\[
 X/\tau\approx\CP\#(S^1\times S^3).
\]
\end{thma}

By \emph{handle monodromy} we refer to the monodromy along
a standard handle generator of the punctured base surface.

An explicit monodromy for such a fibration has long eluded us; we
obtained it in two different ways. The first uses
periodic maps, a reflection, and a local Dehn twist relation to
obtain \eqref{eq:intro-periodic}.
Its cyclic triple pullback gives the fibration over the torus.
The second started with two explicit $38$--chains; their relative
placement led to the periodic relation in chain form.
Proposition~\ref{prop:models} identifies their monodromies by
Hurwitz moves and conjugation. The periodic maps also satisfy
$(PQ)^{12}=t_{\beta_0}^{-1}t_{\beta_1}^{-4}t_{\beta_2}^{-4}t_\gamma^3$;
the four disjoint curves are described in Proposition~\ref{thm:power}.

The Lefschetz fibration of Theorem~\ref{thm:fibration} is modeled on the Albanese
fibration of $\CS$ \cite{CS,CKY,KY}, and we are convinced that $X$ is
diffeomorphic to the Cartwright--Steger surface and that $\pi$ is
its Albanese fibration. In fact, the first construction was found by
studying the monodromy of this Albanese fibration directly. The
polynomial equations for the Albanese map \cite{SzaboAlbanese}
present the generic fiber as a genus-$19$ curve in $\PP^9$, and one
can then see how the curve moves as the base point $p$ varies in
the Albanese torus.
This can be done in several ways. The simplest approach we found
was to present the curve as a degree-$24$ branched cover of $\PP^1$
with $72$ finite simple branch values, with transposition monodromies, and
monodromy $(1\,2)(3\,4)\cdots(23\,24)$ over infinity (so that
$2g-2=-48+72+12=36$), and to continue the $72$ branch values
analytically along loops in the Albanese torus. This gives the
action on $\pi_1(\Sigma_{19})$ and, in the end, enough information
to find monodromies $P,Q,R$ that satisfy the identities proved in
this paper. It is a long and not particularly enlightening method,
but we may return to it in the future.

Our proof of the quotient homeomorphism in
Section~\ref{g19:sec:realquotient} uses the monodromy and involution directly, without assuming the identification with $\CS$. If the
identification also matches the real involutions, our argument
gives another proof of the quotient homeomorphism result in \cite[Theorem~1.1(b)]{SSQ}.

For our second main result, we start from the three Bauer--Catanese surfaces (BC) with positive
first Betti number \cite{BC,CD}, which admit holomorphic genus-4
Lefschetz fibrations over a genus-2 surface $G$. We give
monodromies for all their variations obtained by twisting with characters in
$H^1(G;\Z_5)$. The three vanishing cycles---and their order---are the same throughout
this family, whereas the handle monodromies differ by powers of an order-$5$ deck transformation $Z$.
We write $\lambda=(r_1,s_1,r_2,s_2)$ for the values of a character
on these generators.

\begin{thmb}\label{thm:smallfibrations}
For $\lambda=(r_1,s_1,r_2,s_2)\in(\Z_5)^4$, the words $A_j,B_j,Z$
in \eqref{g4:eq:center}--\eqref{g4:eq:handles} give holomorphic
genus-4 Lefschetz fibrations $f_\lambda\colon X_\lambda\to G$
with monodromy factorizations
\[
 t_{d_1}t_{d_2}t_{d_3}
 [Z^{r_1}A_1,Z^{s_1}B_1][Z^{r_2}A_2,Z^{s_2}B_2]
 = 1
 \quad\text{in }\Mod(\Sigma_4),
\]
where the exponents are integer representatives of the coordinates
of $\lambda$. The curves $d_i$ are given in \eqref{g4:eq:cycles};
each separates the closed fiber into two genus-2 components.
Each fibration has four disjoint sections of square $-1$ and all four handle monodromies for any $\lambda$ are pseudo-Anosov.

The total spaces are compact ball quotients with invariants
\[
 (\eu(X_\lambda),\sigma(X_\lambda))=(15,5),\qquad
 H_1(X_\lambda;\Z)=\Z^4\oplus(\Z_5)^2.
\]
The family $\{ (X_\lambda, f_\lambda)\}$ contains exactly 65 isomorphism classes of
Lefschetz fibrations and 45  diffeomorphism types of total
spaces, where the latter are distinguished by their fundamental groups.

Moreover, the three original BC surfaces
$X_k:=X_{k(1,-1,1,2)}$, $k=0,1,4$, admit real involutions $\kappa_k$ with fixed surface $\# 5\RP$. Their quotients
$Y_k=X_k/\kappa_k$ have fundamental group and integral intersection form
\[
 \pi_1(Y_k)\cong\Z\star\Z,\qquad
 Q_{Y_k}\cong6\langle1\rangle\oplus2\langle-1\rangle.
\]
\end{thmb}

For every $\lambda$, the four handle monodromies of $(X_\lambda, f_\lambda)$ are words in the
Dehn twists about the curves of the same two fixed $4$--chains. 
The three original BC fibrations are precisely, up to isomorphism,
the members for which
the order-$5$ symmetry of the base lifts. Section~\ref{sec:genus4}
gives the construction and classification, and
Section~\ref{sec:realquotients} treats the real quotients.

Our motivation for finding these explicit monodromies comes from
the question of whether the BMY line contains symplectic $4$--manifolds of general
type that admit no complex structure. When $\eu>3\sigma$, many such
examples arise by varying monodromy factorizations; see, for example,
\cite{BH,BK}. The small fibrations of
Theorems~\ref{thm:genus19intro} and~\ref{thm:smallfibrations}
provide starting points for this question: both have only three
critical points, and $X$ has the smallest positive Euler
characteristic possible on the BMY line. With this motivation, Section~\ref{sec:properties} studies the
Nielsen--Thurston types of handle monodromies on the BMY line.
For a holomorphic Lefschetz fibration on a compact ball quotient,
Proposition~\ref{prop:handle-obstruction} shows that monodromy
along a loop nontrivial in the closed base has a pseudo-Anosov
component and no multitwist power. In genus 4, rigidity in
Toledo's inequality \cite{Toledo,BIW} also excludes reducibility,
so every such monodromy is pseudo-Anosov
(Corollary~\ref{cor:all-handle-pa4}). In particular, this holds
for every choice of handle generators.
We finish Section~\ref{sec:properties} by describing larger
Lefschetz fibrations on the BMY line obtained as finite covers
of these examples. We identify Z.~Chen's genus-16 fibration
\cite{Chen} as a cover of $X_1$, and obtain a genus-127 fibration
from a $7$--fold cover of $X$ using Borisov--Yeung's covering
of the CS surface \cite{BY}.

\smallskip
\noindent\emph{Conventions.}
All $4$--manifolds below are oriented, and homeomorphisms and diffeomorphisms
between them preserve orientation unless stated otherwise.
For manifolds, $\approx$ denotes homeomorphism and $\cong$ denotes
diffeomorphism. We use $\simeq$ for homotopy equivalence of spaces
and homotopy of maps.
We write $\Z_m$ for the cyclic group of order $m$.
Let $\Sigma_g^b=\Sigma_{g,b}$ be a compact connected oriented surface of genus $g$
with $b$ boundary components, and set
$\Gamma_g^b=\Mod(\Sigma_g^b,\partial\Sigma_g^b)$ and
$\Gamma_g=\Mod(\Sigma_g)$. Boundary-fixing maps fix the boundary
pointwise. All Dehn twists $t_c$ are right-handed. Products act on the rightmost
factor first, so
\[
 [A,B]=ABA^{-1}B^{-1},\qquad
 \psi t_c\psi^{-1}=t_{\psi(c)}
 \quad\text{for orientation-preserving }\psi.
\]
An orientation-reversing $\psi$ reverses the Dehn twist sign.
For a genus-$g$ Lefschetz fibration over a genus-$h$ base, we use
\[
 t_{d_1}\cdots t_{d_n}\,{\textstyle\prod}_{j=1}^h[A_j,B_j]=1
 \quad\text{in }\Mod(\Sigma_g),
\]
where $A_j,B_j$ are the monodromies along the standard handle
generators of the punctured base. A one-boundary lift with
right-hand side $t_\partial$ records a section of square $-1$.
Isomorphisms of Lefschetz fibrations are smooth and preserve
the orientations of the base and fibers.

\begin{samepage}
\smallskip
\noindent\emph{Acknowledgments.}
OpenAI's ChatGPT and Anthropic's Claude were used for calculations, coding, figure generation, and editing.
The authors take full responsibility for the mathematical arguments and results.\par
\end{samepage}

\section{A genus-19 Lefschetz fibration over the torus}\label{sec:genus19}\label{sec:proof}

Section~\ref{ssec:two-constructions} deduces the torus monodromy from
the periodic relation. Sections~\ref{ssec:local}--\ref{ssec:symmetric}
give the geometric construction, and Section~\ref{sec:fibration}
identifies its cyclic triple pullback. Section~\ref{sec:chain} gives
the chain form and the coordinates used in the group calculations.
Section~\ref{ssec:invariants} determines the topology of the total space.

\subsection{The periodic monodromy}\label{ssec:two-constructions}

The construction uses periodic maps and a reflection.
We build a model of $\Sigma$ by gluing two explicit pieces along
twelve circles. It carries an $S_3$-action generated by a map $Q$
of order $3$ and an orientation-reversing involution $C$, with
\begin{equation}\label{eq:S3}
 Q^3=C^2=1,\qquad CQC=Q^{-1}.
\end{equation}
We call $\Sigma$ with this action the $S_3$-model. It contains a
nonseparating curve $c$ preserved by $C$. The map $Q$ has three
fixed points, local rotation $4\pi/3$ at each, and quotient of
genus $6$. Here and below, local rotations are measured
counterclockwise in oriented charts. The next theorem gives the
periodic factorization used in this construction.

\begin{samepage}
\begin{theorem}[The periodic factorization]\label{thm:identity}
There is a $P\in\Mod(\Sigma)$ of order $3$ such that, with
$R=CP^{-1}C$,
\begin{equation}\label{eq:main}
 PQRt_c=1 .
\end{equation}
The periodic representatives of $P$ and $R$ have three fixed points,
quotient of genus $6$, and local rotation $2\pi/3$ at every fixed point;
in particular $P,R,Q^{-1}$ are conjugate in $\Mod(\Sigma)$.
\end{theorem}
\end{samepage}

To construct $P$, we use the cyclic covering model of
Section~\ref{ssec:periodic} and place it on $\Sigma$ by an
identification $\Psi$. If $\rho$ is its deck map, then
$P=C\rho C$ and $R=\rho^{-1}$. These maps satisfy
\begin{equation}\label{eq:real}
 CPC=R^{-1},\qquad CQC=Q^{-1},\qquad C(c)=c .
\end{equation}

The rotation data also give a chain form of the relation. Its proof
in Section~\ref{sec:chain} uses Nielsen classification to choose a
$38$--chain whose Dehn twist product has the required periodic power.

\begin{theorem}[The chain factorization]\label{cor:chain}
There is a $38$-chain $\fv=(v_1,\dots,v_{38})$ on $\Sigma$ such that, with
$u_i=C(v_{39-i})$, $U=t_{u_1}\cdots t_{u_{38}}$ and
$V=t_{v_1}\cdots t_{v_{38}}$,
\begin{equation}\label{eq:chain}
 t_c=V^{26}\,Q^{-1}\,U^{26}.
\end{equation}
\end{theorem}

The chain relation gives $U^{78}=V^{78}=1$, and $U^{26},V^{26}$
have order $3$ (Lemma~\ref{lem:chain-rotation}). We first use the
periodic relation to obtain the torus factorization in
Theorem~\ref{thm:genus19intro}.

\begin{lemma}\label{lem:Z3Z3}
If $Y^3=Z^3=1$ in a group, then $[YZ^{-1},Z^{-1}Y]=(YZ)^3$.
\end{lemma}

\begin{proof}
Using $Z^{-2}=Z$ and $Y^{-2}=Y$,
$[YZ^{-1},Z^{-1}Y]=YZ^{-2}Y\cdot ZY^{-2}Z=YZY\cdot ZYZ=(YZ)^3$.
\end{proof}

Apply this identity to the periodic factors $Q,R$ to obtain the
handle monodromies in Theorem~\ref{thm:genus19intro}.

\begin{lemma}\label{cor:commutators}
Put $A=QR^{-1}$ and $B=R^{-1}Q$. Then
\begin{equation}\label{eq:commutators}
 [A,B]\,t_c\,t_{Pc}\,t_{P^2c}=1 .
\end{equation}
\end{lemma}

\begin{proof}
By~\eqref{eq:main}, $t_c P=(QR)^{-1}$, and $P^3=1$ gives
$(t_c P)^3=t_c t_{Pc}t_{P^2c}$. Lemma~\ref{lem:Z3Z3}
with $Y=Q$ and $Z=R$ gives $(QR)^3=[A,B]$.
\end{proof}

In the chain description, these handles are $A=QV^{26}$ and
$B=V^{26}Q$, and the vanishing cycles are
$c,U^{-26}(c),U^{26}(c)$. Thus \eqref{eq:commutators} is the
same monodromy factorization written in chain form.

\subsection{A local reflection relation}\label{ssec:local}
We work on the $3$--holed torus $\mathcal T$ in
Figure~\ref{fig:local-surface}. The commutator relation below gives
the product of five Dehn twists in
Corollary~\ref{cor:local-reflection}, which conjugates the reflection
$C_{\mathcal T}$ to $t_{c_{\mathcal T}}C_{\mathcal T}$.

Give $\mathbb T^2=\mathbb R^2/\Z^2$ the orientation $dx\wedge dy$ and fix
$0<r<1/32$. The union ${\mathcal T}^\square=\{|x|\le4r\}\cup\{|y|\le4r\}$ is a
once-holed torus. Remove the open disks of radius $r$ about $P_1=(6r,0)$
and $P_2=(0,6r)$, and call the result ${\mathcal T}$;
write $\partial_0{\mathcal T}$ for the original boundary and $\partial_1{\mathcal T},\partial_2{\mathcal T}$
for the new ones. The reflection $C_{\mathcal T}(x,y)=(y,x)$ preserves $\partial_0{\mathcal T}$
and exchanges $\partial_1{\mathcal T}$ and $\partial_2{\mathcal T}$. Let
\[
 a_1=\{y=-3r\},\quad a_2=\{y=3r\},\quad b_1=\{x=3r\},\quad b_2=\{x=-3r\}.
\]
Thus $b_1=C_{\mathcal T}(a_2)$ and $b_2=C_{\mathcal T}(a_1)$. Let $e_j$ be the circle of radius $2r$
about $P_j$, and let $c_{\mathcal T}$ be the boundary of a regular neighborhood
of the two deleted disks and the segment $P_1P_2$. It is $C_{\mathcal T}$-invariant,
disjoint from $a_1$ and $b_2$, and bounds a pair of pants with $\partial_1{\mathcal T}$
and $\partial_2{\mathcal T}$. Put
\[
 H=t_{a_1}t_{a_2}^{-1},\qquad V_{\mathcal T}=t_{b_1}t_{b_2}^{-1}.
\]
Since $C_{\mathcal T}(e_1)=e_2$ and $C_{\mathcal T}$ reverses orientation,
\begin{equation}\label{eq:local-reflection-action}
 C_{\mathcal T}HC_{\mathcal T}=V_{\mathcal T},\qquad C_{\mathcal T}V_{\mathcal T}C_{\mathcal T}=H,\qquad C_{\mathcal T}t_{e_1}C_{\mathcal T}=t_{e_2}^{-1}
\end{equation}
in the boundary-fixing group. Although $C_{\mathcal T}$ itself
does not fix the boundary pointwise, conjugation by it preserves
this group.

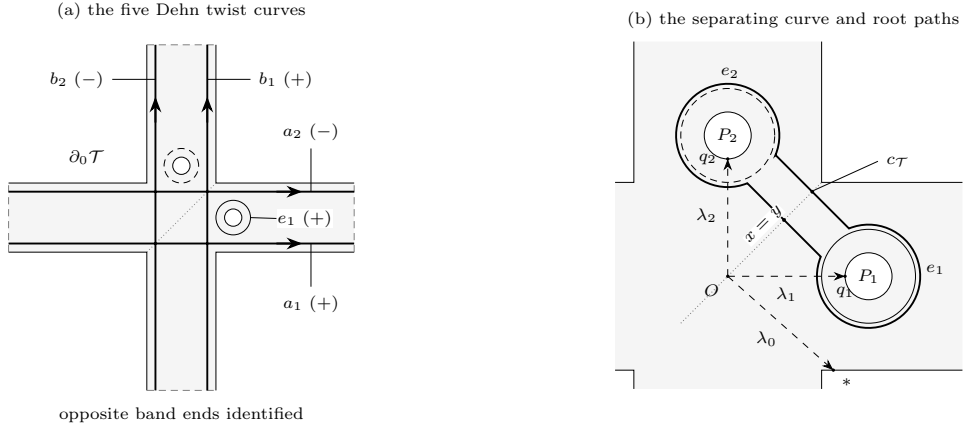
\begin{figure}[tbp]
\centering
\begin{minipage}[c]{.48\linewidth}\centering
\scalebox{0.90}{\begin{tikzpicture}[x=.127cm,y=.127cm,font=\scriptsize,>=Stealth]
\useasboundingbox (-24,-24) rectangle (23,24);
\fill[black!4] (-20,-4)--(-4,-4)--(-4,-20)--(4,-20)--(4,-4)--(20,-4)--(20,4)--(4,4)--(4,20)--(-4,20)--(-4,4)--(-20,4)--cycle;
\draw[thin] (-20,-4)--(-4,-4)--(-4,-20) (4,-20)--(4,-4)--(20,-4)
 (20,4)--(4,4)--(4,20) (-4,20)--(-4,4)--(-20,4);
\draw[densely dashed,black!50] (-20,-4)--(-20,4) (20,-4)--(20,4)
 (-4,-20)--(4,-20) (-4,20)--(4,20);
\draw[black!50,densely dotted] (-4,-4)--(4,4);
\draw[thick] (-20,-3)--(20,-3) (-20,3)--(20,3) (3,-20)--(3,20) (-3,-20)--(-3,20);
\draw[->,thick] (11,-3)--(14,-3);\draw[->,thick] (11,3)--(14,3);
\draw[->,thick] (3,11)--(3,14);\draw[->,thick] (-3,11)--(-3,14);
\draw (6,0) circle[radius=2];\draw[densely dashed] (0,6) circle[radius=2];
\filldraw[fill=white] (6,0) circle[radius=1];\filldraw[fill=white] (0,6) circle[radius=1];
\draw (15,-3)--(15,-8);\node[below] at (15,-8) {$a_1\ (+)$};
\draw (15,3)--(15,8);\node[above] at (15,8) {$a_2\ (-)$};
\draw (3,16)--(8,16);\node[right] at (8,16) {$b_1\ (+)$};
\draw (-3,16)--(-8,16);\node[left] at (-8,16) {$b_2\ (-)$};
\draw (8,0)--(11,-.1);\node[right,fill=white,inner sep=.3pt] at (11,-.1) {$e_1\ (+)$};
\node at (-11,7) {$\partial_0{\mathcal T}$};
\node[below] at (0,-21) {opposite band ends identified};
\node[above] at (0,22) {(a) the five Dehn twist curves};
\end{tikzpicture}}
\end{minipage}\hfill
\begin{minipage}[c]{.49\linewidth}\centering
\scalebox{0.90}{\begin{tikzpicture}[x=.345cm,y=.345cm,font=\scriptsize,>=Stealth]
\useasboundingbox (-5,-6) rectangle (11,11);
\fill[black!4] (-4.8,-4)--(-4,-4)--(-4,-4.8)--(4,-4.8)--(4,-4)--(10,-4)--(10,4)--(4,4)--(4,10)--(-4,10)--(-4,4)--(-4.8,4)--cycle;
\draw[thin] (-4.8,-4)--(-4,-4)--(-4,-4.8) (4,-4.8)--(4,-4)--(10,-4)
 (10,4)--(4,4)--(4,10) (-4,10)--(-4,4)--(-4.8,4);
\begin{scope}[shift={(3,3)},rotate=-45]
\pgfmathsetmacro{\ell}{3*sqrt(2)}\pgfmathsetmacro{\rr}{11/5}
\pgfmathsetmacro{\hh}{17/20}\pgfmathsetmacro{\qq}{sqrt(\rr*\rr-\hh*\hh)}
\pgfmathsetmacro{\ang}{asin(\hh/\rr)}
\draw[thick] ({-\ell+\qq},\hh)--({\ell-\qq},\hh)
 arc[start angle={180-\ang},end angle={-180+\ang},radius=\rr]
 --({-\ell+\qq},-\hh)
 arc[start angle={-\ang},end angle={-360+\ang},radius=\rr]--cycle;
\end{scope}
\draw (6,0) circle[radius=2];\draw[densely dashed] (0,6) circle[radius=2];
\filldraw[fill=white] (6,0) circle[radius=1];\filldraw[fill=white] (0,6) circle[radius=1];
\node at (6,0) {$P_1$};\node at (0,6) {$P_2$};
\draw[black!55,densely dotted] (-2,-2)--(4,4);
\node[rotate=45,fill=white,inner sep=.5pt] at (1.55,2.18) {$x=y$};
\fill ({3-.85/sqrt(2)},{3-.85/sqrt(2)}) circle[radius=.08];
\fill ({3+.85/sqrt(2)},{3+.85/sqrt(2)}) circle[radius=.08];
\draw (3.65,3.65)--(6.4,4.9);\node[right] at (6.4,4.9) {$c_{\mathcal T}$};
\draw[dashed,->] (0,0)--(5,0);\draw[dashed,->] (0,0)--(0,5);
\draw[dashed,->] (0,0)--(4.5,-4);
\foreach \x/\y in {0/0,5/0,0/5,4.5/-4}{\fill (\x,\y) circle[radius=.07];}
\node[below left] at (0,0) {$O$};\node[below] at (2.5,-.1) {$\lambda_1$};
\node[left] at (-.1,2.6) {$\lambda_2$};\node[below left] at (2.45,-2) {$\lambda_0$};
\node[below] at (5,-.1) {$q_1$};\node[left] at (-.1,5) {$q_2$};
\node[below right] at (4.5,-4) {$*$};
\node[right] at (8.1,.3) {$e_1$};\node[above] at (.1,8.1) {$e_2$};
\node[above] at (2.8,10.2) {(b) the separating curve and root paths};
\end{tikzpicture}}
\end{minipage}
\caption{The local surface ${\mathcal T}$. The four corner
arcs of the band model form $\partial_0{\mathcal T}$ after the band ends are identified,
and the two small disks are deleted. The dots on $c_{\mathcal T}$ are its
intersections with the fixed arc of $C_{\mathcal T}$. The paths $\lambda_j$ and the
points $O,q_1,q_2,*$ specify the loops and arcs used in the
calculation described in Appendix~\ref{app:equivalence}.}
\label{fig:local-surface}
\end{figure}

We now compute $[H,V_{\mathcal T}]$, retaining the boundary Dehn twists.

\begin{lemma}\label{thm:local}
In $\Gamma_1^3$,
\begin{equation}\label{eq:local-commutator}
 [H,V_{\mathcal T}]=t_{c_{\mathcal T}}\,t_{e_1}^{-1}t_{e_2}^{-1}.
\end{equation}
\end{lemma}

\begin{proof}
Cap $\partial_1\mathcal T$ and $\partial_2\mathcal T$ by disks,
keeping their centers $x_1,x_2$ fixed. The resulting capping map
has kernel generated by the central Dehn twists $t_{e_1},t_{e_2}$
\cite[Proposition~3.19]{FM}.
Up to the direction of traversal, $H$ pushes $x_1$ once around the
horizontal loop $\alpha=\{y=0\}$ and $V_{\mathcal T}$ pushes $x_2$ once around the
vertical loop $\beta=\{x=0\}$~\cite[Fact~4.7]{FM}. Since $V_{\mathcal T}$ fixes $x_1$,
$V_{\mathcal T}H^{-1}V_{\mathcal T}^{-1}$ is the inverse of the push of $x_1$ around $V_{\mathcal T}(\alpha)$,
so the image of $[H,V_{\mathcal T}]$ is the push of $x_1$ around
$\alpha\cdot V_{\mathcal T}(\alpha)^{-1}$. As $\alpha$ and $\beta$ meet once, this loop
encircles $x_2$ once, and its push is $t_{c_{\mathcal T}}^{\pm1}$. Hence
$[H,V_{\mathcal T}]=t_{c_{\mathcal T}}^{\pm1}t_{e_1}^{k_1}t_{e_2}^{k_2}$ for some integers $k_1,k_2$.

Now cap only $\partial_2{\mathcal T}$ by a disk. Then $b_1$ and $b_2$ become isotopic, so
$V_{\mathcal T}$ becomes trivial, $e_2$ bounds a disk, and $c_{\mathcal T}$ becomes parallel
to $e_1$. The identity becomes $1=t_{e_1}^{\pm1+k_1}$ on a surface where
$t_{e_1}$ has infinite order, so $k_1=\mp1$; symmetrically $k_2=\mp1$. Thus
$[H,V_{\mathcal T}]=(t_{c_{\mathcal T}}t_{e_1}^{-1}t_{e_2}^{-1})^{\pm1}$, and the remaining sign is
fixed by the image of one arc from $\partial_0{\mathcal T}$ to $\partial_1{\mathcal T}$. The
direct check of~\eqref{eq:local-commutator}, described in
Appendix~\ref{app:equivalence}, uses a collection of loops and arcs
that determines the mapping class by the Alexander
method~\cite[Section~2.3]{FM}.
\end{proof}

\enlargethispage{\baselineskip}
\begin{proof}[2nd proof (lantern relation)]
The same identity follows from a lantern relation. Put
$u=H(b_1)$, $v=H(b_2)$ and $E_i=t_{e_i}$. In the local model,
a regular neighborhood of $b_1\cup v$ is a $4$--holed sphere with
boundary curves parallel to $b_2,u,e_1,e_2$; its three interior curves,
in lantern order, are $b_1,v,c_{\mathcal T}$. Thus
\[
 t_{b_1}t_vt_{c_{\mathcal T}}=t_{b_2}t_uE_1E_2.
\]
Conjugating the Dehn twists in $V_{\mathcal T}$ by $H$ gives
\[
 [H,V_{\mathcal T}]=t_ut_v^{-1}t_{b_2}t_{b_1}^{-1}
 =t_{b_2}t_u(t_{b_1}t_v)^{-1}
 =t_{c_{\mathcal T}}E_1^{-1}E_2^{-1}.
\]
This also determines the sign and the two boundary exponents in
Lemma~\ref{thm:local}.
\end{proof}

By \eqref{eq:local-reflection-action} and \eqref{eq:local-commutator},
including $t_{e_1}$ in the conjugating product cancels the two
central boundary Dehn twists. This gives the following corollary.

\begin{corollary}\label{cor:local-reflection}
The product $f_{\mathcal T}=HV_{\mathcal T}t_{e_1}=t_{a_1}t_{a_2}^{-1}t_{b_1}t_{b_2}^{-1}t_{e_1}$ satisfies
$f_{\mathcal T}C_{\mathcal T}f_{\mathcal T}^{-1}C_{\mathcal T}=t_{c_{\mathcal T}}$, that is, $f_{\mathcal T}C_{\mathcal T}f_{\mathcal T}^{-1}=t_{c_{\mathcal T}}C_{\mathcal T}$.
\end{corollary}

\subsection{The surface and the periodic relation}\label{ssec:symmetric}\label{ssec:periodic}\label{ssec:embedded}\label{ssec:comparison}
We now construct $\Sigma$ with the maps $Q,C$ and the curve $c$,
then embed the local relation to prove $PQR\,t_c=1$.

We start with the quotient by the $S_3$-action; its covering data
specify how the surface is assembled. The action has the following
form (Proposition~\ref{prop:surface}):
each of the three reflections in $S_3$ fixes two disjoint circles, and
each of the two rotations fixes three points, which lie on those circles
and therefore have full stabilizer $S_3$.

Dividing by the action turns the fixed circles into boundary, so
$\cO=\Sigma/S_3$ is a surface with boundary, every boundary point being
the image of a point fixed by a reflection. The fixed circles form two
orbits, so $\cO$ has two boundary circles. At each of the three points
with stabilizer $S_3$ the stabilizer is dihedral of order $6$, and the
quotient of a small disk around it is a sector of angle $\pi/3$; these
three points are corners of $\cO$, one on the first boundary circle and
two on the second, cutting the boundary into one arc and two arcs.
Averaging Euler characteristics over the group,
\[
 \chi(\cO)=\tfrac16\bigl(-36+2\cdot3+3\cdot0\bigr)=-5,
\]
which is odd, so with two boundary circles $\cO$ cannot be orientable:
its underlying surface is $N_{5,2}$, with $5$ crosscaps. Two further
circles, disjoint from the boundary, cut $\cO$ into a twice-punctured
Klein bottle $\mathcal K$ and a piece containing the three corners.

The quotient alone does not determine $\Sigma$; the side pairings and labels
in Figure~\ref{fig:quotient} complete the data. The two pieces described next
give these data in geometric form.

Then $\Sigma$ is a regular $S_3$-covering of $\cO$, assembled from
coverings of the two pieces. Over $\mathcal K$, whose group is
\[
 \pi_1(\mathcal K)=\langle\mathfrak a,\mathfrak b,\mathfrak d,
 \mathfrak e\mid\mathfrak a^2\mathfrak b^2\mathfrak d\mathfrak e
 =1\rangle
\]
with $\mathfrak d,\mathfrak e$ peripheral, take the covering with deck
group $S_3=\langle Q,C\rangle$ given by $\mathfrak a\mapsto CQ$,
$\mathfrak b\mapsto QC$, $\mathfrak d\mapsto1$. Both peripheral loops
have trivial image, so each of the two circles has six preimages and the
covering is $\cL\cong\Sigma_{1,12}$. Over the other piece the covering
is the surface $\cN\cong\Sigma_{7,12}$ described next, again with twelve
boundary circles. We use the $S_3$-equivariant identification of
$\partial\cN$ with $\partial\cL$ (see Appendix~\ref{app:equivalence}
for the gluing data), under which the
twelve gluing circles form two free $S_3$-orbits, and
$\Sigma=\cN\cup_\partial\cL$ has genus $7+1+11=19$.

The piece $\cN$ is the oriented ribbon neighborhood of ten curves, whose
$S_3$-quotient is the region of Figure~\ref{fig:quotient}: three families $R_i,S_i,c_i$ ($i\in\Z_3$) and one further curve $\zeta$.
The three $R_i$ pass through a common point $O_R$ and the three $S_i$
through two common points $O_A,O_B$; each $c_i$ crosses $R_i$ twice,
and $c_{i+1}$ and $c_{i+2}$ meet $S_i$ at a common point $t_i$,
so the three $c$-curves meet pairwise once; and $\zeta$ crosses each
$R_i$ twice, which doubles the number of boundary circles from six to
twelve. There are no other intersections, and the cyclic order of the
branches at each of the eighteen intersection points is fixed by
these gluing data; the neighborhood has eighteen vertices and forty-two bands,
so $\chi(\cN)=-24$. The deck action permutes the indices: $Q$ has order
$3$ and fixes $O_R,O_A,O_B$, the three points over the corner
reflectors, while $C$ is the relabelling $i\mapsto-i$ and fixes $R_0$ and
$S_0$ pointwise. The $R_i$ lie over one boundary circle of $\cO$ and the
$S_i$ over the other.

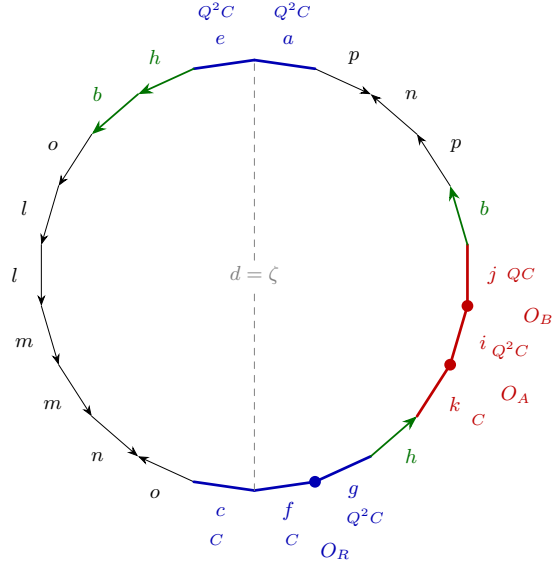
\begin{figure}[tbp]
\centering
\scalebox{0.85}{\begin{tikzpicture}[x=1cm,y=1cm,font=\small,>=Stealth,line cap=round]
\coordinate (p1) at (90.00:3.35);
\coordinate (p2) at (106.36:3.35);
\coordinate (p3) at (122.73:3.35);
\coordinate (p4) at (139.09:3.35);
\coordinate (p5) at (155.45:3.35);
\coordinate (p6) at (171.82:3.35);
\coordinate (p7) at (188.18:3.35);
\coordinate (p8) at (204.55:3.35);
\coordinate (p9) at (220.91:3.35);
\coordinate (p10) at (237.27:3.35);
\coordinate (p11) at (253.64:3.35);
\coordinate (p12) at (270.00:3.35);
\coordinate (p13) at (286.36:3.35);
\coordinate (p14) at (302.73:3.35);
\coordinate (p15) at (319.09:3.35);
\coordinate (p16) at (335.45:3.35);
\coordinate (p17) at (351.82:3.35);
\coordinate (p18) at (368.18:3.35);
\coordinate (p19) at (384.55:3.35);
\coordinate (p20) at (400.91:3.35);
\coordinate (p21) at (417.27:3.35);
\coordinate (p22) at (433.64:3.35);
\coordinate (p23) at (450.00:3.35);
\draw[very thick,blue!70!black] (p1)--(p2);
\node[blue!70!black] at (98.18:3.71) {$ e$};
\node[blue!70!black,font=\scriptsize] at (98.18:4.15) {$ Q^2C$};
\draw[->,thick,green!45!black] (p2)--(p3);
\node[green!45!black] at (114.55:3.73) {$ h$};
\draw[->,thick,green!45!black] (p3)--(p4);
\node[green!45!black] at (130.91:3.73) {$ b$};
\draw[->] (p4)--(p5);
\node[] at (147.27:3.73) {$ o$};
\draw[->] (p5)--(p6);
\node[] at (163.64:3.73) {$ l$};
\draw[->] (p6)--(p7);
\node[] at (180.00:3.73) {$ l$};
\draw[->] (p7)--(p8);
\node[] at (196.36:3.73) {$ m$};
\draw[->] (p8)--(p9);
\node[] at (212.73:3.73) {$ m$};
\draw[->] (p9)--(p10);
\node[] at (229.09:3.73) {$ n$};
\draw[<-] (p10)--(p11);
\node[] at (245.45:3.73) {$ o$};
\draw[very thick,blue!70!black] (p11)--(p12);
\node[blue!70!black] at (261.82:3.71) {$ c$};
\node[blue!70!black,font=\scriptsize] at (261.82:4.15) {$ C$};
\draw[very thick,blue!70!black] (p12)--(p13);
\node[blue!70!black] at (278.18:3.71) {$ f$};
\node[blue!70!black,font=\scriptsize] at (278.18:4.15) {$ C$};
\draw[very thick,blue!70!black] (p13)--(p14);
\node[blue!70!black] at (294.55:3.71) {$ g$};
\node[blue!70!black,font=\scriptsize] at (294.55:4.15) {$ Q^2C$};
\draw[->,thick,green!45!black] (p14)--(p15);
\node[green!45!black] at (310.91:3.73) {$ h$};
\draw[very thick,red!75!black] (p15)--(p16);
\node[red!75!black] at (327.27:3.71) {$ k$};
\node[red!75!black,font=\scriptsize] at (327.27:4.15) {$ C$};
\draw[very thick,red!75!black] (p16)--(p17);
\node[red!75!black] at (343.64:3.71) {$ i$};
\node[red!75!black,font=\scriptsize] at (343.64:4.15) {$ Q^2C$};
\draw[very thick,red!75!black] (p17)--(p18);
\node[red!75!black] at (360.00:3.71) {$ j$};
\node[red!75!black,font=\scriptsize] at (360.00:4.15) {$ QC$};
\draw[->,thick,green!45!black] (p18)--(p19);
\node[green!45!black] at (376.36:3.73) {$ b$};
\draw[->] (p19)--(p20);
\node[] at (392.73:3.73) {$ p$};
\draw[->] (p20)--(p21);
\node[] at (409.09:3.73) {$ n$};
\draw[<-] (p21)--(p22);
\node[] at (425.45:3.73) {$ p$};
\draw[very thick,blue!70!black] (p22)--(p23);
\node[blue!70!black] at (441.82:3.71) {$ a$};
\node[blue!70!black,font=\scriptsize] at (441.82:4.15) {$ Q^2C$};
\draw[dashed,gray] (p12)--(p1);
\node[gray,fill=white,inner sep=1.5pt] at (0,0) {$d=\zeta$};
\fill[blue!70!black] (286.36:3.35) circle (2.6pt);
\node[blue!70!black] at (286.36:4.470000000000001) {$ O_R$};
\fill[red!75!black] (335.45:3.35) circle (2.6pt);
\node[red!75!black] at (335.45:4.470000000000001) {$ O_A$};
\fill[red!75!black] (351.82:3.35) circle (2.6pt);
\node[red!75!black] at (351.82:4.470000000000001) {$ O_B$};
\end{tikzpicture}}
\caption{A fundamental domain for the $S_3$-action: a $22$-gon whose sides
carry all the gluing data. The eight thick sides occur once and are the
mirror boundary of $\cO$, five blue over the $R$-family and three red over
the $S$-family; each is fixed by the reflection printed outside it, and the
three filled vertices, where the two reflections differ, are the corners.
The dashed chord is the side $d=\zeta$ along which the two faces of the
quotient meet. The remaining sides are identified in pairs matching the
arrows, by $h\mapsto Q^2C$, $b\mapsto QC$, $l\mapsto Q^2C$,
$m\mapsto QC$, $n\mapsto QC$ and $o,p\mapsto1$; the green sides $b$ and $h$
are the image of the vanishing cycle $c$ and its translates. Six
copies, one for each element of $S_3$, glued by these elements along the
paired sides and by the reflections along the mirror sides, return
$\Sigma$ with its $S_3$-action.}
\label{fig:quotient}
\end{figure}

\begin{proposition}\label{prop:surface}
$Q^3=C^2=1$ and $CQC=Q^{-1}$. The map $Q$ has three fixed points, rotation
$4\pi/3$ at each, and quotient $\Sigma_6(3,3,3)$. The involution $C$ has
fixed set $R_0\cup S_0$ and quotient $N_{18,2}$. The curve $c=c_0$ is
nonseparating, $C(c)=c$ with its orientation reversed, and the
central curves $c_i=Q^ic$ meet pairwise once.
\end{proposition}

Here $\Sigma_6(3,3,3)$ is the genus--$6$ orbifold with three cone points of
order $3$, and $N_{18,2}$ is the nonorientable surface with $18$ crosscaps and two
boundary circles. The reconstruction of $\Sigma$, $Q$ and $C$ and the
verification of Proposition~\ref{prop:surface} are described in
Appendix~\ref{app:equivalence}.

With $Q$ and $C$ in place, we use a cyclic covering to construct
$\rho$ on the same surface. Let $W_0\cong\Sigma_{1,3}$ be
the connected triple cover of a pair of pants on which all three boundary
loops have character $1\in\Z_3$, and let $W_0'\cong\Sigma_{16,3}$ be a
cyclic triple cover of $\Sigma_{4,3}$ branched at three points. Their
boundary rotations are opposite, so they glue to a periodic map on a
closed surface. The explicit homeomorphism
$\Psi\colon\Sigma_{\mathrm{per}}\to\Sigma$ places this cyclic covering
model on $\Sigma$, with deck map $\rho$; its defining data are
described in Appendix~\ref{app:equivalence}.

\begin{proposition}\label{prop:periodic}
The map $\rho$ has order $3$, three fixed points in $W'=\Psi(W_0')$,
rotation $4\pi/3$ at each, and quotient of genus $6$. Hence $P=C\rho C$
and $R=\rho^{-1}$ have rotation $2\pi/3$. On $W=\Psi(W_0)$ the boundary-fixed
part of $\rho$ is the star product $t_{\ell_1}t_{\ell_2}t_c t_\eta$ of a
star with leaves $\ell_1,\ell_2,c$ and transverse curve $\eta$
(Figure~\ref{fig:star}).
\end{proposition}

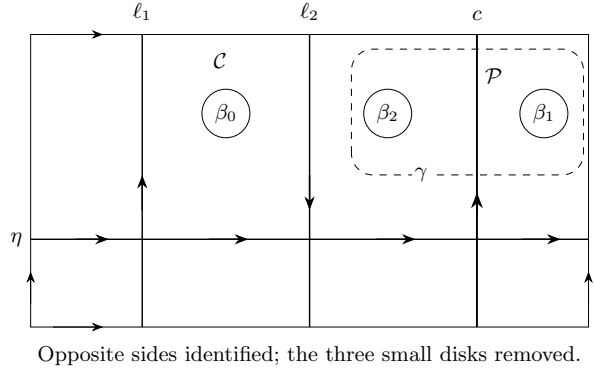
\begin{figure}[tbp]
\centering
\scalebox{0.90}{\begin{tikzpicture}[x=8.2cm,y=4.3cm,>=Stealth,font=\small]
 \draw (0,0) rectangle (1,1);
 \foreach \y in {0,1} {\draw[->] (.04,\y)--(.13,\y);}
 \foreach \x in {0,1} {\draw[->] (\x,.08)--(\x,.19);}
 \draw[semithick] (.20,0)--(.20,1);
 \draw[semithick,->] (.20,.40)--(.20,.52);
 \draw[semithick] (.50,0)--(.50,1);
 \draw[semithick,->] (.50,.52)--(.50,.40);
 \draw[thick] (.80,0)--(.80,1);
 \draw[thick,->] (.80,.39)--(.80,.46);
 \draw[semithick] (0,.30)--(1,.30);
 \draw[semithick,->] (.08,.30)--(.14,.30);
 \draw[semithick,->] (.32,.30)--(.39,.30);
 \draw[semithick,->] (.62,.30)--(.69,.30);
 \draw[semithick,->] (.88,.30)--(.95,.30);
 \foreach \x/\lab in {.35/0,.64/2,.92/1} {
   \filldraw[fill=white] (\x,.73) ellipse[x radius=.043,y radius=.083];
   \node at (\x,.73) {$\beta_{\lab}$};}
 \draw[dashed,rounded corners=9pt] (.575,.52) rectangle (.991,.95);
 \node[fill=white,inner sep=1pt] at (.70,.52) {$\gamma$};
 \node[above] at (.20,1.02) {$\ell_1$};
 \node[above] at (.50,1.02) {$\ell_2$};
 \node[above] at (.80,1.02) {$c$};
 \node[left] at (0,.30) {$\eta$};
 \node at (.34,.91) {$\mathcal C$};
 \node at (.83,.86) {$\mathcal P$};
 \node[below] at (.50,-.035) {Opposite sides identified; the three small disks removed.};
\end{tikzpicture}}
\caption{The star in $W$, with leaves $\ell_1,\ell_2,c$ and transverse
curve $\eta$. The dashed curve $\gamma$ bounds the pair of pants
$\mathcal P$ with $\beta_1,\beta_2$; on its other side lies the $2$--holed
torus $\mathcal C$, with remaining boundary $\beta_0$. The configuration
is placed on $\Sigma$ by $\Psi$.}
\label{fig:star}
\end{figure}

To relate $\rho$ to the $S_3$-action, we embed the local reflection
surface $\mathcal T$ in $\Sigma$.
Removing one return band from the fixed circle $R_0$ leaves an interval
$I$, and ${\mathcal T}_\Sigma=N(R_1\cup R_2\cup I\cup c_0)$ is a $C$-invariant
$3$--holed torus (different from $W$ in Figure~\ref{fig:star}).

\begin{lemma}\label{lem:embedding}
There is an orientation-preserving homeomorphism
$\phi:\mathcal T\to\mathcal T_\Sigma$ preserving the boundary framings
such that
\[
 \phi C_{\mathcal T}\phi^{-1}=C|_{\mathcal T_\Sigma},\qquad
 \phi(c_{\mathcal T})=c.
\]
\end{lemma}

We use the same names for $a_1,a_2,b_1,b_2,e_1$ and their images
under $\phi$. The construction data for $\phi$ and these image
curves are described in Appendix~\ref{app:equivalence}.

Let $f_0$ be the image of $f_{\mathcal T}$, that is, $f_0=t_{a_1}t_{a_2}^{-1}t_{b_1}t_{b_2}^{-1}t_{e_1}$
for these curves, extended by the identity. The subsurface is $C$-invariant and both sides agree with $C$
outside it, so Corollary~\ref{cor:local-reflection} gives
\begin{equation}\label{eq:embedded}
 f_0Cf_0^{-1}=t_c C .
\end{equation}

The local relation is now in place. We do not need $f_0=\rho Q$:
only $h=f_0^{-1}\rho Q$ must commute with $C$, so that the two maps
conjugate $C$ to the same reflection. We check this through the
induced actions on the fundamental group
$\pi_1(\Sigma)=\langle g_1,\dots,g_{38}\mid\mathfrak r\rangle$.
Appendix~\ref{app:equivalence} describes the presentation data,
based actions, and substitution checks used in this comparison. 
The following criterion gives the mapping-class equality.

\begin{lemma}\label{lem:recognition}
If homeomorphisms $F_1,F_2$ of $\Sigma$ satisfy
$(F_1)_*(g_j)=q\,(F_2)_*(g_j)\,q^{-1}$ for $1\le j\le38$ with a single word
$q$, then $F_1=F_2$ in $\Mod^{\pm}(\Sigma)$.
\end{lemma}

\begin{proof}
The two outer automorphisms agree, and the Dehn--Nielsen--Baer
correspondence is injective~\cite[Theorem~8.1]{FM}.
\end{proof}

\begin{lemma}\label{lem:h-commutes}
The mapping class $h=f_0^{-1}\rho Q$ commutes with $C$.
\end{lemma}

\begin{proof}
The based actions of $\rho,Q,C,f_0^{-1}$ give those of $hC$ and $Ch$.
For the $66$--letter word $q_h$ (see Appendix~\ref{app:equivalence}),
substitution gives
\[
 (hC)_*(g_j)=q_h(Ch)_*(g_j)q_h^{-1},\qquad 1\leq j\leq38.
\]
The substitution checks use free cancellation and replacement of
more than half of a cyclic conjugate of the
surface relator by the inverse of its complement.
Lemma~\ref{lem:recognition} gives the result.
\end{proof}

\begin{proof}[Proof of Theorem~\ref{thm:identity}]
By~\eqref{eq:embedded} and Lemma~\ref{lem:h-commutes},
$(\rho Q)C(\rho Q)^{-1}=f_0hCh^{-1}f_0^{-1}=f_0Cf_0^{-1}=t_c C$. Since
$QCQ^{-1}=CQ$, the left side is $\rho CQ\rho^{-1}$, so $t_c C\rho=\rho CQ$.
Put $P=C\rho C$, of order $3$; then $R=CP^{-1}C=\rho^{-1}$. Since $C$
reverses orientation and fixes $c$, $Ct_c C=t_c^{-1}$, and solving
$t_c C\rho=\rho CQ$ for $Q$ gives
\[
 Q=C\rho^{-1}t_c C\rho=(C\rho^{-1}C)(Ct_c C)\rho=P^{-1}t_c^{-1}R^{-1},
\]
which is~\eqref{eq:main}. The rotation data are
Propositions~\ref{prop:surface} and~\ref{prop:periodic}.
\end{proof}

\begin{remark}\label{rem:Monden}
Monden \cite{Monden} expresses a nonseparating Dehn twist as a
product of six elements of order $3$. This expression can be
shortened to four factors, as also observed by Monden (private
communication), but that does not further compress to three factors, which is essential for our purposes.
\end{remark}

The maps satisfy a second identity, of a different kind, which may be of independent interest.

We compute this power of $PQ$ using the $3$--holed torus
$W\subset\Sigma$ filled by $c,\rho(c),\rho^2(c)$, where
$\rho=R^{-1}$, and its complement.

\begin{proposition}[A power identity]\label{thm:power}
The boundary curves of $W$ can be labelled $\beta_0,\beta_1,\beta_2$, and
there is a curve $\gamma\subset W$ cutting $W$ into a $2$--holed torus
bounded by $\beta_0,\gamma$ and a pair of pants bounded by
$\beta_1,\beta_2,\gamma$, such that $\mathcal Y=PQ$ satisfies
\begin{equation}\label{eq:power}
 \mathcal Y^{12}=t_{\beta_0}^{-1}\,t_{\beta_1}^{-4}\,t_{\beta_2}^{-4}\,t_\gamma^{3}.
\end{equation}
\end{proposition}

The four curves are disjoint, so $\mathcal Y$ is a twelfth root of a
multitwist with both positive and negative exponents; in particular
it has infinite order. The power determines $\gamma$, $\beta_0$, and the
unordered pair $\{\beta_1,\beta_2\}$ as isotopy classes. We choose
the last two labels as in Figure~\ref{fig:star}.

\begin{proof}[Proof of Proposition~\ref{thm:power}]
Since $\mathcal Y=PQ=t_c^{-1}\rho$, write $\rho=\tau F$ with $\tau$ supported on $W$ and $F$ on $W'$, both
boundary-fixed, $\tau^3=t_{\beta_0}t_{\beta_1}t_{\beta_2}$ and
$F^3=(t_{\beta_0}t_{\beta_1}t_{\beta_2})^{-1}$. By
Proposition~\ref{prop:periodic}, $\tau=t_{\ell_1}t_{\ell_2}t_c t_\eta=t_c K$,
where $K=t_{\ell_1}t_{\ell_2}t_\eta$, because $t_c$ commutes with the
Dehn twists about the other two leaves. Now $K=t_{\ell_2}(t_{\ell_1}t_\eta t_{\ell_2})t_{\ell_2}^{-1}$ is conjugate to the product $t_{\ell_1}t_\eta t_{\ell_2}$ along the $3$--chain $\ell_1,\eta,\ell_2$, whose
neighborhood is the $2$--holed torus $\mathcal C$ bounded by $\beta_0,\gamma$ (Figure~\ref{fig:star}), and $t_{\ell_2}$ commutes with
$t_{\beta_0}t_\gamma$; so the chain relation gives $K^4=t_{\beta_0}t_\gamma$. As $K$ and $F$ commute,

\begin{align*}
 \mathcal Y^{12}&=(KF)^{12}=(K^4)^3(F^3)^4\\*
 &=(t_{\beta_0}t_\gamma)^3
   (t_{\beta_0}t_{\beta_1}t_{\beta_2})^{-4}\\*
 &=t_{\beta_0}^{-1}t_{\beta_1}^{-4}t_{\beta_2}^{-4}t_\gamma^3.\qedhere
\end{align*}

\end{proof}

\subsection{The fibration over the torus}\label{sec:fibration}\label{ssec:orbifold}\label{ssec:triple}

We now describe the genus-19 fibration with monodromy
\eqref{eq:intro-genus19} as a cyclic $3$--fold pullback of an
orbifold fibration over the sphere, and identify its handle and
peripheral loops. We begin with the vanishing cycles $c,P(c),P^2(c)$.

\begin{proposition}\label{prop:configuration}
The curves $\xi_i=\rho^i(c)$, $i\in\Z_3$, lie on $W$, have
simultaneous representatives meeting pairwise once transversally with no
triple point, and have regular neighborhood isotopic to $W$ minus
boundary collars. The complement $W'$ is connected of genus $16$. The same
holds for $\eps_i=P^i(c)=C(\xi_i)$ on $C(W)$. For the chosen
orientations, the algebraic intersection matrix of
$(\eps_0,\eps_1,\eps_2)$ is
$\left(\begin{smallmatrix}0&1&-1\\-1&0&1\\1&-1&0\end{smallmatrix}\right)$.
\end{proposition}

The $\eps_i$ are the vanishing cycles in the handle system of
Theorem~\ref{thm:fibration}. The proof uses an explicit ribbon model of
the deck orbit on the small cyclic cover; the curve orientations and
verification data are described in Appendix~\ref{app:equivalence}.

We now construct the orbifold fibration whose pullback will have
these vanishing cycles.
Choose points $s_p,s_q,s_r,s_d$ on $S^2$, let $S^\circ$ be the complement
of small disks about them, and choose peripheral loops $p,q,r,d$ as in
Figure~\ref{fig:sphere-paths}, so that
\[
 \pi_1(S^\circ,*)=\langle p,q,r,d\mid pqrd=1\rangle .
\]

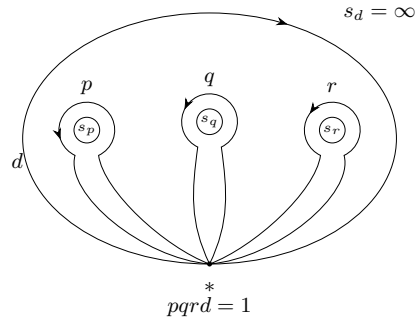
\begin{figure}[tbp]
\centering
\scalebox{0.90}{\begin{tikzpicture}[scale=.82,>=Stealth,
  oriented/.style={postaction={decorate},decoration={markings,
    mark=at position .58 with {\arrow{Stealth}}}},
  every node/.style={font=\small}]

  \draw[oriented] (0,-1.2) arc[start angle=270,end angle=-90,
    x radius=3.35,y radius=2.25];
  \node at (-3.45,.65) {$d$};
  \node at (3.05,3.28) {$s_d=\infty$};

  \draw[oriented] (0,-1.2) .. controls (-.8,-.9) and (-1.9,.1) ..
    (-1.989,.747) arc[start angle=-65,end angle=245,radius=.5]
    .. controls (-2.6,.05) and (-1.2,-1.1) .. (0,-1.2);
  \draw[oriented] (0,-1.2) .. controls (.36,-.45) and (.32,.10) ..
    (.211,.897) arc[start angle=-65,end angle=245,radius=.5]
    .. controls (-.32,.10) and (-.36,-.45) .. (0,-1.2);
  \draw[oriented] (0,-1.2) .. controls (1.2,-1.1) and (2.6,.05) ..
    (2.411,.747) arc[start angle=-65,end angle=245,radius=.5]
    .. controls (1.9,.10) and (.8,-.9) .. (0,-1.2);
  \foreach \x/\y/\lab in {-2.2/1.2/p,0/1.35/q,2.2/1.2/r}{
    \draw (\x,\y) circle (.23);
    \node at (\x,\y) {$\scriptstyle s_{\lab}$};
    \node at (\x,\y+.76) {$\lab$};
  }
  \fill (0,-1.2) circle (.045);
  \node[below=4pt] at (0,-1.2) {$*$};
  \node at (0,-1.97) {$pqrd=1$};
\end{tikzpicture}}
\caption{The $4$--holed sphere with positive peripheral loops and
$s_d$ at infinity. The clockwise outer loop represents
$d=(pqr)^{-1}=r^{-1}q^{-1}p^{-1}$.}
\label{fig:sphere-paths}
\end{figure}

By Theorem~\ref{thm:identity}, $\mu\colon(p,q,r,d)\mapsto(P,Q,R,t_c)$ defines
a monodromy homomorphism $\pi_1(S^\circ)\to\Mod(\Sigma)$. Glue the corresponding $\Sigma$-bundle over $S^\circ$ to a
Lefschetz neighborhood of a nodal fiber over the $d$-disk and to the
models $(\Sigma\times D^2)/(\Z_3)$ over the other three disks, where
$\Z_3$ acts by the periodic representative on $\Sigma$ and by rotation
on $D^2$.

\begin{proposition}\label{prop:orbifold}
The result is a genus-$19$ fibration $\pi_S\colon Z\to S^2$ of a compact
oriented $4$-orbifold with three multiple fibers of multiplicity $3$,
whose reduced fibers have genus $6$, and one irreducible nodal fiber. The
quotient singularities are of type $\frac13(1,2)$ on the $P$- and
$R$-fibers and $\frac13(1,1)$ on the $Q$-fiber, three on each.
\end{proposition}

\begin{proof}
Over the $d$-disk this is the standard Lefschetz neighborhood. Over each of
the other three disks the generator of $\Z_3$ acts on $\Sigma\times D^2$
by $(X^{-1},e^{2\pi i/3})$: monodromy is parallel transport, so the
boundary monodromy of the quotient is $X$, and the quotient has one
multiple fiber of multiplicity $3$ whose reduced fiber $\Sigma/X$ has
genus $6$ by Propositions~\ref{prop:surface} and~\ref{prop:periodic}.
At a fixed point of $X$, where $X$ rotates by $\theta$, the generator acts
on the tangent space by $(e^{-i\theta},e^{2\pi i/3})$, so the rotation data
of Propositions~\ref{prop:surface} and~\ref{prop:periodic} give the stated
types.
\end{proof}

A cyclic $3$--fold change of base removes the multiple fibers.
Since $P^3=Q^3=R^3=1$, the monodromy factors through
$\Gamma=\langle p,q,r,d\mid p^3=q^3=r^3=1,\ pqrd=1\rangle$. Let
$\chi\colon\Gamma\to\Z_3$ be $\chi(p)=\chi(q)=\chi(r)=1$, $\chi(d)=0$.
Any character that is nonzero on $p,q,r$ and zero on $d$ has
$\chi(p)+\chi(q)+\chi(r)=0$, hence equals $\chi$ or $-\chi$; so the
branched cover $E\to S^2$ defined by $\ker\chi$ is canonical. It is a torus,
since $e(E)=3\cdot2-3\cdot2=0$. Pulling back and normalizing over the three
branch points gives a smooth genus-$19$ Lefschetz fibration $\pi\colon X\to
E$ with three nodal fibers over the three lifts of $s_d$ and no multiple
fibers. Its monodromy is the restriction of $\mu$ to
$\pi_1(E^\circ)=\ker\chi$, which is free on the Schreier generators
\[
 a_i=p^iqp^{-(i+1)},\qquad b_i=p^irp^{-(i+1)}\qquad(i=0,1),
\]
with $a_2=(a_0a_1)^{-1}$ and $b_2=(b_0b_1)^{-1}$.

\begin{lemma}\label{lem:schreier}
The loops $\mathsf a=qr^{-1}$, $\mathsf b=r^{-1}q$ and $e_i=p^idp^{-i}$
$(i=0,1,2)$ form a standard system for the $3$--punctured torus
$E^\circ$: they satisfy $[\mathsf a,\mathsf b]e_0e_1e_2=1$, each $e_i$ is a
positively oriented peripheral loop about a different lift of $s_d$, and
$\mathsf a,\mathsf b,e_0,e_1$ is a basis of $\pi_1(E^\circ)$.
\end{lemma}

\begin{proof}
Reidemeister--Schreier rewriting gives $\mathsf a=a_0b_0^{-1}$,
$\mathsf b=b_0b_1a_1^{-1}a_0^{-1}$ and $e_0=b_0b_1a_1^{-1}$,
$e_1=b_0^{-1}a_0a_1$, $e_2=b_1^{-1}a_0^{-1}$; the product
$[\mathsf a,\mathsf b]e_0e_1e_2$ reduces freely to the empty word, and
these words generate the free group on $a_0,a_1,b_0,b_1$, so
$\mathsf a,\mathsf b,e_0,e_1$ is a basis (\S\ref{ssec:finite}). The
$e_i$ are conjugates of $d$ by $1,p,p^2$, which lie in distinct cosets of
$\ker\chi$, so they are peripheral about distinct lifts of $s_d$. A basis
with this peripheral structure and relation is carried to a geometric
system by an orientation-preserving homeomorphism of $E^\circ$, by the
Dehn--Nielsen--Baer theorem for punctured surfaces~\cite[Section~8.2]{FM}.
\end{proof}

Thus $\mu(\mathsf a)=A$, $\mu(\mathsf b)=B$, and
$\mu(e_i)=t_{P^i(c)}$. This identifies \eqref{eq:intro-genus19}
as the monodromy factorization of the pullback fibration.

\subsection{The chain form}\label{sec:chain}

We first deduce Theorem~\ref{cor:chain} from the periodic construction.
We then give a cyclic-cover description of the fiber and identify the
based coordinates used in the topology calculations.

\begin{lemma}\label{lem:chain-rotation}
Let $\fv$ be a $38$-chain on $\Sigma$ and $V=t_{v_1}\cdots t_{v_{38}}$.
Then $V^{78}=1$, and $V^{26}$ is periodic of order $3$ with three fixed
points, quotient of genus $6$, and local rotation $4\pi/3$ at each fixed
point.
\end{lemma}

\begin{proof}
A regular neighborhood of the chain is $\Sigma_{19,1}$ with a disk as
complement. The chain relation gives $V^{78}=t_\partial$ on this
neighborhood, hence $V^{78}=1$ on $\Sigma$
\cite[Proposition~4.12]{FM}. In the standard branched-cover model of the
chain, the capped map is represented by a one-step rotation of a regular
$78$-gon with opposite sides identified; the center is the capped point
\cite[Sections~7.1 and~9.4]{FM}.
Its $26$th power fixes the center and the two vertex classes, and no
other points. Thus it has order $3$ and three fixed points. Writing $g_0$
for the quotient genus, Riemann--Hurwitz gives
$-36=3(2-2g_0)-6$, so $g_0=6$.

To determine the sign, use collar coordinates $(s,\theta)$ with $s$
increasing away from the capping disk and $\theta$ counterclockwise.
The right-handed boundary twist is
$(s,\theta)\mapsto(s,\theta-2\pi s)$.
A boundary-fixing lift of a one-step clockwise rotation therefore has
$78$th power $t_\partial$; the counterclockwise rotation gives
$t_\partial^{-1}$. Changing the lift by a boundary twist changes this
exponent by a multiple of $78$. Hence $V^{78}=t_\partial$ selects the
clockwise rotation, and $V^{26}$ has angle
$-26(2\pi/78)=-2\pi/3\equiv4\pi/3$ at the capped point.
If the rotations at the three fixed points are $2\pi r_j/3$, with
$r_j\in\{1,2\}$, the cyclic covering relation gives
$\sum_j r_j^{-1}=0$ in $\Z_3$
\cite[Definition~2.1]{PRS}. Since $r_j^{-1}=r_j$ in $\Z_3$, the three
$r_j$ must all agree. Thus all three rotations are $4\pi/3$.
\end{proof}

\begin{proof}[Proof of Theorem~\ref{cor:chain}]
Choose any $38$--chain $\fv_0$. By Lemma~\ref{lem:chain-rotation} and
Theorem~\ref{thm:identity}, $V_0^{-26}$ and $R$ have order $3$, quotient
genus $6$, and local rotation $2\pi/3$ at each of their three fixed points.
Nielsen classification gives an orientation-preserving homeomorphism
$\psi$ with $\psi V_0^{-26}\psi^{-1}=R$
\cite{Ni}, \cite[Proposition~2.2]{PRS}.
Set $v_i=\psi(v_{0,i})$ and $u_i=C(v_{39-i})$. Then
\[
 V^{-26}=R,\qquad U=CV^{-1}C,\qquad
 U^{-26}=CV^{26}C=CR^{-1}C=P.
\]
The periodic relation now gives
$t_c=R^{-1}Q^{-1}P^{-1}=V^{26}Q^{-1}U^{26}$.
\end{proof}

This proof chooses the chains by conjugacy; it does not require a
separate list of chain curves in the coordinates below.

\subsubsection{Cyclic-cover coordinates}\label{ssec:chain-model}\label{g19:sec:derivation}
We construct $\hat\Sigma$ as a cyclic $3$--fold cover of a genus--$6$
surface branched at three points. Removing disks about these points
gives $\Sigma_{6,3}$. We specify an unramified cover of this surface,
with a single circle over each boundary component, and then cap
those circles. The ribbon graph below describes $\Sigma_{6,3}$:
thicken its vertex to a disk and its edges to bands, attached in
the indicated cyclic order.
Start with a ribbon graph with one vertex and $14$ oriented edges
$e_1,\dots,e_{14}$, whose vertex rotation, in signed-edge notation, is
\begin{equation*}
(1,-10,-14,4,13,-7,12,-1,7,6,-11,-12,5,-5,11,10,2,-9,-4,8,9,3,-8,-3,-2,-13,-6,14).
\end{equation*}
Its regular neighborhood is $\Sigma_{6,3}$, with boundary words
$r_0=(-5)$, $r_1=(-2,-9,3)$ and
\begin{equation*}
r_2=(-10,2,-13,-7,6,14,4,8,-3,-8,9,-4,13,-6,-11,10,-14,1,7,12,5,11,-12,-1).
\end{equation*}
The sheet increments
$(\mu_1,\dots,\mu_{14})=(1,0,1,1,1,1,1,1,2,2,0,1,1,1)\in(\Z_3)^{14}$
define a cyclic triple cover: for $s\in\Z_3$, lift $e_j$ to an edge from
sheet $s$ to sheet $s+\mu_j$. Each $r_i$ has nonzero increment, so its
preimage is a single circle covering that boundary component with
degree $3$. Cap these circles by disks, extending the covering over
each disk by $z\mapsto z^3$. The resulting closed surface has
\[
 \chi(\hat\Sigma)=3\chi(\Sigma_{6,3})+3
 =3(2-2\cdot6-3)+3=-36,
\]
and hence genus $19$. Adding one to the sheet index gives the deck
map $\hat Q$, of order $3$, whose fixed points are the centers of the
three capping disks. The edges $e_{13},e_{14}$ have
increment one, and the three lifts of $e_{14}e_{13}^{-1}$ are curves
$\hat c,\hat Q(\hat c),\hat Q^2(\hat c)$ meeting pairwise once.
Their $3$--holed torus neighborhood is shown in
Figure~\ref{g19:fig:torus}, after the identification with $\Sigma$
in Proposition~\ref{prop:models}.

Order the lifted edges by sheet and then by the index of the base edge.
Take the two lifts of $e_{14}$ from sheets $0$ and $1$ as a spanning tree.
After collapsing them, number the $40$ free edges in their inherited order.
The first two capping relations eliminate generators $5$ and $2$;
renumbering the remaining generators in order gives
a presentation $\pi_1(\hat\Sigma)=\langle y_1,\dots,y_{38}\mid {\mathfrak r}_Y\rangle$
in which $\hat c=y_{38}y_{37}^{-1}$; the surface is oriented oppositely to the
polygon word ${\mathfrak r}_Y$, so that $t_{\hat c}$ is right-handed.
Appendix~\ref{app:equivalence} gives this presentation and the based
automorphisms representing $\hat Q,\hat C,\hat R$, and $t_{\hat c}$.
The map $\hat Q$ is the deck map above; the other maps are identified
with the geometric ones in Proposition~\ref{prop:models}. Put
$\hat P=\hat C\hat R^{-1}\hat C$.

\subsubsection{Identification of the coordinates}\label{ssec:models}
The coordinates above describe a different configuration on the same
fiber. We first give the Hurwitz moves producing this configuration,
then identify its based actions with those used in the group calculations.

\begin{proposition}[The second solution]\label{prop:second-solution}
If $(P,Q,R,c,C)$ with $P^3=Q^3=R^3=1$ satisfies~\eqref{eq:main}
and~\eqref{eq:real}, then so does
\[
 (P',Q',R',c',C')=(QRQ^{-1},\ Q,\ P,\ P^{-1}c,\ QC),
\]
with the same rotation data. The new quadruple is obtained from the old
one by the Hurwitz moves $\sigma_1\sigma_2\sigma_1$, which exchange the first
and third entries, followed by conjugation by $P^{-1}$.
\end{proposition}

\begin{proof}
With $\sigma_i$ replacing the pair $(x_i,x_{i+1})$ by
$(x_ix_{i+1}x_i^{-1},x_i)$ one finds
\[
 \sigma_1\sigma_2\sigma_1(P,Q,R,t_c)=(PQRQ^{-1}P^{-1},\ PQP^{-1},\ P,\ t_c),
\]
and conjugating by $P^{-1}$ gives $(P',Q',R',t_{c'})$. Hurwitz moves and
conjugation preserve the product, so $P'Q'R't_{c'}=1$. Next, $C'$
reverses orientation, and using $CQC=Q^{-1}$, $CRC=P^{-1}$ and $Cc=c$,
\begin{align*}
 (QC)^2&=Q(CQC)=1,\\
 C'Q'C'&=Q(CQ^2C)=Q^{-1},\\
 C'P'C'&=Q(CQC)(CRC)=P^{-1}=R'^{-1},\\
 C'(c')&=Q(CP^{-1}C)(Cc)=QRc=P^{-1}t_c^{-1}c=c'.
\end{align*}
The rotation data are preserved since $P'\sim R$ and $R'=P$.
\end{proof}

The braid $\sigma_1\sigma_2\sigma_1$ is the half-twist of a disk containing
the three cone points of the base sphere; it exchanges the two cone points
of rotation $2\pi/3$ and fixes the third. Conjugation by $P^{-1}$ only
changes the identification of the fiber, so the two solutions give
isomorphic orbifold fibrations and the same fibration $X$
(Section~\ref{sec:fibration}). Applying the proposition twice returns
$(P,Q,R,c,C)$ conjugated by $Q$.

\begin{remark}[The two configurations]\label{rem:configurations}
The configurations $(Q,C,c)$ and $(Q,QC,P^{-1}c)$ are not
equivalent: no homeomorphism of $\Sigma$ normalizing the $S_3$-action
carries $c$ to $P^{-1}c$. Such a homeomorphism would conjugate $C$,
the only reflection preserving $c$, to $QC$, the only one preserving
$P^{-1}c$, and it would carry the mirror of $\cO$ with one corner to
itself. But $C$ fixes two points of $c$, which lie on that mirror since
$c$ crosses $R_0$, whereas the two points of $P^{-1}c$ fixed by
$QC$ lie on the mirror with two corners. In the cyclic-cover coordinates, the two points
of $\hat c$ fixed by $\hat C$ also lie on the mirror with two corners, as
they must if $(\hat Q,\hat C,\hat c)$ corresponds to the second
configuration.
\end{remark}

\begin{proposition}\label{prop:models}
There is an orientation-preserving homeomorphism
$\Theta\colon\hat\Sigma\to\Sigma$ carrying the periodic maps and curve to the second solution of
Proposition~\ref{prop:second-solution}:
\[
 (\hat P,\hat Q,\hat R,\hat c,\hat C)
 \longmapsto (QRQ^{-1},Q,P,P^{-1}c,QC).
\]
\end{proposition}

\begin{proof}
Write the two surface groups as
\[
 \Pi_A=\langle g_1,\ldots,g_{38}\mid\mathfrak r_A\rangle,
 \qquad \Pi_B=\langle y_1,\ldots,y_{38}\mid\mathfrak r_Y\rangle.
\]
The generator dictionaries in Appendix~\ref{app:equivalence} define
homomorphisms $\theta:\Pi_B\to\Pi_A$ and $\psi:\Pi_A\to\Pi_B$.
Substitution verifies both surface relators and gives
$\theta\psi(g_i)=g_i$ and $\psi\theta(y_i)=y_i$ for all $i$.
The induced homology map preserves the integral intersection form.
By the Dehn--Nielsen--Baer theorem \cite[Section~8]{FM}, $\theta$ is
induced, up to an inner automorphism, by an orientation-preserving
homeomorphism $\Theta:\hat\Sigma\to\Sigma$.

The same substitutions give $\Theta(\hat c)=P^{-1}(c)$ as unoriented
curves. With compatible
basepoint paths, substitution gives
\[
 \theta\widetilde{\hat Q}=\Inn(k_Q)\widetilde Q\theta,
 \qquad
 \theta\widetilde{\hat R}=\Inn(k_R)\widetilde P\theta,
 \qquad
 \theta\widetilde{\hat C}=(\widetilde Q\widetilde C)\theta,
\]
where
\[
 k_Q=g_9g_{28}g_8^{-1}g_{19}g_{32}^{-1}g_{14}g_{26},
 \qquad k_R=g_5^{-1}g_{27}^{-1}.
\]
These equalities hold on all 38 generators, with the same correction
word in each column. Each comparison is a finite sequence of free
cancellations and substitutions using the surface relator; the
dictionaries and checks are described in Appendix~\ref{app:equivalence}.
They identify the outer actions, hence
the mapping classes, by the Dehn--Nielsen--Baer theorem \cite[Section~8]{FM}.
Since $\hat P=\hat C\hat R^{-1}\hat C$, the reflection identities give
\[
 \Theta\hat P\Theta^{-1}=(QC)P^{-1}(QC)
 =Q(CP^{-1}C)Q^{-1}=QRQ^{-1}.
\]

This proves the stated comparison.
\end{proof}

\begin{corollary}\label{thm:chain-model}\label{g19:lem:placement}
In $\Mod^{\pm}(\hat\Sigma)$,
\[
 \hat P\hat Q\hat R\,t_{\hat c}=1,\qquad
 \hat C\hat P\hat C=\hat R^{-1},\qquad
 \hat C\hat Q\hat C=\hat Q^{-1},\qquad \hat C(\hat c)=\hat c,
\]
and $\hat P^3=\hat Q^3=\hat R^3=\hat C^2=1$.
\end{corollary}

\begin{proof}
Transport the second solution of Proposition~\ref{prop:second-solution}
by the homeomorphism in Proposition~\ref{prop:models}.
\end{proof}

To compare the fibrations, use the three Hurwitz moves in
Proposition~\ref{prop:second-solution}, followed by conjugation by
$P^{-1}$. They take $(P,Q,R,t_c)$ to
$(QRQ^{-1},Q,P,P^{-1}t_c P)$, which is the image of the tuple in cyclic-cover coordinates under
$\Theta$. On the base, these moves permute the three cone meridians
up to conjugation and preserve their common character $1\in\Z_3$;
the nodal meridian still has character zero. The base homeomorphism
therefore lifts to the cyclic triple cover, and the bundle isomorphism
extends over the periodic and nodal neighborhoods. Thus the two
Lefschetz fibrations are isomorphic, with the induced identification
of their bases. The comparison does not keep a chosen torus handle
system fixed.

After this identification, we use the periodic construction for the
monodromy proof. The chain form follows by Nielsen classification, as in
the proof of Theorem~\ref{cor:chain}.

\begin{proposition}[Another handle system]\label{prop:chain-handles}
Put $c'=P^{-1}(c)$. The same fibration
has the factorization
\begin{equation}\label{eq:chain-handles}
 \begin{gathered}
 t_{c'}t_{Q(c')}t_{Q^2(c')}[A',B']=1,\\
 A'=R^{-1}P,\qquad
 B'=P^{-1}R^{-1}P^{-1}.
 \end{gathered}
\end{equation}
The two handle maps are conjugate by $P^{-1}$.
\end{proposition}

\begin{proof}
Use the second solution
$(P',Q,R',t_{c'})=(QRQ^{-1},Q,P,t_{c'})$, move its middle periodic value
past the third, and use the canonical triple pullback. The handle
words, after a common conjugation by $Q^{-1}$, are
\[
 Q^{-1}P't_{c'}^{-1},\qquad t_{c'}Qt_{c'}P'Q.
\]
Since $P'QR't_{c'}=1$, we have $t_{c'}^{-1}=QRP$. These words simplify
to $A'$ and $B'$ above. The positive relation follows
by substitution, using $P^3=R^3=Q^3=1$; the peripheral order is
$(c',Qc',Q^2c')$. The changes of base loops and their
lifts are recorded in Appendix~\ref{g19:app:identification}.
Finally $P^{-1}A' P=B'$.
\end{proof}

\subsubsection{Two pictures of the surface}\label{ssec:restored-pictures}

For these pictures, write $C'=QC$ and $c'=P^{-1}(c)$.
As in the proof of Theorem~\ref{cor:chain}, choose a $38$--chain
$\fv'$ with $P=(V')^{-26}$, and put $u'_i=C'(v'_{39-i})$.
The second solution then gives
\[
 QRQ^{-1}=(U')^{-26},\qquad
 Q=(U')^{26}t_{c'}^{-1}(V')^{26}.
\]
The first picture is the standard chain model for $P$; the second
shows the vanishing cycles of \eqref{eq:chain-handles}, with reflection
$C'$. These are descriptions of the same fiber adapted to different
periodic maps, not a specification of the relative position of both chains.

\paragraph{The cyclic chain picture.}
Write $V'=t_{v'_1}\cdots t_{v'_{38}}$. The chain $\fv'$ extends to a cycle $v'_1,\ldots,v'_{39}$ for which
\[
V'(v'_i)=v'_{i+1},\qquad P(v'_i)=v'_{i+13}
\]
with indices modulo 39. The three blocks
\[
(v'_1,\ldots,v'_{12}),\quad(v'_{14},\ldots,v'_{25}),\quad
(v'_{27},\ldots,v'_{38})
\]
have disjoint neighborhoods $S'_0,S'_1,S'_2\cong\Sig_6^1$, cyclically
permuted by $P$. Their complement $N'_P$ is $\Sig_1^3$.
If $L'_j$ is the product of the $12$ Dehn twists in the $j$th block, then
\[
V'=L'_0t_{v'_{13}}L'_1t_{v'_{26}}L'_2.
\]
Each $\partial S'_j$ is separating in $\Sig_{19}$. This decomposition
is adapted to the chain and to $P$; see Figure~\ref{g19:fig:chain}.

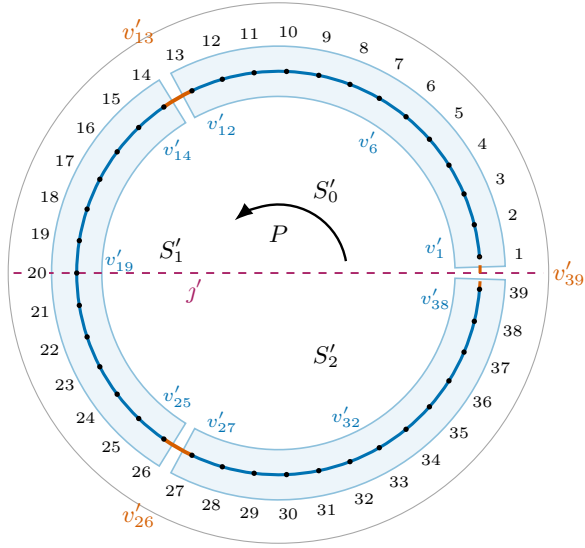
\begin{figure}[tbp]
\centering
\resizebox{.4941\textwidth}{!}{\begin{tikzpicture}[font=\small,>=Latex]
\definecolor{chainblue}{HTML}{0072B2}
\definecolor{connector}{HTML}{D55E00}
\definecolor{mirror}{HTML}{A82065}
\draw[black!35] (0,0) circle (3.55);
\filldraw[fill=chainblue!7,draw=chainblue!45,line width=.6pt] (2.31898,0.06883) arc[start angle=1.7,end angle=118.3,radius=2.32] -- (-1.41278,2.62382) arc[start angle=118.3,end angle=1.7,radius=2.98] -- cycle;
\filldraw[fill=chainblue!7,draw=chainblue!45,line width=.6pt] (-1.21909,1.97388) arc[start angle=121.7,end angle=238.3,radius=2.32] -- (-1.56591,-2.53542) arc[start angle=238.3,end angle=121.7,radius=2.98] -- cycle;
\filldraw[fill=chainblue!7,draw=chainblue!45,line width=.6pt] (-1.09988,-2.04271) arc[start angle=241.7,end angle=358.3,radius=2.32] -- (2.97869,-0.08841) arc[start angle=358.3,end angle=241.7,radius=2.98] -- cycle;
\draw[mirror,dashed,line width=.7pt] (-3.48,0)--(3.48,0);
\draw[chainblue,line width=1.15pt] (2.64141,0.21324) arc[start angle=4.615384615384615,end angle=13.846153846153847,radius=2.65];
\fill (2.64141,0.21324) circle (.035);
\node[font=\tiny,fill=white,inner sep=.5pt] at (3.15972,0.25508) {1};
\draw[chainblue,line width=1.15pt] (2.57300,0.63419) arc[start angle=13.846153846153847,end angle=23.076923076923077,radius=2.65];
\fill (2.57300,0.63419) circle (.035);
\node[font=\tiny,fill=white,inner sep=.5pt] at (3.07789,0.75863) {2};
\draw[chainblue,line width=1.15pt] (2.43795,1.03871) arc[start angle=23.076923076923077,end angle=32.30769230769231,radius=2.65];
\fill (2.43795,1.03871) circle (.035);
\node[font=\tiny,fill=white,inner sep=.5pt] at (2.91633,1.24253) {3};
\draw[chainblue,line width=1.15pt] (2.23975,1.41633) arc[start angle=32.30769230769231,end angle=41.53846153846154,radius=2.65];
\fill (2.23975,1.41633) circle (.035);
\node[font=\tiny,fill=white,inner sep=.5pt] at (2.67925,1.69426) {4};
\draw[chainblue,line width=1.15pt] (1.98355,1.75728) arc[start angle=41.53846153846154,end angle=50.76923076923077,radius=2.65];
\fill (1.98355,1.75728) circle (.035);
\node[font=\tiny,fill=white,inner sep=.5pt] at (2.37278,2.10210) {5};
\draw[chainblue,line width=1.15pt] (1.67598,2.05270) arc[start angle=50.76923076923077,end angle=60.0,radius=2.65];
\fill (1.67598,2.05270) circle (.035);
\node[font=\tiny,fill=white,inner sep=.5pt] at (2.00485,2.45550) {6};
\draw[chainblue,line width=1.15pt] (1.32500,2.29497) arc[start angle=60.0,end angle=69.23076923076923,radius=2.65];
\fill (1.32500,2.29497) circle (.035);
\node[font=\tiny,fill=white,inner sep=.5pt] at (1.58500,2.74530) {7};
\draw[chainblue,line width=1.15pt] (0.93970,2.47779) arc[start angle=69.23076923076923,end angle=78.46153846153847,radius=2.65];
\fill (0.93970,2.47779) circle (.035);
\node[font=\tiny,fill=white,inner sep=.5pt] at (1.12410,2.96400) {8};
\draw[chainblue,line width=1.15pt] (0.53007,2.59645) arc[start angle=78.46153846153847,end angle=87.6923076923077,radius=2.65];
\fill (0.53007,2.59645) circle (.035);
\node[font=\tiny,fill=white,inner sep=.5pt] at (0.63408,3.10594) {9};
\draw[chainblue,line width=1.15pt] (0.10670,2.64785) arc[start angle=87.6923076923077,end angle=96.92307692307692,radius=2.65];
\fill (0.10670,2.64785) circle (.035);
\node[font=\tiny,fill=white,inner sep=.5pt] at (0.12764,3.16743) {10};
\draw[chainblue,line width=1.15pt] (-0.31942,2.63068) arc[start angle=96.92307692307692,end angle=106.15384615384616,radius=2.65];
\fill (-0.31942,2.63068) circle (.035);
\node[font=\tiny,fill=white,inner sep=.5pt] at (-0.38210,3.14689) {11};
\draw[chainblue,line width=1.15pt] (-0.73728,2.54537) arc[start angle=106.15384615384616,end angle=115.38461538461539,radius=2.65];
\fill (-0.73728,2.54537) circle (.035);
\node[font=\tiny,fill=white,inner sep=.5pt] at (-0.88195,3.04484) {12};
\draw[connector,line width=1.45pt] (-1.13604,2.39414) arc[start angle=115.38461538461539,end angle=124.61538461538461,radius=2.65];
\fill (-1.13604,2.39414) circle (.035);
\node[font=\tiny,fill=white,inner sep=.5pt] at (-1.35896,2.86394) {13};
\draw[chainblue,line width=1.15pt] (-1.50537,2.18091) arc[start angle=124.61538461538461,end angle=133.84615384615384,radius=2.65];
\fill (-1.50537,2.18091) circle (.035);
\node[font=\tiny,fill=white,inner sep=.5pt] at (-1.80077,2.60886) {14};
\draw[chainblue,line width=1.15pt] (-1.83572,1.91119) arc[start angle=133.84615384615384,end angle=143.07692307692307,radius=2.65];
\fill (-1.83572,1.91119) circle (.035);
\node[font=\tiny,fill=white,inner sep=.5pt] at (-2.19594,2.28621) {15};
\draw[chainblue,line width=1.15pt] (-2.11852,1.59197) arc[start angle=143.07692307692307,end angle=152.30769230769232,radius=2.65];
\fill (-2.11852,1.59197) circle (.035);
\node[font=\tiny,fill=white,inner sep=.5pt] at (-2.53423,1.90435) {16};
\draw[chainblue,line width=1.15pt] (-2.34646,1.23152) arc[start angle=152.30769230769232,end angle=161.53846153846155,radius=2.65];
\fill (-2.34646,1.23152) circle (.035);
\node[font=\tiny,fill=white,inner sep=.5pt] at (-2.80690,1.47317) {17};
\draw[chainblue,line width=1.15pt] (-2.51362,0.83917) arc[start angle=161.53846153846155,end angle=170.76923076923077,radius=2.65];
\fill (-2.51362,0.83917) circle (.035);
\node[font=\tiny,fill=white,inner sep=.5pt] at (-3.00686,1.00384) {18};
\draw[chainblue,line width=1.15pt] (-2.61568,0.42509) arc[start angle=170.76923076923077,end angle=180.0,radius=2.65];
\fill (-2.61568,0.42509) circle (.035);
\node[font=\tiny,fill=white,inner sep=.5pt] at (-3.12895,0.50850) {19};
\draw[chainblue,line width=1.15pt] (-2.65000,0.00000) arc[start angle=180.0,end angle=189.23076923076923,radius=2.65];
\fill (-2.65000,0.00000) circle (.035);
\node[font=\tiny,fill=white,inner sep=.5pt] at (-3.17000,0.00000) {20};
\draw[chainblue,line width=1.15pt] (-2.61568,-0.42509) arc[start angle=189.23076923076923,end angle=198.46153846153845,radius=2.65];
\fill (-2.61568,-0.42509) circle (.035);
\node[font=\tiny,fill=white,inner sep=.5pt] at (-3.12895,-0.50850) {21};
\draw[chainblue,line width=1.15pt] (-2.51362,-0.83917) arc[start angle=198.46153846153845,end angle=207.69230769230768,radius=2.65];
\fill (-2.51362,-0.83917) circle (.035);
\node[font=\tiny,fill=white,inner sep=.5pt] at (-3.00686,-1.00384) {22};
\draw[chainblue,line width=1.15pt] (-2.34646,-1.23152) arc[start angle=207.69230769230768,end angle=216.92307692307693,radius=2.65];
\fill (-2.34646,-1.23152) circle (.035);
\node[font=\tiny,fill=white,inner sep=.5pt] at (-2.80690,-1.47317) {23};
\draw[chainblue,line width=1.15pt] (-2.11852,-1.59197) arc[start angle=216.92307692307693,end angle=226.15384615384616,radius=2.65];
\fill (-2.11852,-1.59197) circle (.035);
\node[font=\tiny,fill=white,inner sep=.5pt] at (-2.53423,-1.90435) {24};
\draw[chainblue,line width=1.15pt] (-1.83572,-1.91119) arc[start angle=226.15384615384616,end angle=235.3846153846154,radius=2.65];
\fill (-1.83572,-1.91119) circle (.035);
\node[font=\tiny,fill=white,inner sep=.5pt] at (-2.19594,-2.28621) {25};
\draw[connector,line width=1.45pt] (-1.50537,-2.18091) arc[start angle=235.3846153846154,end angle=244.6153846153846,radius=2.65];
\fill (-1.50537,-2.18091) circle (.035);
\node[font=\tiny,fill=white,inner sep=.5pt] at (-1.80077,-2.60886) {26};
\draw[chainblue,line width=1.15pt] (-1.13604,-2.39414) arc[start angle=244.6153846153846,end angle=253.84615384615384,radius=2.65];
\fill (-1.13604,-2.39414) circle (.035);
\node[font=\tiny,fill=white,inner sep=.5pt] at (-1.35896,-2.86394) {27};
\draw[chainblue,line width=1.15pt] (-0.73728,-2.54537) arc[start angle=253.84615384615384,end angle=263.0769230769231,radius=2.65];
\fill (-0.73728,-2.54537) circle (.035);
\node[font=\tiny,fill=white,inner sep=.5pt] at (-0.88195,-3.04484) {28};
\draw[chainblue,line width=1.15pt] (-0.31942,-2.63068) arc[start angle=263.0769230769231,end angle=272.3076923076923,radius=2.65];
\fill (-0.31942,-2.63068) circle (.035);
\node[font=\tiny,fill=white,inner sep=.5pt] at (-0.38210,-3.14689) {29};
\draw[chainblue,line width=1.15pt] (0.10670,-2.64785) arc[start angle=272.3076923076923,end angle=281.53846153846155,radius=2.65];
\fill (0.10670,-2.64785) circle (.035);
\node[font=\tiny,fill=white,inner sep=.5pt] at (0.12764,-3.16743) {30};
\draw[chainblue,line width=1.15pt] (0.53007,-2.59645) arc[start angle=281.53846153846155,end angle=290.7692307692308,radius=2.65];
\fill (0.53007,-2.59645) circle (.035);
\node[font=\tiny,fill=white,inner sep=.5pt] at (0.63408,-3.10594) {31};
\draw[chainblue,line width=1.15pt] (0.93970,-2.47779) arc[start angle=290.7692307692308,end angle=300.0,radius=2.65];
\fill (0.93970,-2.47779) circle (.035);
\node[font=\tiny,fill=white,inner sep=.5pt] at (1.12410,-2.96400) {32};
\draw[chainblue,line width=1.15pt] (1.32500,-2.29497) arc[start angle=300.0,end angle=309.2307692307692,radius=2.65];
\fill (1.32500,-2.29497) circle (.035);
\node[font=\tiny,fill=white,inner sep=.5pt] at (1.58500,-2.74530) {33};
\draw[chainblue,line width=1.15pt] (1.67598,-2.05270) arc[start angle=309.2307692307692,end angle=318.46153846153845,radius=2.65];
\fill (1.67598,-2.05270) circle (.035);
\node[font=\tiny,fill=white,inner sep=.5pt] at (2.00485,-2.45550) {34};
\draw[chainblue,line width=1.15pt] (1.98355,-1.75728) arc[start angle=318.46153846153845,end angle=327.6923076923077,radius=2.65];
\fill (1.98355,-1.75728) circle (.035);
\node[font=\tiny,fill=white,inner sep=.5pt] at (2.37278,-2.10210) {35};
\draw[chainblue,line width=1.15pt] (2.23975,-1.41633) arc[start angle=327.6923076923077,end angle=336.9230769230769,radius=2.65];
\fill (2.23975,-1.41633) circle (.035);
\node[font=\tiny,fill=white,inner sep=.5pt] at (2.67925,-1.69426) {36};
\draw[chainblue,line width=1.15pt] (2.43795,-1.03871) arc[start angle=336.9230769230769,end angle=346.15384615384613,radius=2.65];
\fill (2.43795,-1.03871) circle (.035);
\node[font=\tiny,fill=white,inner sep=.5pt] at (2.91633,-1.24253) {37};
\draw[chainblue,line width=1.15pt] (2.57300,-0.63419) arc[start angle=346.15384615384613,end angle=355.38461538461536,radius=2.65];
\fill (2.57300,-0.63419) circle (.035);
\node[font=\tiny,fill=white,inner sep=.5pt] at (3.07789,-0.75863) {38};
\draw[connector,line width=1.15pt,dashed] (2.64141,-0.21324) arc[start angle=355.38461538461536,end angle=364.61538461538464,radius=2.65];
\fill (2.64141,-0.21324) circle (.035);
\node[font=\tiny,fill=white,inner sep=.5pt] at (3.15972,-0.25508) {39};
\node[font=\scriptsize,text=chainblue] at (2.06294,0.33526) {$v'_{1}$};
\node[font=\scriptsize,text=chainblue] at (1.18726,1.72004) {$v'_{6}$};
\node[font=\scriptsize,text=chainblue] at (-0.74112,1.95418) {$v'_{12}$};
\node[font=\scriptsize,text=chainblue] at (-1.32181,1.61892) {$v'_{14}$};
\node[font=\scriptsize,text=chainblue] at (-2.08322,0.16818) {$v'_{19}$};
\node[font=\scriptsize,text=chainblue] at (-1.32181,-1.61892) {$v'_{25}$};
\node[font=\scriptsize,text=chainblue] at (-0.74112,-1.95418) {$v'_{27}$};
\node[font=\scriptsize,text=chainblue] at (0.89597,-1.88821) {$v'_{32}$};
\node[font=\scriptsize,text=chainblue] at (2.06294,-0.33526) {$v'_{38}$};
\node[font=\small,text=connector] at (-1.83500,3.17831) {$v'_{13}$};
\node[font=\small,text=connector] at (-1.83500,-3.17831) {$v'_{26}$};
\node[font=\small,text=connector] at (3.82000,-0.00000) {$v'_{39}$};
\node[text=mirror,fill=white,inner sep=1.5pt] at (-1.1,-.23) {$\jmath'$};
\draw[->,line width=.9pt] (0.88633,0.15628) arc[start angle=10,end angle=130,radius=.9];
\node at (0,.52) {$P$};
\node at (0.62500,1.08253) {$S'_0$};
\node at (-1.4,.30) {$S'_1$};
\node at (0.62500,-1.08253) {$S'_2$};
\end{tikzpicture}
}
\caption{The cyclic chain in the branched double-cover model. The
numbers label branch points, and the arcs lift to the curves $v'_i$.
The double cover, with its boundary capped, is $\Sigma_{19}$.
The shaded blue disks lift to three genus-6 subsurfaces, cyclically
permuted by $P$. The dashed arc $v'_{39}$ completes the cycle and is
omitted from the Dehn twist product of the $38$--chain. The dashed
reflection lifts to $\jmath'$, with $\jmath'(v'_i)=v'_{39-i}$; the conjugating
map is $h'=C'\jmath'$.}
\label{g19:fig:chain}
\end{figure}

\paragraph{The vanishing-cycle torus and its complement.}\label{g19:sec:complement}

Choose a $\langle Q,C'\rangle$-invariant regular neighborhood
$N'$ of $c'\cup Q(c')\cup Q^2(c')$. Then
\begin{equation}\label{g19:eq:decomposition}
\Sig_{19}=N'\cup_{\alpha_0\sqcup\alpha_1\sqcup\alpha_2}W'_c,
\qquad N'\cong\Sig_1^3,\qquad W'_c\cong\Sig_{16}^3.
\end{equation}
The representatives of $Q$ and $C'$ preserve each $\alpha_i$
setwise; their restrictions to $N'$ and $W'_c$ need not fix the boundary
pointwise.

The three boundary classes have rank two in $H_1(\Sig_{19};\Z)$.
This shows that $W'_c$ is connected: for a connected subsurface with
three boundary components, each component of its complement gives a
boundary homology relation, and the boundary span has rank three
minus the number of complementary components. Euler characteristic
then gives
\[
-36=\chi(\Sig_{19})=\chi(N')+\chi(W'_c)=-3+\chi(W'_c),
\]
so $W'_c\cong\Sig_{16}^3$. Each $\alpha_i$ is nonseparating, since after
cutting along just that circle, the two pieces remain joined through
the other two circles.

To describe the action on $N'$, use coordinates
$(r,s)\in\mathbb R^2/\Z^2$ and the affine maps
\begin{equation}\label{g19:eq:affine}
Q(r,s)=(s,-r-s),\qquad C'(r,s)=(-r-s,s).
\end{equation}
Remove small invariant disks around $(0,0)$, $(1/3,1/3)$, and
$(2/3,2/3)$. The remaining surface is $N'$. The curve equations are
\[
c':\ s=1/6,\qquad a:\ r=1/6,\qquad b:\ r+s=5/6
\quad\pmod 1.
\]
The rotation cyclically permutes $c',a,b$, while $C'$ preserves $c'$
and interchanges $a,b$. Its fixed set on the filled torus is
$2r+s=0\pmod1$, a single circle; removing the three disks leaves
three properly embedded real arcs in $N'$. There are two fixed points
on each boundary circle. Figure~\ref{g19:fig:torus} uses the Euclidean
coordinates $(x,y)=(r+s/2,\sqrt3s/2)$, in which $Q$ is rotation
through $-2\pi/3$ and $C'$ is reflection across the vertical axis.
The three removed disks are replaced in the closed surface by
the connected complementary surface $W'_c$.

\begin{figure}[tbp]
\centering
\resizebox{.441\textwidth}{!}{\begin{tikzpicture}[font=\small,>=Latex]
\definecolor{ca}{HTML}{0072B2}
\definecolor{cb}{HTML}{D55E00}
\definecolor{cc}{HTML}{009E73}
\definecolor{real}{HTML}{A82065}
\definecolor{boundaryone}{HTML}{7B3294}
\definecolor{boundarytwo}{HTML}{B89B30}
\begin{scope}[x=6cm,y=6cm]
\fill[black!3] (0.5000000,0.2886751) -- (0.0000000,0.5773503) -- (-0.5000000,0.2886751) -- (-0.5000000,-0.2886751) -- (-0.0000000,-0.5773503) -- (0.5000000,-0.2886751) -- cycle;
\begin{scope}
\clip (0.5000000,0.2886751) -- (0.0000000,0.5773503) -- (-0.5000000,0.2886751) -- (-0.5000000,-0.2886751) -- (-0.0000000,-0.5773503) -- (0.5000000,-0.2886751) -- cycle;
\draw[cc,line width=1.45pt,postaction={decorate},decoration={markings,mark=at position .58 with {\arrow{Latex}}}] (-0.5000000,0.1443376) -- (0.5000000,0.1443376);
\draw[ca,line width=1.45pt,postaction={decorate},decoration={markings,mark=at position .58 with {\arrow{Latex}}}] (0.3750000,0.3608439) -- (-0.1250000,-0.5051815);
\draw[cb,line width=1.45pt,postaction={decorate},decoration={markings,mark=at position .58 with {\arrow{Latex}}}] (0.1250000,-0.5051815) -- (-0.3750000,0.3608439);
\filldraw[fill=white,draw=black!65,line width=.6pt] (0,0) circle (.046);
\node[font=\scriptsize] at (0,0) {$\alpha_0$};
\filldraw[fill=white,draw=boundaryone,line width=.7pt] (0.5000000,0.2886751) circle (.064);
\filldraw[fill=white,draw=boundarytwo,line width=.7pt] (0.0000000,0.5773503) circle (.064);
\filldraw[fill=white,draw=boundaryone,line width=.7pt] (-0.5000000,0.2886751) circle (.064);
\filldraw[fill=white,draw=boundarytwo,line width=.7pt] (-0.5000000,-0.2886751) circle (.064);
\filldraw[fill=white,draw=boundaryone,line width=.7pt] (-0.0000000,-0.5773503) circle (.064);
\filldraw[fill=white,draw=boundarytwo,line width=.7pt] (0.5000000,-0.2886751) circle (.064);
\end{scope}
\draw[black!60,line width=.6pt] (0.4445744,0.3206751) -- (0.0554256,0.5453503);
\draw[black!60,line width=.6pt] (-0.0554256,0.5453503) -- (-0.4445744,0.3206751);
\draw[black!60,line width=.6pt] (-0.5000000,0.2246751) -- (-0.5000000,-0.2246751);
\draw[black!60,line width=.6pt] (-0.4445744,-0.3206751) -- (-0.0554256,-0.5453503);
\draw[black!60,line width=.6pt] (0.0554256,-0.5453503) -- (0.4445744,-0.3206751);
\draw[black!60,line width=.6pt] (0.5000000,-0.2246751) -- (0.5000000,0.2246751);
\draw[real,line width=1.4pt,dash pattern=on 3pt off 2pt] (0.0000000,0.0460000) -- (0.0000000,0.5133503);
\draw[real,line width=1.4pt,dash pattern=on 3pt off 2pt] (0.0000000,-0.0460000) -- (0.0000000,-0.5133503);
\draw[real,line width=1.4pt,dash pattern=on 3pt off 2pt] (-0.5000000,-0.2246751) -- (-0.5000000,0.2246751);
\draw[real,line width=1.4pt,dash pattern=on 3pt off 2pt] (0.5000000,-0.2246751) -- (0.5000000,0.2246751);
\node[text=cc] at (0.3400000,0.1893376) {$c'$};
\fill (0.2500000,0.1443376) circle (.007);
\node[text=ca] at (0.3100000,0.3000000) {$a$};
\fill (0.0000000,-0.2886751) circle (.007);
\node[text=cb] at (-0.3100000,0.3000000) {$b$};
\fill (-0.2500000,0.1443376) circle (.007);
\filldraw[fill=white,draw=real,line width=.7pt] (0.0000000,0.1443376) circle (.012);
\filldraw[fill=white,draw=real,line width=.7pt] (-0.5000000,0.1443376) circle (.012);
\filldraw[fill=white,draw=real,line width=.7pt] (0.5000000,0.1443376) circle (.012);
\draw[->,black!65,line width=.8pt] (.34,.5888973) arc[start angle=60,end angle=-60,radius=.68];
\node at (.74,.12) {$Q$};
\node[text=real,font=\scriptsize,fill=white,inner sep=1pt] at (.11,.43) {$\mathrm{Fix}(C')$};
\node[font=\scriptsize,align=center] at (0,-.70) {Opposite sides are identified by translations.\\$C'(x,y)=(-x,y)$; dashed arcs are fixed by $C'$.};
\end{scope}
\node[font=\scriptsize,text=boundaryone] at (-2.15,-3.20) {$\alpha_1$};
\node[font=\scriptsize,text=boundarytwo] at (2.15,-3.20) {$\alpha_2$};
\end{tikzpicture}
}
\caption{The vanishing-cycle torus $N'\cong\Sig_1^3$, with
$a=Q(c')$ and $b=Q^2(c')$. The dashed arcs are fixed by $C'$.
The connected complement is $W'_c\cong\Sig_{16}^3$.}
\label{g19:fig:torus}
\end{figure}
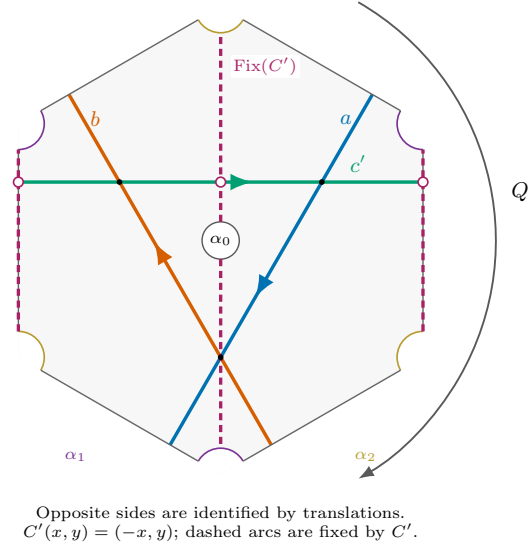

We next describe the action on the connected complement $W'_c$.
The quotient $\Sig_{19}/\langle Q\rangle$ is a genus-6 surface
$B_Q$ with three order-$3$ branch points. The covering is free over
$N'$, and its quotient ${\mathcal P}_c=N'/\langle Q\rangle$ is an embedded pair of
pants disjoint from the branch points. The three boundary coverings
are connected and have degree three.

Let $B_c=\overline{B_Q\setminus {\mathcal P}_c}$. It is connected, has three boundary
components, and
\[
\chi(B_c)=\chi(B_Q)-\chi({\mathcal P}_c)=-10-(-1)=-9.
\]
Thus $B_c\cong\Sig_4^3$. All three branch points lie in its interior,
and the restriction is the connected cyclic branched covering
\begin{equation}\label{g19:eq:Mcover}
W'_c\longrightarrow B_c\cong\Sig_4^3
\end{equation}
of degree three, with a single point over each branch point.
Counting Euler characteristics gives $\chi(W'_c)=3(-9)-3(3-1)=-33$, as above.
The deck map is $Q|_{W'_c}$; in particular it has three fixed points in
$W'_c$.

Section~\ref{ssec:chain-model} specifies the covering by a quotient
ribbon graph and its sheet increments. The pair of pants ${\mathcal P}_c$ is
a neighborhood of bands 13 and 14; its complement gives $W'_c$ with
the inherited covering. Matching the labeled boundary lifts
specifies the gluing and the global action of $Q$.

The involution $C'$ also preserves $W'_c$, reverses its orientation,
and satisfies
\[
(C'|_{W'_c})^2=1,\qquad (C'|_{W'_c})(Q|_{W'_c})(C'|_{W'_c})=(Q|_{W'_c})^{-1}.
\]
It reverses the orientation of each $\alpha_i$. The based actions used
in Corollary~\ref{thm:chain-model}, together with the gluing and boundary
words in Section~\ref{ssec:chain-model}, specify the global
mapping classes and their restrictions to $W'_c$.

The homology action satisfies
\[
\rank_{\F_2}((C')_*-I)=18,
\qquad \operatorname{diag}(\Omega (C')_*)\not\equiv0\pmod2.
\]
Here $\Omega$ is the integral intersection matrix of the identification.
Since the involution has fixed points on $c'$, the real-involution
normal form gives two fixed ovals, and the nonzero diagonal gives
the nondividing type; see \cite[Section~2]{KK}. The finite group
$\langle Q,C'\rangle$ can be realized simultaneously by Nielsen
realization \cite{Kerckhoff}.

The six fixed points on $\partial W'_c$ are endpoints of three properly
embedded fixed arcs in $W'_c$. There may also be closed fixed components
contained in $W'_c$. Together with the three arcs in $N'$, the fixed set
consists of the two ovals described above.

\begin{remark}\label{rem:two-decompositions}
The chain torus $N'_P$ has three genus-6 complementary components
and separating boundary curves. The vanishing-cycle torus $N'$ has a
connected genus-16 complement and nonseparating boundary
curves. No homeomorphism of the closed surface takes $N'_P$ to $N'$.
\end{remark}

\subsection{Topology of the total space}\label{ssec:invariants}\label{g19:sec:CScomparison}
We first calculate the homology and intersection form of $X$; the group comparison that follows gives asphericity and the $s$-cobordism with $\CS$.

\begin{proposition}\label{prop:invariants}
$e(X)=3$, $\sigma(X)=1$, $b_1(X)=2$, $H_1(X;\Z)\cong\Z^2$,
$H_2(X;\Z)\cong\Z^5$ with form $3\langle1\rangle\oplus2\langle-1\rangle$,
and the total space is minimal symplectic of general type.
\end{proposition}

\begin{proof}
The base torus has Euler characteristic zero and each nodal fiber adds
one, so $e(X)=3$. All singular fibers are irreducible, so $\pi$ is
relatively minimal. Since the base has positive genus, $X$ is minimal
\cite[Lemma~2.8]{DZ}. The usual construction gives a symplectic
form~\cite{GS}; in particular $X$ is almost complex, so
$e+\sigma\equiv0\pmod4$. With three nonseparating vanishing cycles,
\cite[Proposition~2.3]{SY} bounds the signature by $|\sigma(X)|\le2$,
and the congruence forces $\sigma(X)=1$.

For $A=QR^{-1}$, the characteristic polynomial in
Appendix~\ref{app:geometric-pa} gives
\[
 \det(I-A_*)=2^2h(1)h^*(1)=16.
\]
Thus $A_*-I$ is invertible over $\Q$, so the fiber contributes
nothing to $H_1(X;\Q)$. The map to $H_1(E;\Q)$ is onto, with kernel
the image of the fiber, and hence $b_1(X)=2$.

Integrally, the split surjection $H_1(X;\Z)\to H_1(E;\Z)=\Z^2$
has kernel the quotient of
\[
 \mathcal H=H_1(\Sigma;\Z)\big/\bigl(\operatorname{im}(A_*-I)
 +\operatorname{im}(B_*-I)+\langle[\eps_0],[\eps_1],[\eps_2]\rangle\bigr)
\]
by the class $[w]$ of the based correction
$[\tilde A,\tilde B]\tilde t_0\tilde t_1\tilde t_2=\Inn(w)$.
Here $\tilde A=\tilde Q\tilde\rho$, $\tilde B=\tilde\rho\tilde Q$,
and $\tilde t_i=\tilde P^i\widetilde t_c\tilde P^{-i}$.
Appendix~\ref{ssec:finite} gives $\mathcal H\cong\Z_2$ and the
functional $\ell$ detecting this quotient, with $\ell([w])=1$.
Thus $[w]$ generates $\mathcal H$, and $H_1(X;\Z)\cong\Z^2$.
Poincar\'e duality gives torsion-free $H_2(X;\Z)$ of rank $5$;
its form is unimodular, indefinite and odd since $\sigma=1$, hence
$3\langle1\rangle\oplus2\langle-1\rangle$.
Finally, $b_2^+(X)=3$ excludes rational and ruled surfaces, and
$K_\omega^2=2e+3\sigma=9>0$ gives symplectic Kodaira dimension $2$
\cite[Corollary~2.6]{DZ}.
\end{proof}

We next compare $\pi_1(X)$ with $\pi_1(\CS)$.
We compute the group before making the $3$--fold change of base.
This leaves the symmetry visible and gives a smaller presentation.

\begin{proposition}\label{prop:genus19group}
The total space $X$ satisfies
\begin{equation}\label{g19:eq:pi1}
 \pi_1(X)\cong\pi_1(\CS).
\end{equation}
It is aspherical and topologically $s$-cobordant to $\CS$.
\end{proposition}

\begin{proof}
For the group calculation, we use the original fiber generators.
Lemma~\ref{lem:genus19fibercomparison} identifies the resulting
presentation with the fibration of Proposition~\ref{prop:models},
including the character defining the triple pullback. Write the fiber group as
$\Pi=\langle x_1,\ldots,x_{38}\mid\mathfrak r\rangle$, its based periodic actions
as $\mathsf P,\mathsf Q,\mathsf R$, and the vanishing cycle as
$c_x=x_{18}x_{32}x_{28}$. Over the punctured sphere, van Kampen's
theorem adjoins three base loops, acting by these automorphisms.
Filling the nodal neighborhood kills $c_x$; filling each periodic
neighborhood imposes the cube of its base loop, with the based correction.
The resulting group is
\begin{equation}\label{g19:eq:orbifoldgroup}
\mathcal E=\left\langle
\begin{array}{c|l}
 \Pi,p,q,r & c_x=1,\quad
 p x_i p^{-1}=\mathsf P(x_i),\\
 &q x_i q^{-1}=\mathsf Q(x_i),\quad
 r x_i r^{-1}=\mathsf R(x_i),\\
 &p^3=w_P,\quad q^3=w_Q,\quad r^3=w_R,\quad pqr=w_0
\end{array}\right\rangle.
\end{equation}
Here $1\le i\le38$. The words $w_P,w_Q,w_R,w_0$ are characterized by
$\mathsf P^3=\Inn(w_P)$, and similarly for $\mathsf Q,\mathsf R$, and
$\mathsf t_c\mathsf P\mathsf Q\mathsf R=\Inn(w_0)$, where
$\mathsf t_c$ represents the Dehn twist $t_{c_x}$ with the chosen
basepoint path. The fiber is killed by the character
$\chi:\mathcal E\to\Z_3$, while $\chi(p)=\chi(q)=\chi(r)=1$.
The cyclic pullback gives
$\pi_1(X)=\Gamma:=\ker\chi$. In particular, the cube and product
corrections in \eqref{g19:eq:orbifoldgroup} must not be omitted.

We reduce $\mathcal E$ to three generators, keeping track of $\chi$
to recover $\pi_1(X)$. In the original coordinates, put
\[
 s=x_{37}p x_{20}x_7x_{20}^{-1}r^{-1},\qquad
 t=x_{20}^{-1}r x_{20},\qquad
 a=r x_{20}s t,\quad b=s^2t,\quad c=sa.
\]
The Tietze reduction gives
\begin{equation}\label{eq:small-orbifold-presentation}
 \mathcal E=\langle a,b,c\mid r_1,\ldots,r_{10}\rangle,
 \qquad \chi(a)=\chi(c)=2,\quad\chi(b)=1,
\end{equation}
with the following relators, where a bar denotes inverse:
\[
\begin{array}{r@{\quad}l}
 r_{1}=&a^3 \\
 r_{2}=&(\bar c\bar a)^3 \\
 r_{3}=&(\bar b c\bar a)^3 \\
 r_{4}=&cab\bar{c}\bar{a}\bar{c}bacaba \\
 r_{5}=&\bar{b}\bar{c}a\bar{b}\bar{c}ba\bar{c}\bar{b}\bar{c}bc \\
 r_{6}=&(c\bar a\bar b c)^3 \\
 r_{7}=&a\bar{c}\bar{b}^{2}cb\bar{c}\bar{b}c^{2}b\bar{a}b\bar{c} \\
 r_{8}=&b\bar{a}c\bar{b}a\bar{b}^{3}a\bar{b}\bar{c}b\bar{a}ba \\
 r_{9}=&\bar{b}cb^{2}\bar{a}b\bar{c}a\bar{b}ac\bar{b}a\bar{b}\bar{c} \\
 r_{10}=&\bar{b}a\bar{b}\bar{c}b\bar{a}\bar{c}b^{3}\bar{a}bacabc
\end{array}
\]
The reduction from \eqref{g19:eq:orbifoldgroup} is the long finite
step: it eliminates the fiber generators and the original base
loops in favor of $a,b,c$. Appendix~\ref{g19:app:CSgroup}
identifies the records of this reduction and the generator
expressions in both directions.
The inverse generator expressions ensure that
\eqref{eq:small-orbifold-presentation} presents the whole group.

We now pass to $\pi_1(X)=\ker\chi$. Using the transversal
$1,b,b^2$, its seven Schreier generators, in order, are
\[
 u_1=ab^{-2},\quad u_2=cb^{-2},\quad
 u_3=ba,\quad u_4=bc,\quad
 u_5=b^2ab^{-1},\quad u_6=b^3,\quad u_7=b^2cb^{-1}.
\]
Each relator $r_i$, read from each of the three cosets, gives one
kernel relator. For example, $a^3$ gives the cyclic versions of
$u_1u_5u_3=1$. The kernel is generated by
\[
 h_1=u_5,\qquad h_2=u_5^{-1}u_2^{-1}u_3u_2u_5,
 \qquad h_3=u_5^3u_3u_2.
\]
The generators can be eliminated in the following order, after
adjoining these definitions:
\[
\begin{array}{r@{\;=\;}l}
 u_5&h_1,\\
 u_1&u_3^{-1}h_1^{-1},\\
 u_2&u_3^{-1}h_1^{-3}h_3,\\
 u_6&h_1u_3u_4^{-1}h_1^{-1}h_2^{-1}u_4,\\
 u_7&h_1u_3u_4^{-1}h_1^{-1}h_3^{-1}h_1^3,\\
 u_3&u_4^{-1}h_1^{-2}h_3h_1h_2h_3^{-1}h_1u_4,\\
 u_4&h_2h_1^{-3}h_3^3h_2h_3^{-1}h_2^{-2}h_1^{-1}h_3^{-1}h_1^2.
\end{array}
\]
Each line follows from a rewritten relator after the preceding
substitutions and reductions using the other relators. The last
line leaves only $h_1,h_2,h_3$, so these generate $\Gamma$.
Simplifying the remaining relators gives the thirteen-relator
presentation; see Appendix~\ref{g19:app:CSgroup} for the reduction
record. To identify it with the CS presentation on $a_1,a_2,a_3$ \cite{CSpresentations}, consider
the assignments
\[
 (h_1,h_2,h_3)\longmapsto
 (a_3^{-1},a_1^{-1}a_2a_1,a_1),\qquad
 (a_1,a_2,a_3)\longmapsto
 (h_3,h_3h_2h_3^{-1},h_1^{-1}).
\]
These assignments respect all 13 relators in both directions. Both compositions
are the identity after free cancellation, so they give
$\Gamma\cong\pi_1(\CS)$. Appendix~\ref{g19:app:CSgroup} identifies
the record of these relator substitutions.

Since $\CS$ is a closed aspherical $4$--manifold and
$\eu(X)=\eu(\CS)=3$, Hillman's criterion
\cite[Corollary~3.5.1]{Hillman} makes $X$ aspherical.
The group isomorphism gives $X\simeq\CS$, with the equivalence
preserving orientation because both signatures are $1$. As $\CS$ is complex
hyperbolic, \cite[Theorem~9.12]{Hillman} gives the $s$-cobordism.
\end{proof}
\enlargethispage{0.15in}

\section{Genus-4 Lefschetz fibrations over a genus-2 surface}\label{sec:genus4}

The genus-4 family is obtained by twisting the BC fibrations by
characters of the base. We give the monodromies first, then their
covering construction and classification. As in Theorem~\ref{thm:smallfibrations},
$X_k=X_{k(1,-1,1,2)}$, $k=0,1,4$, denotes an original BC surface.

\subsection{The monodromy family}

Let $G$ be the closed genus-2 surface given by the cyclic
cover $w^5=x(x-1)^3$ of the $x$-sphere. Let $Y$ be the blowup of
$\mathbb{CP}^1_x\times\mathbb{CP}^1_y$ at $(0,0),(1,1),(\infty,\infty)$;
thus $Y\cong(S^2\times S^2)\# 3\overline{\mathbb{CP}}{}^2$.
The degree-$25$ branched covers $X_k\to Y$ are determined by
\begin{equation}\label{g4:eq:catalogfields}
 w^5=x(x-1)^3,\qquad z^5=x^k y(y-1)(y-x)^2.
\end{equation}
The surfaces $X_0,X_1,X_4$ correspond to $S_2,S_4,S_1$ of
Bauer--Catanese \cite[Section~3]{BC}; the five values of $k$ form
the covering orbits $\{0,3\}$, $\{1,2\}$ and $\{4\}$.
The case $k=0$ is the standard $n=5$ construction of
Catanese--Dettweiler \cite[Definition~4.1 and Remark~4.1]{CD}.
We describe the branching and check these identifications in
Section~\ref{g4:sec:geometry}.

Consider the cyclic $5$--fold cover of a disk with three ordered branch
points, with local monodromy $1\in\Z_5$ at each point. Its total space is
$\Sigma_4^1$: it has Euler characteristic $5-3\cdot4=-7$ and one boundary
component. Let $U,V\in\Gamma_4^1$ be the boundary-fixing lifts of the two
positive adjacent half twists. Each is a product of four positive Dehn
twists:
\begin{equation}\label{g4:eq:UV}
 U=t_{\al_1}t_{\al_2}t_{\al_3}t_{\al_4},\qquad
 V=t_{\be_1}t_{\be_2}t_{\be_3}t_{\be_4}.
\end{equation}
The two $4$--chains are described in
Section~\ref{g4:sec:identification}. If $\ga_0,\ga_1$ are the boundaries of their
genus-2 regular neighborhoods, the chain and braid relations give
\begin{equation}\label{g4:eq:basic}
 UVU=VUV,\qquad U^{10}=t_{\ga_0},\qquad V^{10}=t_{\ga_1}.
\end{equation}
Put $Z=(UV)^3$. This element is central in $\langle U,V\rangle$, and
\begin{equation}\label{g4:eq:center}
 Z^5=(UV)^{15}=t_\partial.
\end{equation}
After capping the boundary, $Z$ is a deck transformation of exact order $5$.

To write the four handle monodromies, set
\begin{equation}\label{g4:eq:D}
 D=U^2V^2U^{-8}V^4U^{-2}
\end{equation}
and define
\begin{equation}\label{g4:eq:handles}
\begin{aligned}
 A_1&=Z^{-3}D(U^2V^{-2}U^{-6})D^{-1},&
 B_1&=Z^3D(U^8V^2U^{-4})D^{-1},\\
 A_2&=Z^2D(U^8V^{-2}U^{-2})D^{-1},&
 B_2&=Z^{-1}D(U^2V^{-4})D^{-1}.
\end{aligned}
\end{equation}
Finally, let
\begin{equation}\label{g4:eq:cycles}
 d_1=\ga_1,\qquad d_2=V^{-1}(\ga_0),\qquad
 d_3=U^2V^2(\ga_0).
\end{equation}

For the character $\lambda=(r_1,s_1,r_2,s_2)\in H^1(G;\Z_5)$,
the four handle monodromies are
\begin{equation}\label{g4:eq:twisthandles}
 Z^{r_1}A_1,\quad Z^{s_1}B_1,\quad
 Z^{r_2}A_2,\quad Z^{s_2}B_2.
\end{equation}
In $\Gamma_4^1$ we choose integer representatives for the exponents;
changing them by multiples of five changes the individual lifts by
boundary Dehn twists and leaves the commutators unchanged. The original
three fibrations have $\lambda=k(1,-1,1,2)$, with handle words
\begin{equation}\label{g4:eq:cataloghandles}
\renewcommand{\arraystretch}{1.3}
\begin{array}{c|cccc}
 &A_1^{(k)}&B_1^{(k)}&A_2^{(k)}&B_2^{(k)}\\ \hline
 X_0&A_1&B_1&A_2&B_2\\
 X_1&ZA_1&Z^{-1}B_1&ZA_2&Z^2B_2\\
 X_4&Z^{-1}A_1&ZB_1&Z^{-1}A_2&Z^{-2}B_2
\end{array}
\end{equation}
All three rows use the same identification of the fiber.

\begin{proposition}\label{g4:prop:main}
For every $\lambda\in H^1(G;\Z_5)$, the mapping classes in
\eqref{g4:eq:twisthandles} satisfy
\begin{equation}\label{g4:eq:mainboundary}
 t_{d_1}t_{d_2}t_{d_3}
 [Z^{r_1}A_1,Z^{s_1}B_1]
 [Z^{r_2}A_2,Z^{s_2}B_2]=t_\partial
 \quad\text{in }\Gamma_4^1.
\end{equation}
Capping the boundary gives the closed factorization of
Theorem~\ref{thm:smallfibrations}.
All three vanishing cycles separate $\Sigma_4$ into two genus-2
subsurfaces. This is the monodromy of a holomorphic fibration
$X_\lambda\to G$, with $(\eu,\sigma)=(15,5)$ and a section
of square $-1$.
\end{proposition}

We prove \eqref{g4:eq:mainboundary} in Section~\ref{g4:sec:proof}.
It records one of the four sections; the covering construction below
gives all four and the holomorphic realizations.

\subsection{The fiber and the Dehn twist curves}\label{g4:sec:identification}

A real fiber is obtained by adding the point at infinity to
\begin{equation}\label{g4:eq:realcurve}
 \eta^5=t^3-4t.
\end{equation}
The inverse image of a disk containing $-2,0,2$ is the surface used above.
The inverse image of the two real intervals has vertices $L,M,R$, edges
$e_j:L\to M$, and edges $f_j:M\to R$, where $j\in\Z_5$.

Its oriented thickening is determined by the positive vertex orders
\[
\begin{aligned}
 L &: (e_0,e_1,e_2,e_3,e_4),\\
 M &: (f_0,\bar e_0,f_1,\bar e_1,f_2,\bar e_2,f_3,\bar e_3,f_4,\bar e_4),\\
 R &: (\bar f_0,\bar f_1,\bar f_2,\bar f_3,\bar f_4).
\end{aligned}
\]
Here, a bar reverses the orientation of an edge.

Using the tree $e_0\cup f_0$ and base vertex $L$, put
\[
 E_j=e_je_0^{-1},\qquad F_j=e_0f_jf_0^{-1}e_0^{-1}\quad(1\leq j\leq4),
 \qquad E_0=F_0=1.
\]
These freely generate the fundamental group of the surface with boundary.
With a compatible access path to the boundary, its positive boundary word is
\begin{equation}\label{g4:eq:r}
\begin{split}
 r={}&E_1F_2F_3^{-1}E_3^{-1}E_4F_1^{-1}E_1^{-1}E_2\\
 &\hspace{9mm}\cdot F_3F_4^{-1}E_4^{-1}F_1F_2^{-1}E_2^{-1}E_3F_4.
\end{split}
\end{equation}
The Dehn twist curves in \eqref{g4:eq:UV} are the simple $2$--edge cycles
\begin{equation}\label{g4:eq:edgecurves}
 \al_j\sim E_jE_{j+1}^{-1},\qquad
 \be_j\sim F_jF_{j+1}^{-1}\quad(1\leq j\leq4),
\end{equation}
where indices are taken modulo $5$. Their ordered positive products are
the lifted half twists by the cyclic-cover chain formula
\cite[Theorem 3.6]{CKS}; see also \cite[Section 5]{GM}.

For $x_1,\ldots,x_4$ set
\[
 w_\partial(x_1,x_2,x_3,x_4)=x_1x_2^{-1}x_3x_4^{-1}x_1^{-1}x_2x_3^{-1}x_4.
\]
Then
\begin{equation}\label{g4:eq:gamma}
 \ga_0\sim w_\partial(E_1,E_2,E_3,E_4),\qquad
 \ga_1\sim w_\partial(F_1,F_2,F_3,F_4).
\end{equation}
Each supporting ribbon subgraph has two vertices, five edges and one
boundary component, so its neighborhood has genus $2$. This proves the
separating type of $\ga_i$ and hence of all three curves in \eqref{g4:eq:cycles}.

The induced automorphisms are
\begin{equation}\label{g4:eq:actions}
\begin{array}{ll}
 U(E_j)=E_1E_{j+1}^{-1},&U(F_j)=E_jF_j,\\
 V(E_j)=E_jF_{j+1}F_1^{-1},&V(F_j)=F_1F_{j+1}^{-1}.
\end{array}
\end{equation}
They fix $r$, satisfy the braid relation, and obey
$(UV)^{15}=\Inn(r)$, the action of the positive boundary Dehn twist
in \eqref{g4:eq:center}.

\begin{figure}[!tp]
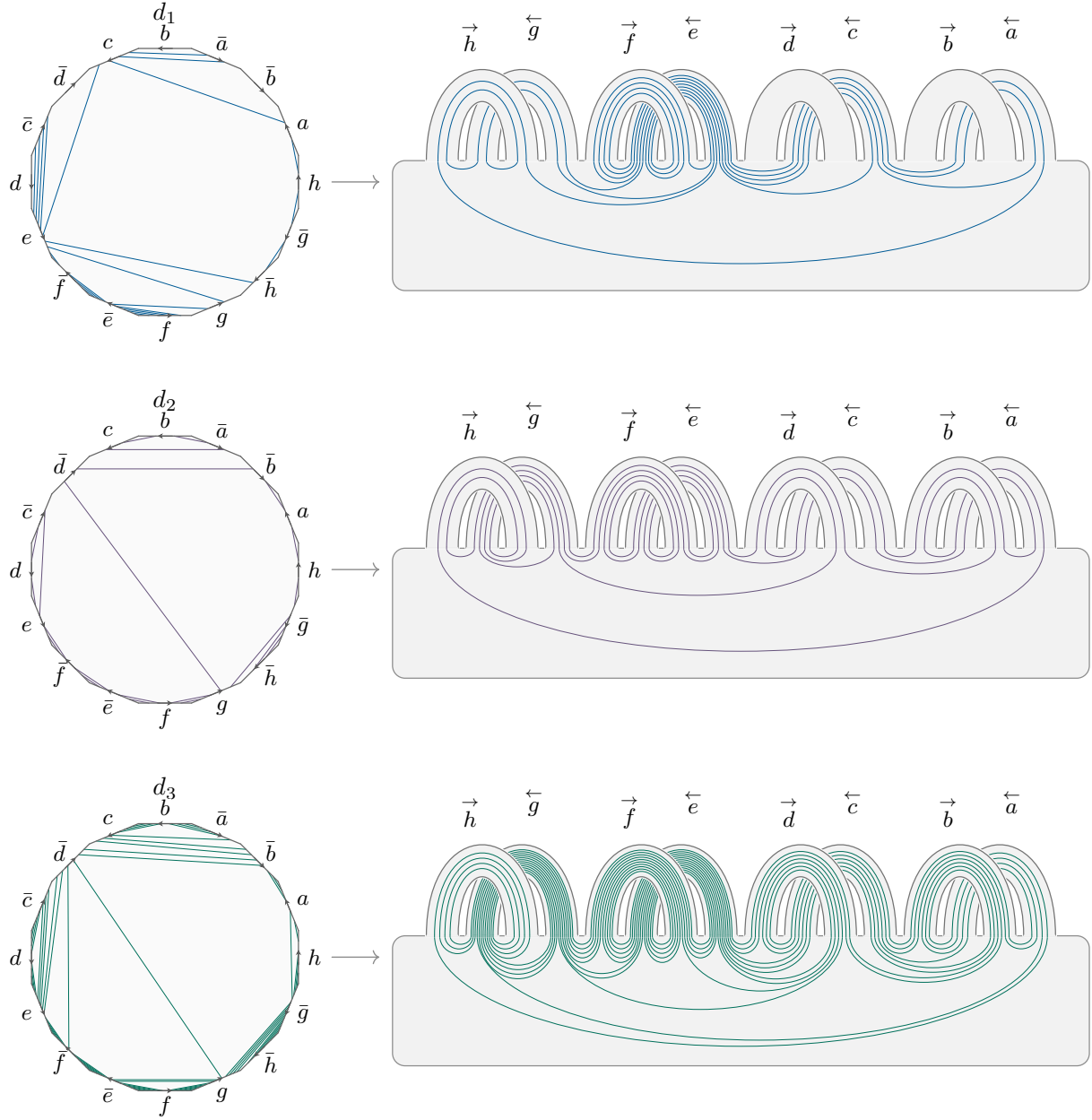

\centering
\resizebox{\textwidth}{!}{%
}
\par\vspace{6pt}
\caption{The three vanishing cycles on $\Sigma_4^1$ in a Kirby diagram
friendly style. In each row, the polygonal vertex disk with paired band
ends is on the left and the corresponding disk-and-band model on the right.
The side arrows specify the pairings, with strand endpoints matched in
order along these arrows. Arrows over the band labels indicate positive
traversal.}
\label{fig:genus4cycles}
\end{figure}

Figure~\ref{fig:genus4cycles} places the vanishing cycles on separate
copies of the same fiber. Each polygon is the vertex disk of the
standard $8$--loop bouquet, with the bands suppressed. Its sides
label the band ends, in the counterclockwise order
\[
 (a,\bar b,\bar a,b,\ c,\bar d,\bar c,d,\
  e,\bar f,\bar e,f,\ g,\bar h,\bar g,h).
\]
The same order runs from right to left along the top of the disk in
the band model. Continuing through a band from $x$ to $\bar x$ records $x$;
the reverse continuation records $x^{-1}$. This is the vertex order,
not the boundary word, which is $[a,b][c,d][e,f][g,h]$.
When using the side-paired polygon alone, truncate its vertices to
leave the boundary unfilled. Reading the curves gives the words in
Appendix~\ref{g4:app:words}, up to conjugacy and inversion.

\subsection{Derivation of the factorization}\label{g4:sec:proof}

\subsubsection*{The geometric meridians and their lifts}
We first determine the monodromies about $x=0$ and $x=1$.
For the family \eqref{g4:eq:catalogfields}, use
\begin{equation}\label{g4:eq:familycoordinates}
 t=-\frac1{y-x},\qquad \eta=\nu(x)\frac{z}{y-x},\qquad
 \nu(x)^5=\frac1{x^{k+1}(1-x)}.
\end{equation}
Substitution gives
\begin{equation}\label{g4:eq:movingbranches}
 \eta^5=t\left(t-\frac1x\right)\left(t-\frac1{x-1}\right).
\end{equation}
Choose the positive real fifth root at $x_*=1/2$. The finite branch
points are then $-2,0,2$, and \eqref{g4:eq:movingbranches} is the
fiber \eqref{g4:eq:realcurve}.

Let $p,q$ be counterclockwise meridians about $0,1$, with access
paths along the real interval from $x_*$. Write $u,v$ for the positive
half twists on $[-2,0]$ and $[0,2]$, and let
$\Delta^2=(uv)^3$ be the positive full twist of the three points.
Along $p$, the point $1/x$ travels clockwise around the other two;
the latter have no relative winding. The resulting pure braid is
$\Delta^{-2}u^2$. One can see the sign and the factor $u^2$ by
rotating all three points once clockwise and then undoing the
clockwise rotation of the left pair. Along $q$, the point $1/(x-1)$
travels clockwise around the other two, giving $\Delta^{-2}v^2$.
These descriptions use a disk large enough to contain the trajectories
of the three points.

Let $\vartheta(t,\eta)=(t,\zeta\eta)$, where
$\zeta=e^{2\pi i/5}$. The lifts of $u,v,\Delta^2$ are $U,V,Z$.
Rotating the three points once counterclockwise multiplies the lifted
coordinate by $\zeta^3$, so on the closed fiber
\begin{equation}\label{g4:eq:deckconvention}
 Z=\vartheta^3,\qquad \vartheta=Z^2\quad\text{in }\Gamma_4.
\end{equation}
The continuation of $\nu$ along $p$ multiplies it by $\zeta^{-(k+1)}$, and along $q$ by
$\zeta^{-1}$. On returning to the original $z$ coordinate, the
inverse factors occur. Hence the geometric monodromies are
\begin{equation}\label{g4:eq:geometricPQ}
 \rho_k(p)=\vartheta^{k+1}Z^{-1}U^2=Z^{2k+1}U^2,
 \qquad
 \rho_k(q)=\vartheta Z^{-1}V^2=ZV^2.
\end{equation}
In particular, for $k=0$ they are $P=ZU^2$ and $Q=ZV^2$.
We choose these words as their boundary-fixing lifts. Changing such
a lift by a power of $t_\partial$ does not change the commutators.

\subsubsection*{The relation on the genus-2 base}
We use the following free group identity. In the
free group $F(p,q)$ put
\[
 c=pqp^{-4},\quad e=p^3qp^{-1},\quad f=p^4qp^{-2},\quad
 k=q^2p^{-1},\quad h=ck.
\]
Direct cancellation gives
\begin{equation}\label{g4:eq:free}
 (h^{-1}q^5h)(h^{-1}(pq)^{-5}h)(k^{-1}p^5k)
 [e^{-1},f][c^{-1},k^{-1}]=1.
\end{equation}
This is the oriented genus-2 relation for the $5$--sheeted cover of the
$3$--punctured sphere defined by
\[
 p\longmapsto1,\qquad q\longmapsto3\quad\text{in }\Z_5.
\]
Indeed, a Schreier basis is
\[
 p^5,\quad b=qp^{-3},\quad c=pqp^{-4},\quad d=p^2q,
 \quad e=p^3qp^{-1},\quad f=p^4qp^{-2}.
\]
Here $k=be$, $q^5=becfd=kcfd$, and $(pq)^5=cp^5bfed$.
The four handle words in \eqref{g4:eq:free}, together with the first and third
peripheral words, form a free basis: $e,f,c,k$ recover $b$, the third
peripheral word recovers $p^5$, and $q^5$ recovers $d$.
The peripheral words are conjugates of the positive loops around the three punctures.
Thus the relation specifies the orientation and the identification of the
punctured base.

Now substitute
\begin{equation}\label{g4:eq:PQ}
 P=ZU^2,\qquad Q=ZV^2
\end{equation}
for $p,q$, and denote the corresponding uppercase words by $C,E,F,K,H$.
These are the lifts just obtained from \eqref{g4:eq:geometricPQ}.
In $\Gamma_4^1$ one has
\begin{equation}\label{g4:eq:powers}
 P^5=Z^5t_{\ga_0},\qquad Q^5=Z^5t_{\ga_1},\qquad
 (PQ)^{-5}=Z^{-15}t_{V^{-1}(\ga_0)}.
\end{equation}
For the last equality, use the braid identity
\[
 U^2V^2(V^{-1}U^2V)=Z
\]
and \eqref{g4:eq:basic}. In the closed surface the three central corrections
in \eqref{g4:eq:powers} vanish, leaving the three positive local monodromies.

Before changing the identification, the handles are
\[
 (E^{-1},F,C^{-1},K^{-1})
 =(PQ^{-1}P^{-3},\ P^4QP^{-2},\ P^4Q^{-1}P^{-1},\ PQ^{-2}),
\]
and the vanishing cycles are
\[
 H^{-1}(\ga_1),\quad H^{-1}V^{-1}(\ga_0),\quad K^{-1}(\ga_0).
\]
Conjugate the entire representation by $H$. The first two curves become
$d_1,d_2$, while the third becomes
\[
 HK^{-1}(\ga_0)=C(\ga_0)=U^2V^2(\ga_0)=d_3.
\]
Here $P$ fixes $\ga_0$, and $Z$ fixes both supporting boundaries and
commutes with $U,V$. Moreover, $H=Z^{-1}D$, so this conjugation gives exactly
the four handle words in \eqref{g4:eq:handles}.
After the substitution and conjugation, the central contribution to
\eqref{g4:eq:free} is $Z^{5-15+5}=Z^{-5}$. Moving it to the other side
gives \eqref{g4:eq:mainboundary}, with the three Dehn twists followed by
the two commutators.

\subsubsection*{The original three fibrations}
Identify the fibers over $x_*=1/2$ by $z_k=x_*^{k/5}z_0$,
using the positive real root. The deck transformation is
$\vartheta(y,z)=(y,\zeta z)$, with the convention
\eqref{g4:eq:deckconvention}.
The fifth root of $x$ defines an unramified cover of $G$,
since its zero and pole have order $5$. Its character $\chi$ is
the $p$-exponent sum, so the handle words in \eqref{g4:eq:free} give
\begin{equation}\label{g4:eq:character}
 (\chi(a_1),\chi(b_1),\chi(a_2),\chi(b_2))=(-2,2,3,1)\pmod5.
\end{equation}
Continuing $z_k=x^{k/5}z_0$ around a loop yields
$\rho_k(\ell)=\vartheta^{k\chi(\ell)}\rho_0(\ell)$.
Thus the four corrections are $(Z^k,Z^{-k},Z^k,Z^{2k})$,
as in \eqref{g4:eq:cataloghandles}; the vanishing cycles are unchanged.
For arbitrary $\lambda$, the identity
\[
 [Z^{r_j}A_j,Z^{s_j}B_j]=[A_j,B_j],\qquad j=1,2,
\]
holds already in $\Gamma_4^1$, proving the remaining relations
of Proposition~\ref{g4:prop:main}.

\subsection{The three covers, their sections and real structures}\label{g4:sec:geometry}\label{g4:sec:twists}

For $k=0$, \eqref{g4:eq:catalogfields} is the construction in
\cite[Definition~4.1]{CD} with $n=5$,
$(m_0,m_1,m_2,m_3)=(1,1,1,2)$ and $(n_0,n_1,n_2)=(1,1,3)$.
It has three singular fibers, each consisting of two genus-2 curves
meeting once. The unramified twists are locally isomorphic over
$G$, so this description holds for every $k$.

Write $H_i,V_i$ for the proper transforms of $y=i,x=i$, respectively,
$E_i$ for the exceptional sphere over $(i,i)$, and $\Delta$ for the
proper transform of the diagonal, where $i=0,1,\infty$. In coordinates
corresponding to the fifth roots $z,w$, the changes of sheet around these curves are
\begin{equation}\label{g4:eq:inertia}
\begin{array}{c|cccccccccc}
 &H_0&H_1&H_\infty&\Delta&V_0&V_1&V_\infty&E_0&E_1&E_\infty\\ \hline
 z&1&1&1&2&k&0&-k-2&k+3&3&1-k\\
 w&0&0&0&0&1&3&1&1&3&1
\end{array}\pmod5.
\end{equation}
The fifteen crossings are $V_i\cap H_j$ for $i\ne j$, and the three
intersections of each $E_i$ with $V_i,H_i,\Delta$. Each pair of
sheet-change vectors at a crossing is independent, for every $k$.

We compute the invariants from the branched cover. Let
$\mathcal D\subset Y$ be the union of the ten branch curves. Each is
a sphere of square $-1$ and meets three others. The complement has
Euler characteristic $2$, each branch sphere with its intersection
points removed has Euler characteristic $-1$, and there are fifteen
crossings. The numbers of sheets over these three parts are $25$, $5$
and $1$, respectively. Hence
\[
 \eu(X_k)=25\cdot2-10\cdot5+15=15.
\]
At a crossing, the two sheet changes are independent, so the local
cover is $(u,v)\mapsto(u^5,v^5)$ and the covering surface is smooth.
The branched-cover signature formula gives
\[
 \sigma(X_k)=25\sigma(Y)
 -\frac{25}{3}\left(1-\frac1{5^2}\right)
       \sum_{D\subset\mathcal D}D^2
 =-75+80=5.
\]
Here the $G$-signature theorem \cite{AS} gives the correction
from the branch spheres; the summed contributions at crossings vanish.
Thus $(\eu(X_k),\sigma(X_k))=(15,5)$, in agreement with
\cite{BC,CD}.

\subsubsection*{Identification with the literature}
Blow down the four disjoint spheres $(V_\infty,H_\infty,E_0,E_1)$
to obtain $\mathbb{CP}^2$. The remaining six branch curves become the
six lines joining the four resulting points. In the convention of \cite[Section 2]{BC}, their ordered
six-tuple of sheet-change vectors is
\begin{equation}\label{g4:eq:BCtuple}
 \mathcal U(k)=\bigl((1-k,1),(1,0),(1,0),(2,0),(0,3),(k,1)\bigr).
\end{equation}
Here the $u$-entries correspond to $(E_\infty,H_0,H_1)$ and the
$v$-entries to $(\Delta,V_1,V_0)$.

Let $s_{12},s_{23}$ interchange the indicated indices simultaneously
in the $u$ and $v$ triples.
The other two transformations from \cite[Section 3]{BC} are
\[
\begin{aligned}
 s_{01}(u;v)&=(u_1,u_3+v_1+v_2,u_2+v_1+v_3;
                    u_1+u_2+u_3,v_2,v_3),\\
 s_{04}(u;v)&=(u_1,u_2,u_3;
                    u_1+v_2+v_3,u_2+v_1+v_3,u_3+v_1+v_2).
\end{aligned}
\]
All entries lie in $(\Z_5)^2$. Acting on each entry by the matrices
\[
 M_0=\begin{pmatrix}2&0\\1&2\end{pmatrix},\qquad
 M_1=I,\qquad M_4=\begin{pmatrix}2&3\\1&1\end{pmatrix},
\]
one obtains
\[
\begin{aligned}
 M_0s_{12}s_{01}s_{04}s_{23}s_{12}s_{01}s_{23}\mathcal U(0)&=\mathcal U_2,\\
 s_{12}s_{23}s_{12}\mathcal U(1)&=\mathcal U_4,\\
 M_4s_{23}s_{12}s_{01}s_{04}s_{12}s_{01}\mathcal U(4)&=\mathcal U_1.
\end{aligned}
\]
Here the rightmost transformation acts first, and $\mathcal U_j$ are
the representatives in \cite{BC}. By \cite[Theorem 3.3]{BC}, these are precisely the three
types with positive first Betti number in that construction; the fourth surface $S_3$ has
first Betti number zero.

\subsubsection*{Four disjoint sections}
The inverse images of $H_0,H_1,H_\infty,\Delta$ are four
disjoint genus-2 sections, denoted $S_0,S_1,S_\infty,S_\Delta$.
Indeed the covering restricted to each such branch curve has degree
five and is identified with $G\to\PP^1_x$. If $\pi:X_k\to Y$ is
the covering and $D$ is one of these four branch curves, then
$\pi^*D=5S_D$. Since $D^2=-1$,
\[
 25S_D^2=(\pi^*D)^2=25D^2=-25.
\]
Thus all four section squares are $-1$. Their distribution between
the two genus-2 components of the singular fibers is independent
of $k$:
\[
\begin{array}{c|c}
 x=0&(S_0,S_\Delta)\mid(S_1,S_\infty)\\
 x=1&(S_1,S_\Delta)\mid(S_0,S_\infty)\\
 x=\infty&(S_\infty,S_\Delta)\mid(S_0,S_1).
\end{array}
\]
Consequently, each closed monodromy representation has a compatible
four-boundary lift with right-hand side
$t_{\delta_0}t_{\delta_1}t_{\delta_\infty}t_{\delta_\Delta}$.

\subsubsection*{Complex conjugation}
The equations and blowup centers are real, so conjugation extends
to an antiholomorphic involution of $X_k$ covering conjugation
$\tau_0$ on $G$. The real coordinate change
\eqref{g4:eq:familycoordinates} at $x=1/2$ identifies its fiber
involution with
\begin{equation}\label{g4:eq:kappa}
 \kappa(t,\eta)=(\ov t,\ov\eta).
\end{equation}
Its fixed locus is one circle: each real $t$ has one real fifth
root, and there is one point over infinity.

In the ribbon-graph basis, label the real lift of the first interval by $e_0$.
Conjugation then acts on edges as
\begin{equation}\label{g4:eq:reflectionedges}
 e_j\longmapsto e_{-j},\qquad f_j\longmapsto f_{1-j}.
\end{equation}
These permutations reverse all three vertex orders. They define the
following involution $c_0$ in the vertex-based free group:
\[
 c_0(E_j)=E_{-j},\qquad c_0(F_j)=F_{1-j}F_1^{-1}.
\]
To use the boundary access convention of \eqref{g4:eq:r}, put
\begin{equation}\label{g4:eq:kappaaction}
 s=F_4^{-1}E_3^{-1}E_2F_2F_1^{-1},\qquad
 \kappa_*(v)=s\,c_0(v)\,s^{-1}.
\end{equation}
One checks by free reduction that
\begin{equation}\label{g4:eq:realidentities}
\begin{gathered}
 s\,c_0(s)=1,\qquad \kappa_*(r)=r^{-1},\qquad \kappa_*^2=1,\\
 \kappa U\kappa=U^{-1},\qquad \kappa V\kappa=V^{-1},\qquad
 \kappa Z\kappa=Z^{-1}.
\end{gathered}
\end{equation}
Here $\kappa$ fixes the basepoint and preserves the boundary
setwise. These identities are in the extended mapping class group.
Both $\ga_0$ and $\ga_1$ are preserved as unoriented curves. Thus
\[
 \kappa(d_1)=d_1,\qquad \kappa(d_2)=V(\ga_0),\qquad
 \kappa(d_3)=U^{-2}V^{-2}(\ga_0).
\]
In particular, $\kappa P\kappa=P^{-1}$ and $\kappa Q\kappa=Q^{-1}$,
so $P\kappa$ and $Q\kappa$ are orientation-reversing involutions.

For loops based at a real regular value, naturality gives
$\rho(\tau_{0*}\ell)=\kappa\rho(\ell)\kappa^{-1}$.
The change of fiber identification by $H$ in Section~\ref{g4:sec:proof}
conjugates $\kappa$ by the same map.

\subsubsection*{Twisting by characters of the base}

The deck map $Z$ extends to a holomorphic automorphism of $X_0$,
fixing the four sections pointwise. For
$\lambda\in H^1(G;\Z_5)$, define
\begin{equation}\label{g4:eq:twistrepresentation}
 \rho_\lambda(\gamma)
 =Z^{\lambda(\bar\gamma)}\rho_0(\gamma),
\end{equation}
where $\bar\gamma$ is the image of $\gamma$ in the closed base.
The peripheral monodromies are unchanged. On the handle generators,
this is exactly \eqref{g4:eq:twisthandles}.

For $\lambda\ne0$, let $\widetilde{G}_\lambda\to G$
be the unramified cyclic cover defined by $\lambda$, with deck
generator $d$. Pull back $X_0$ to this cover and take the quotient
\begin{equation}\label{g4:eq:twistquotient}
 X_\lambda=
 (X_0\times_{G}\widetilde{G}_\lambda)
 /\langle(Z^{-1},d)\rangle.
\end{equation}
The free holomorphic diagonal action gives the monodromy
\eqref{g4:eq:twistrepresentation}, without changing the local
fibration near any fiber. The pullback is a common unramified
$5$--fold cover of $X_0$ and $X_\lambda$. Thus $X_\lambda$ is a
compact ball quotient with $(\eu,\sigma)=(15,5)$.
For $\lambda=0$ we retain $X_0$.

The four sections descend, since $Z$ fixes them pointwise.
Their squares change from $-1$ to $-5$ under pullback and back
to $-1$ in the quotient. They remain disjoint and give a
four-boundary lift with the same boundary multitwist as above.
The explicit words describe the
one-boundary lift.

This is the usual torsion twist of a cyclic cover, described in
\cite[Remark~1.7]{BC}: one tensors the defining fifth root by the
pullback of a line bundle in $\operatorname{Pic}(G)[5]$.
Our classification below concerns these twists, not all ball
quotients with the same invariants. The original abelian covers
and their fibrations are those of \cite{BC} and
\cite[Propositions~4.2--4.3 and Remark~4.1]{CD}.

\subsection{Homology and classification}\label{g4:sec:homology}\label{g4:sec:classification}

We first calculate the homology of these genus--$4$ fibrations.

\begin{proposition}\label{g4:prop:homology}
For every $\lambda\in H^1(G;\Z_5)$,
\begin{equation}\label{g4:eq:catalogH1}
 H_1(X_\lambda;\ZZ)=\ZZ^4\oplus(\Z_5)^2.
\end{equation}
The remaining integral homology groups are
\[
 H_2(X_\lambda;\Z)=\Z^{21}\oplus(\Z_5)^2,\qquad
 H_3(X_\lambda;\Z)=\Z^4,
\]
and $H_0(X_\lambda;\Z)=H_4(X_\lambda;\Z)=\Z$.
The total space has $(\eu,\sigma)=(15,5)$ and
$(b_2^+,b_2^-)=(13,8)$, and the intersection form on the free part of
$H_2(X_\lambda;\Z)$ is
\[
 Q_{X_\lambda}\cong13\langle1\rangle\oplus8\langle-1\rangle.
\]
\end{proposition}

\begin{proof}
Put $R=\Z[t]/(t^4-t^3+t^2-t+1)$. On each of the two blocks generated
by the classes $E_j$ and $F_j$, let $t$ act by
$E_j\mapsto E_1-E_{j+1}$ and $F_j\mapsto F_1-F_{j+1}$, respectively.
The vectors $E_1,tE_1,t^2E_1,t^3E_1$ form an integral basis of the
first block, and similarly for $F_1$. Thus
$H_1(\Sigma_4;\Z)=R^2$. Abelianizing \eqref{g4:eq:actions} gives
\begin{equation}\label{g4:eq:Ractions}
 U_* =\begin{pmatrix}t&1\\0&1\end{pmatrix},\qquad
 V_* =\begin{pmatrix}1&0\\-t&t\end{pmatrix},\qquad
 Z_*=-t^3 I.
\end{equation}
Modulo $t+1$, the defining polynomial becomes $5$. Hence
\[
 R/(t+1)=\F_5,\qquad U_*^2=V_*^2=Z_*=I
 \quad\text{on }(R/(t+1))^2.
\]
Every handle word in \eqref{g4:eq:twisthandles} acts trivially on
this quotient. The fiber coinvariants therefore surject onto
$(\Z_5)^2$.

For the reverse bound, the common conjugation by $D_*$ can be omitted.
The determinants over $R$ of $Z_*^rA_{1*}-I$, reduced to degree at
most three, are, for $r=0,1,2,3,4$, respectively,
\[
 \begin{gathered}
 3t^3+t-1,\qquad 3t^2+t+3,\qquad -4t^3+t^2-t+4,\\
 -t^3-2t^2-t,\qquad 2t^3-2t^2-1.
 \end{gathered}
\]
Each has norm $25$ from $R$ to $\Z$, obtained by taking the determinant
of multiplication on the basis $1,t,t^2,t^3$. Thus
$\det_{\Z}(Z_*^rA_{1*}-I)=25$ for every $r$. The order of the fiber
coinvariants divides $25$, so the preceding surjection is an isomorphism.
The separating vanishing cycles add no homology relations, and a
section splits off $H_1(G;\Z)=\Z^4$.

Since $(\eu,\sigma)=(15,5)$, we have $b_2=21$ and
$(b_2^+,b_2^-)=(13,8)$. Poincar\'e duality and the universal
coefficient theorem give the remaining homology groups. The
intersection form is unimodular and indefinite, and a section of
square $-1$ shows that it is odd. The classification of indefinite
unimodular forms then gives the stated diagonal form.
\end{proof}

Thus neither the homology groups nor the intersection forms
distinguish these fibrations.

Complex conjugation gives diffeomorphic total spaces for some
nonisomorphic Lefschetz fibrations in this family.

\begin{proposition}\label{g4:prop:classification}
The family $\{f_\lambda\}$ contains exactly 65 isomorphism classes
of Lefschetz fibrations. These are also its
holomorphic fibration classes and the biholomorphism classes of
its total spaces. There are exactly 45 isomorphism classes of
total-space fundamental groups, and 45 diffeomorphism
types of total spaces.
\end{proposition}

\begin{proof}
A generic fiber has automorphism group $\langle Z\rangle$.
Its cyclic pentagonal subgroup is unique outside Bring's curve
\cite[p.~407]{CI}, so every generic automorphism descends to the
sphere. The branch exponents $(1,1,1,2)$ force it to fix the
exponent-2 point: no nontrivial common multiplier preserves their
multiplicities. A generic four-point set has no nonidentity
M\"obius symmetry fixing that point. Thus the relative automorphisms
are the constant deck transformations, and two twists over the
identity of $G$ are isomorphic exactly when their
characters agree.

The full automorphism group of $G$ is cyclic of order $10$
\cite[Lemma~1]{SB}, generated by
\[
 d(x,w)=(x,\zeta_5w),\qquad
 h(x,w)=(1/x,-w/x),
\]
where $h$ is the hyperelliptic involution. To make its action on
characters linear, use $X_4$ as the origin of the family. Both maps
lift to $X_4$, commuting with $Z$:
\[
 \widehat d(x,w,y,z)=(x,\zeta_5w,y,z),\qquad
 \widehat h(x,w,y,z)=(1/x,-w/x,1/y,-z/(x^2y)).
\]
Substitution checks the equations and $\widehat h^2=1$; the maps
extend over the smooth minimal surfaces.

For $p,q$ from Section~\ref{g4:sec:proof}, set
$\ell_s=p^sqp^{-m_s}$, where $m_s\equiv s+3\pmod5$ and
$0\le m_s<5$. Their classes generate $H_1(G;\Z)$ with
sole relation $\sum\ell_s=0$. With $\lambda_4=4(1,-1,1,2)$,
the coordinates $v_s=(\lambda-\lambda_4)(\ell_s)$ identify the
character space, centered at $X_4$, with
\begin{equation}\label{g4:eq:charspace}
 \mathcal V=\{(v_0,\ldots,v_4)\in\F_5^5:\ \textstyle\sum v_s=0\}.
\end{equation}
The handle words in \eqref{g4:eq:free} give the explicit change of
coordinates
\begin{equation}\label{g4:eq:characterdictionary}
 v=(r_1-s_2-1,\,-r_2-1,\,r_2+s_2-s_1-1,\,-r_1-1,\,s_1-1).
\end{equation}
The deck map cyclically permutes the $\ell_s$, and $h$ acts by
$-I$ on homology. Their actions on $\mathcal V$ are therefore
\[
 T(v)_s=v_{s+1},\qquad N(v)=-v,
\]
with subscripts modulo $5$. All these base automorphisms lift at
$X_4$, so the holomorphic fibration classes are exactly the orbits
of $\langle T,N\rangle$. The identity fixes 625 characters, each
nonidentity power of $T$ fixes the five constant characters, and
each $-T^j$ fixes only zero. Burnside's lemma gives
\begin{equation}\label{g4:eq:holcount}
 \frac{625+4\cdot5+5\cdot1}{10}=65.
\end{equation}

Coordinatewise conjugation on $X_4$ inverts $Z$. On base homology
it sends $\ell_s$ to $-\ell_{2-s}$; the two signs cancel on the
twisting characters. Thus conjugation of the family adds the action
$J(v)_s=v_{2-s}$. The group $\langle T,N,J\rangle$ has order 20.
Each of its ten additional elements $\pm T^jJ$ fixes 25 characters:
a reflection on five coordinates has two exchanged pairs and one
fixed coordinate, and both eigenspaces in $\mathcal V$ have
dimension two. The number of classes under holomorphic or
antiholomorphic isomorphisms is therefore
\begin{equation}\label{g4:eq:extendedcount}
 \frac{625+4\cdot5+5\cdot1+10\cdot25}{20}=45.
\end{equation}

For the smooth classification, Proposition~\ref{g4:prop:homology}
gives
$q(X_\lambda)=2=g(G)$. All holomorphic one-forms on
$X_\lambda$ come from the base, and their common kernel on the
regular locus is the fiber direction. Hence every biholomorphism
or antibiholomorphism of total spaces preserves the fibrations.
Mostow rigidity \cite{Mostow} realizes an isomorphism of their
fundamental groups by a holomorphic or antiholomorphic isometry.
Both preserve the real $4$--dimensional orientation. This proves
the count of 45 group types and diffeomorphism types.

Finally, suppose a fibration isomorphism $F:X_\lambda\to X_\mu$
covers $u\colon G\to G$. Its Mostow isometry $I$ covers a holomorphic
or antiholomorphic map $v$. Since $F\simeq I$, composing with a
section gives $u\simeq v$.
Since $\deg u=1$, the isometry cannot be antiholomorphic. Thus
fibration isomorphism gives exactly the same 65
classes as holomorphic fibration isomorphism.
\end{proof}

The constant characters are the original $k$-family, since
$v_s=4(k-4)$, and form the three orbits
$\{4\},\{0,3\},\{1,2\}$. These are precisely the members for which
the order-$5$ base symmetry lifts. The other 620 characters give
62 fibration classes. Of the 65 classes, 25 are fixed by
conjugation and the remaining 40 form 20 pairs with diffeomorphic
total spaces but nonisomorphic Lefschetz fibrations.

\section{Real quotients}\label{sec:realquotients}\label{sec:orbitgroup}

We determine the real quotient of $X$ up to homeomorphism and
compute the fundamental groups and ordinary intersection forms of the
quotients of $X_0,X_1,X_4$.
Each quotient map is a double cover
branched along the fixed surface, with normal-disk model $z\mapsto z^2$.

We first recall how to compute the fundamental group of an orbit space.
Let $S$ be a connected smooth manifold with fundamental group
$\Gamma$, and let $\kappa$ be a smooth involution with a fixed point.
Choose a lift of this point to the universal cover. The lift of
$\kappa$ fixing it is an involution and induces
$r\in\operatorname{Aut}(\Gamma)$. Set
\begin{equation}\label{eq:coequalizer}
\mathcal{Q}_r=\Gamma/\langle\!\langle g^{-1}r(g):g\in\Gamma\rangle\!\rangle.
\end{equation}
Armstrong's theorem identifies $\pi_1(S/\kappa)$ with
$\Gamma\rtimes_r \Z_2$ modulo the normal subgroup generated by point
stabilizers \cite{Armstrong}. Killing the chosen lift first gives
$\mathcal{Q}_r$. If $\operatorname{Fix}(\kappa)$ is connected, lifting paths
in the fixed locus shows that all lifts of $\kappa$ with fixed points
are conjugate by $\Gamma$. In that case,
\begin{equation}\label{eq:connectedreal}
\pi_1(S/\kappa)\cong \mathcal{Q}_r.
\end{equation}
If $S$ is a closed almost complex $4$--manifold and $\kappa$
is anti-complex, its fixed surface $F$ has normal Euler number
$-\chi(F)$. The branched-cover and equivariant signature formulas
\cite{AS} give
\begin{equation}\label{eq:realinvariants}
\eu(S/\kappa)=\frac{\eu(S)+\chi(F)}2,\qquad
\sigma(S/\kappa)=\frac{\sigma(S)-\chi(F)}2.
\end{equation}
These formulas also hold when $F$ is disconnected. In particular,
they apply to an anti-symplectic involution with an anti-invariant
compatible almost complex structure.

\subsection{The genus-19 quotient}
\label{g19:sec:realquotient}\label{g19:sec:group}

We first construct the real structure on the fibration in cyclic-cover coordinates,
using the based involution of Lemma~\ref{lem:genus19fibercomparison},
whose fiber mapping class is $\hat C$. We then transport it to
$\pi:X\to E$ by Proposition~\ref{prop:models}.

\subsubsection*{Extending the involution}
Place the three periodic values on a sphere so that a reflection
exchanges the $\hat P$- and $\hat R$-values and fixes the $\hat Q$-value
and the nodal value. Use reflected access paths, with peripheral
relation $dpqr=1$. The identities in Corollary~\ref{thm:chain-model} give
\begin{equation}\label{g19:eq:realgluing}
 \hat P(\hat Q\hat C)\hat P^{-1}=t_{\hat c}^{-1}\hat C.
\end{equation}
This is the gluing relation for the real structures on the two sides
of the real axis. To see this, remove small disks around the four
values. The upper half is an annulus, with monodromy $\hat P$ around
its inner boundary. Along its outer boundary the real fiber
involution starts as $\hat C$; passing the real periodic disk gives
$\hat Q\hat C$, whereas passing the nodal disk in the opposite direction
gives $t_{\hat c}^{-1}\hat C$. Equation~\eqref{g19:eq:realgluing} identifies
these ends by transport around the annulus. Isotopic real involutions
are conjugate by a diffeomorphism isotopic to the identity, and the
space of involutions in a fixed isotopy class is simply connected
in genus greater than one \cite[Section~4, Proposition~4.3]{Salepci}. We may therefore
choose a representative of $\hat P$ giving this identification and
close the boundary isotopies. Doubling along the remaining real
boundary intervals gives the real bundle.

Fill the nonreal periodic disks by mutually reflected order-$3$
models. At the real periodic disk use the dihedral action
$\langle\hat Q,\hat C\rangle$ on $\Sigma_{19}\times D^2$; its real map
is $(\hat C,\mathrm{conj})$. At the node use the real Lefschetz model
specified by $(\hat C,\hat c)$ \cite[Section~3]{Salepci}. Reflection negates
the degree-three character, so it lifts to the cyclic branched base
cover. Above each periodic disk the pullback is
$\Sigma_{19}\times D^2$. The gluing relation
\eqref{g19:eq:realgluing} and these local models therefore give a real
involution on the fibration in cyclic-cover coordinates over the torus.

We now transport this involution by the fibration isomorphism of
Proposition~\ref{prop:models}, using the induced map of the base.
On the second configuration, the fiber reflection is $C'=QC$ and
the real vanishing curve is $c'=P^{-1}(c)$.
There is one real node, since reflection permutes its three lifts
by $s\mapsto-s$ in $\Z_3$. In Figure~\ref{g19:fig:torus}, $C'$ acts on
$c'=\{s=1/6\}$ by $r\mapsto-r-1/6$, with two fixed points. The real local model is therefore
$(z_1,z_2)\mapsto z_1^2-z_2^2$, whose real restriction is a saddle.
The map from $F=\Fix(\tau)$ to the real part of the base is a
submersion elsewhere, so
\begin{equation}\label{g19:eq:realEuler}
 \chi(F)=-1.
\end{equation}
Components without critical points contribute zero to this equality.

\subsubsection*{The quotient group}
We use the group $\mathcal E$ and its index-$3$ subgroup
$\Gamma\cong\pi_1(X)$ from \eqref{g19:eq:orbifoldgroup}.

The based fiber involution $\kappa$ satisfies
\begin{equation}\label{g19:eq:basedreal}
 \kappa^2=1,\qquad
 \kappa\mathsf P\kappa\mathsf R=\Inn(u),\qquad
 \kappa\mathsf Q\kappa\mathsf Q=\Inn(v),\qquad
 \kappa\mathsf t_c\kappa\mathsf t_c=1,
\end{equation}
where $\mathsf t_c$ is the based automorphism representing the positive
Dehn twist about $c_x$.
With the same access paths, the geometric involution induces $\varrho $ on
$\mathcal E$, where
\begin{equation}\label{g19:eq:realextension}
 \varrho |_\Pi=\kappa,\qquad
 \varrho (p)=u r^{-1},\qquad \varrho (q)=v q^{-1},\qquad
 \varrho (r)=p^{-1}\kappa(u).
\end{equation}
The corrections in \eqref{g19:eq:realextension} are determined by
\eqref{g19:eq:basedreal}: before filling the singular values the
fiber group injects and has trivial center. The based identities
give $\varrho ^2=1$ and $\chi \varrho =-\chi$, so $\varrho $ restricts to $\Gamma$.

We use the transversal $1,r,r^2$. In the quotient
$\mathcal{Q}_\varrho =\Gamma/\langle\!\langle a^{-1}\varrho (a):a\in\Gamma\rangle\!\rangle$,
the 114 fiber generators $r^sx_ir^{-s}$ are trivial, and the seven
base generators reduce to
\begin{equation}\label{g19:eq:realreduction}
\setlength{\arraycolsep}{5pt}
\begin{array}{c|ccccccc}
 a&pr^{-1}&qr^{-1}&rpr^{-2}&rqr^{-2}&r^2p&r^2q&r^3\\ \hline
 a\text{ in }\mathcal{Q}_\varrho &z&z^{-1}&z^{-2}&z^{-1}&z&z^2&1
\end{array},\qquad z=r^2p.
\end{equation}
Each generator elimination uses a relator in which that generator
occurs once; see Appendix~\ref{g19:app:group}. The remaining
relators cancel freely. Conversely, these substitutions satisfy
all the original relators. Thus
\begin{equation}\label{g19:eq:realgroup}
 \mathcal{Q}_\varrho \cong\Z.
\end{equation}

\begin{proposition}\label{prop:genus19realquotient}
For the real involution $\tau$ above, the fixed surface is
$\Fix(\tau)\cong\# 3\RP$, and the quotient $Y=X/\tau$
satisfies
\[
 \pi_1(Y)\cong\Z,\qquad
 (\eu(Y),\sigma(Y))=(1,1),\qquad
 Y\approx \CP\#(S^1\times S^3).
\]
\end{proposition}

\begin{proof}

By the orbit-space calculation above, $\pi_1(Y)$ is a quotient of
$\mathcal{Q}_\varrho \cong\Z$. Invariant rational cohomology gives
\[
 b_1(Y)=\dim H_1(X;\Q)_\varrho 
 =\dim H_1(\mathcal{Q}_\varrho ;\Q)=1.
\]
Hence $\pi_1(Y)\cong\Z$.

Choose a fiber-positive closed form $\beta$, standard near the
critical points. Antisymmetrizing it to $(\beta-\tau^*\beta)/2$
and adding a large pullback of an anti-invariant base area form
makes $\tau$ anti-symplectic. Thus $F$ is Lagrangian with normal
Euler number $-\chi(F)$, and \eqref{eq:realinvariants} gives
\[
 \eu(Y)=\frac{3+\chi(F)}2=1,\qquad
 \sigma(Y)=\frac{1-\chi(F)}2=1.
\]
In particular $b_2(Y)=1$.

Over $\Lambda=\Z[t,t^{-1}]$, the module $\pi_2(Y)$ is stably
free of rank one and its equivariant intersection form is
nonsingular \cite[Corollaries 3.2 and 4.4]{HKT}.
Here $H^2(\Z;\Lambda)=H^3(\Z;\Lambda)=0$.
Taking determinants of a stable free isomorphism shows that this module is free. A
rank-one nonsingular Hermitian form is a self-conjugate unit of
$\Lambda$, hence $1$ or $-1$; the signature selects $1$.
The form is therefore odd. The quotient is smooth and has vanishing Kirby--Siebenmann
invariant. The classification for fundamental group $\Z$ gives
the claimed homeomorphism \cite[Theorem A]{HKT}.

To identify the fixed surface, let $F_0$ be one of its components
and lift $\tau$ to the universal cover, fixing a point over $F_0$.
By Proposition~\ref{prop:genus19group}, this cover is contractible.
Smith theory gives a connected mod 2 acyclic fixed surface
$\widetilde F_0$ covering $F_0$; see \cite[Section~3]{Edmonds}.
A mod 2 acyclic surface without boundary is a plane. Hence
$F_0$ is aspherical and its inclusion in $X$ is $\pi_1$-injective.
The group $\pi_1(X)\cong\pi_1(\CS)$ contains no
$\Z^2$ by Preissmann's theorem \cite{Preissmann}, excluding tori
and Klein bottles. Every component therefore has negative Euler
characteristic. Since $\chi(F)=-1$, we obtain $F\cong\# 3\RP$.
\end{proof}

\subsection{The three genus-4 quotients}
For the real structures $\kappa_k$ of Section~\ref{g4:sec:geometry},
we use the common based fiber involution in \eqref{g4:eq:kappaaction}.

\begin{proposition}\label{prop:real4}
For $k=0,1,4$, put $Y_k=X_k/\kappa_k$. Then
\[
X_k(\mathbb R)\cong\# 5\RP,\qquad
\pi_1(Y_k)\cong\Z\star\Z,
\]
and
\[
(\eu(Y_k),\sigma(Y_k))=(6,4),\qquad
Q_{Y_k}\cong6\langle1\rangle\oplus2\langle-1\rangle.
\]
Here $Q_{Y_k}$ is the ordinary integral intersection form.
\end{proposition}

\begin{proof}
Here $Y$ is again the blowup of $\PP^1\times\PP^1$ at the three
real points, as in Section~\ref{g4:sec:geometry}.
The odd-degree local models, including
$(u,v)\mapsto(u^5,v^5)$ at crossings, give one real point over
each real point of $Y$. Compactness makes the map on real loci
a homeomorphism. Thus
\[
X_k(\mathbb R)\cong T^2\# 3\RP\cong\# 5\RP.
\]

We compute the quotient fundamental group. Let $\Pi$ be the eight-generator
closed fiber group in the ribbon-graph basis. Set
\[
P_k=Z^{1+2k}U^2,\qquad Q=ZV^2.
\]
The orbifold total-space presentation has the fiber relator, the three
vanishing-cycle relations, the conjugation actions of $p,q$ by
$P_k,Q$, and $p^5=q^5=(pq)^5=1$. Its index-$5$ kernel under
$p\mapsto1$, $q\mapsto3$ is $\pi_1(X_k)$. The real automorphism
restricts to $\kappa$ on $\Pi$ and sends $p,q$ to their inverses.
We use based automorphisms of the closed fiber group.

We use the transversal $1,p,p^2,p^3,p^4$ and write
\[
 b_s=p^sqp^{-\overline{s+3}}\quad(0\leq s\leq4),
\]
where the bar denotes the representative in $\{0,1,2,3,4\}$.
After imposing $g=r(g)$, the eliminations in
Appendix~\ref{app:realreduction} kill every fiber generator and give,
for all three values of $k$,
\begin{equation}\label{eq:real4reduction}
 p^5=b_1=1,\qquad b_0=b_2^{-1},\quad b_3=b_4^{-1}.
\end{equation}
Each elimination uses a relator in which the generator occurs once.
The remaining relators cancel freely, and these substitutions
satisfy all the original relators. Hence the quotient is freely
generated by
\[
 b_2=p^2q,\qquad b_4=p^4qp^{-2},
\]
and $\mathcal{Q}_r=\langle b_2\rangle\star\langle b_4\rangle\cong\Z\star\Z$. Equation~\eqref{eq:connectedreal} proves the
fundamental-group claim.

The same rewriting with the fiber generators omitted presents
$\pi_1(G/\mathrm{conj})$ on $b_2,b_4$. Thus
$Y_k\to G/\mathrm{conj}$ induces this isomorphism.
The base quotient has one boundary component and Euler
characteristic $-1$.

Equation~\eqref{eq:realinvariants} gives $\eu(Y_k)=6$ and
$\sigma(Y_k)=4$. Since $b_1(Y_k)=2$, we obtain $b_2(Y_k)=8$ and
$(b_2^+,b_2^-)=(6,2)$. The integral form is unimodular and indefinite.
It is odd, since an even unimodular form has signature divisible by
eight. The classification of odd indefinite unimodular integral forms
gives the diagonal form.
\end{proof}

\section{Handle monodromies and finite covers}\label{sec:properties}

We study the Nielsen--Thurston types of handle monodromies, first
for the examples above and then for holomorphic Lefschetz fibrations
on compact ball quotients. We end with examples obtained by finite covers.

\subsection{Handle monodromies}\label{sec:pa}

We now analyze the handle monodromies in all these examples. 

\subsubsection*{Genus 19: homology of the fiber and a double cover}
We consider the handle maps $A=QR^{-1}$ and $B=R^{-1}Q$
in \eqref{eq:commutators}.

\begin{proposition}\label{prop:pa19}
The mapping classes $A,B$ are conjugate and pseudo-Anosov on
$\Sigma_{19}$.
\end{proposition}

\begin{proof}
The action of $A$ on $H_1(\Sigma_{19};\Q)$ does not by itself
exclude reducibility, so we also use an unramified double cover of
the fiber, as in \cite[Sections~3--4]{MalesteinSouto}.

Put $S=A_*$, $K=\ker(S^2+I)$ and
$L=K\cap H_1(\Sigma_{19};\Z)$. Write $\bar L$ for its reduction modulo
$2$. The calculations in Appendix~\ref{app:geometric-pa} give the
following facts. The space $K$ is an isotropic $2$--plane containing
all vectors with finite $S$-orbit, and every $S$-invariant rational $2$--plane
is contained in $K$. There is a nonzero $A$-invariant character
$\eta\in H^1(\Sigma_{19};\Z_2)$ with
\[
 \eta(L)=0,\qquad \eta\notin\operatorname{PD}(\bar L).
\]
Let $\widetilde\Sigma\to\Sigma_{19}$ be the double cover it defines,
of genus $37$, with deck involution $\jmath$. A lift $\widetilde A$ acts on
\[
 V^-=\ker(\jmath_*+I)\subset H_1(\widetilde\Sigma;\Q),
 \qquad \dim V^-=36,
\]
with irreducible characteristic polynomial $r(t)$ satisfying
$r(1)=-20$. Thus $r$ has a real root greater than one and no root of
unity. In particular, $A$ is not periodic and no nonzero vector of
$V^-$ has a finite $\widetilde A_*$-orbit.

Suppose $A$ preserves a nonempty essential multicurve $\mathcal D$.
If $c\in\mathcal D$ is nonseparating, then $[c]$ has finite $S$-orbit,
so $[c]\in L$ and $\eta(c)=0$. Hence $c$ has two lifts $c_1,c_2$.
The cover of $\Sigma_{19}\setminus c$ is connected: otherwise $\eta$
would be the mod-$2$ Poincar\'e dual of $[c]$, contrary to the choice
of $\eta$. It follows that $[c_1]-[c_2]\ne0$ in $V^-$. But
$\widetilde A$ permutes the lifts of $\mathcal D$, giving this class
a finite orbit, a contradiction.

Thus all components of $\mathcal D$ are separating. If $\eta$ is
nonzero on both sides of one of them, both sides lift to connected
surfaces joined along the two lifts of the curve. These lifts are
nonseparating and have opposite homology classes, again giving a
nonzero vector of $V^-$ with finite orbit. This too is impossible.
The dual graph of the cut surface is a tree, so $\eta$ is nonzero on
exactly one component $S_0$ of $\Sigma_{19}\setminus\mathcal D$.
This component is preserved by $A$.

Let $g_0\ge1$ be the genus of $S_0$. All curves over $\mathcal D$ now
separate $\widetilde\Sigma$. Capping the boundary of the connected
double cover of $S_0$ therefore gives an
$\widetilde A_*$-invariant symplectic summand of $V^-$ of dimension
$2g_0-2$. Since $\mathcal D$ is nonempty and essential, a component
outside $S_0$ has positive genus, and this summand is proper.
Irreducibility of $r$ forces $g_0=1$. But capping $S_0$ then gives
an $S$-invariant symplectic $2$--plane in
$H_1(\Sigma_{19};\Q)$, whereas every invariant $2$--plane lies in
the isotropic space $K$. This final contradiction proves that $A$
is pseudo-Anosov. So is $B=Q^{-1}AQ$. \end{proof}

\subsubsection*{Genus 4: a $4$--punctured-sphere calculation}
Modulo the central element $Z=(UV)^3$, the lifted braid subgroup acts
on the $4$--punctured sphere. Use the usual matrices
\[
\overline U=\begin{pmatrix}1&1\\0&1\end{pmatrix},\qquad
\overline V=\begin{pmatrix}1&0\\-1&1\end{pmatrix}
\]
in $\mathrm{PSL}(2,\Z)$. These come from the double cover branched at the four punctures;
$Z$ maps to $-I$ in $\mathrm{SL}(2,\Z)$.
See also Dashyan's use of this sphere monodromy in the related
Hirzebruch construction \cite[Section~4]{Dashyan}.

Substitution of \eqref{g4:eq:handles}, omitting the common
conjugation by $D$ and the central powers, gives absolute traces
$6,6,14,10$.

\begin{proposition}\label{prop:pa4}
For every $\lambda\in H^1(G;\Z_5)$, all four handle
monodromies in \eqref{g4:eq:twisthandles} are pseudo-Anosov on
$\Sigma_4$. The stretch factors of
$Z^{r_1}A_1,Z^{s_1}B_1,Z^{r_2}A_2,Z^{s_2}B_2$
are, respectively,
\[
 3+2\sqrt2,\quad 3+2\sqrt2,\quad 7+4\sqrt3,\quad 5+2\sqrt6.
\]
\end{proposition}

\begin{proof}
The four quotient matrices are hyperbolic, so their sphere maps
are pseudo-Anosov. Under the fully ramified $5$--fold cover,
each one-pronged singularity at a branch point becomes five-pronged.
The foliations therefore extend over the closed fiber, with the
same stretch factors $(T+\sqrt{T^2-4})/2$, where $T$ is the
absolute trace. Deck powers preserve these foliations and factors;
equivalently, $(Z^jA_i)^5=A_i^5$ and $(Z^jB_i)^5=B_i^5$.
\end{proof}

Thus the handle monodromies are pseudo-Anosov although the product
of their commutators is the inverse of a product of three separating Dehn twists.

\subsection{Monodromy restrictions on the BMY line}\label{sec:handle-obstruction}

Let $f:X\to\mathcal B$ be a Lefschetz fibration with connected closed
regular fiber $F$ of genus $g\geq2$ and base of positive genus.
Write $\mathcal B^\circ$ for the complement of the critical values,
$\mu:\pi_1(\mathcal B^\circ)\to\Mod(F)$ for the monodromy, and
$j:\pi_1(F)\to\pi_1(X)$ for the inclusion homomorphism. For
$\gamma\in\pi_1(\mathcal B^\circ)$, let $\bar\gamma$ denote its image
in $\pi_1(\mathcal B)$. A multitwist is a product of integer powers
of Dehn twists about disjoint essential curves; we allow either sign and
the empty product.

\begin{proposition}\label{prop:handle-obstruction}
Suppose $\pi_1(X)$ is torsion-free, contains no subgroup isomorphic
to $\Z^2$, and $j(\pi_1(F))\ne1$. If $\bar\gamma\ne1$, every
essential curve with finite orbit under $\mu(\gamma)$ is
nullhomotopic in $X$, and no nonzero power of $\mu(\gamma)$ is a
multitwist. In particular, $\mu(\gamma)$ has a pseudo-Anosov
component. If $j$ is injective, then $\mu(\gamma)$ is pseudo-Anosov.
\end{proposition}

\begin{proof}
Pulling back over $\gamma$ gives a mapping torus. If $\alpha$ is its
based monodromy and $t$ is the image in $\pi_1(X)$ of a loop going
once around the base, then
\[
 t\,j(w)\,t^{-1}=j(\alpha(w)),\qquad f_*(t)=\bar\gamma.
\]
Suppose the conjugacy class of $w\in\pi_1(F)$ is periodic. For some
$m>0$ and $u\in\pi_1(F)$ we have $\alpha^m(w)=uwu^{-1}$; for an
unoriented curve, square the return power if necessary. Thus
$s=j(u)^{-1}t^m$ commutes with $j(w)$ and projects to
$\bar\gamma^m$. If $j(w)\ne1$, these commuting elements generate
$\Z^2$: a relation $j(w)^a s^b=1$ projects to
$\bar\gamma^{mb}=1$, forcing $b=0$, and torsion-freeness then gives
$a=0$. This proves the first assertion, also for arbitrary periodic
conjugacy classes.

Suppose now that $\mu(\gamma)^m=\prod_i t_{c_i}^{n_i}$ with $m>0$.
Each $c_i$ is nullhomotopic in $X$ by the first assertion. A
multitwist acts trivially on $\pi_1(F)$ after the curves supporting its Dehn twists
are killed, up to a common inner automorphism. Adjusting $t^m$ by
an element of $j(\pi_1(F))$ therefore gives an element centralizing
this nontrivial subgroup and still projecting to $\bar\gamma^m$.
The same argument produces a forbidden $\Z^2$. Negative powers
are excluded by inversion. The Nielsen--Thurston classification \cite{FM}
now gives a pseudo-Anosov component, since a map without one has a
multitwist power. If $j$ is injective, the first assertion excludes
reducibility as well.
\end{proof}

The hypothesis $\bar\gamma\ne1$ excludes meridians of critical
values, whose monodromies are Dehn twists. Over a torus it also
excludes commutators of handle loops. Reducibility is still possible:
an invariant curve may be nullhomotopic in $X$.

\begin{corollary}\label{cor:all-handle-pa4}
For a holomorphic Lefschetz fibration on a compact complex ball
quotient, every monodromy $\mu(\gamma)$ with $\bar\gamma\ne1$ has
a pseudo-Anosov component, and no nonzero power is a multitwist.
If the fiber genus is $2\leq g\leq4$, then $\mu(\gamma)$ is
pseudo-Anosov. In these genera, every handle monodromy is therefore
pseudo-Anosov, for any choice of handle generators.
\end{corollary}

\begin{proof}
The fundamental group is torsion-free and contains no $\Z^2$
subgroup by negative curvature \cite{Preissmann}. The fiber image
is nontrivial: otherwise its inclusion would lift to a holomorphic
map from a compact curve to the ball, and hence be constant.
Proposition~\ref{prop:handle-obstruction} proves the first assertion
and excludes periodicity.

Suppose a positive power of $\mu(\gamma)$ preserves a nonempty
essential multicurve $\mathcal C$. Let $g_1,\ldots,g_r$ be the genera
obtained by cutting $F$ along $\mathcal C$ and capping the boundary
circles. We claim that
\begin{equation}\label{eq:handle-strict}
 \sum_{j=1}^r\max\{g_j-1,0\}>\frac23(g-1).
\end{equation}
Replace $\gamma$ by a positive power preserving every curve and
every component of the cut surface, with their orientations.
By Proposition~\ref{prop:handle-obstruction}, all curves of
$\mathcal C$ are nullhomotopic in $X$. Capping with nullhomotopies gives
maps $v_j:S_j\to X$ from the closed surfaces of genera $g_j$, with
\[
 [F]=\sum_j(v_j)_*[S_j].
\]
This equality follows by attaching $3$--dimensional $2$-handles
to $F\times[0,1]$ along $\mathcal C$ and extending the map using
the nullhomotopies. Embedded disks in $X$ are not needed.

Let $K\in H^2(X;\Z)$ be the canonical class. Toledo's inequality,
expressed in terms of $K$, gives \cite{Toledo,BIW}
\[
 |K\cdot v_*[S]|\leq3(q-1)
\]
for any continuous map $v:S\to X$ from a closed oriented surface of
genus $q\geq2$. Equality implies that the induced surface
representation is discrete and faithful and preserves a complex
geodesic on which it acts cocompactly.
Genus-$0$ and genus-$1$ components have zero pairing, since $X$ is
aspherical and its abelian subgroups are cyclic. Adjunction and
$F^2=0$ therefore give
\[
 2g-2=K\cdot[F]\leq3\sum_j\max\{g_j-1,0\}.
\]

If equality held, any summand of genus at least two would attain
equality in Toledo's inequality. Its image
$H_j=(v_j)_*\pi_1(S_j)$ would preserve a complex geodesic $L$ and
act on it cocompactly. The normalizer of $H_j$ in $\pi_1(X)$ also
preserves $L$, which is determined by its limit set. Its action on
$L$ is discrete: the restriction map from the stabilizer of $L$
is proper with compact kernel, and this kernel meets the
torsion-free discrete group $\pi_1(X)$ trivially. Since $H_j$ acts
cocompactly, it has finite index in its normalizer.

On the other hand, $H_j$ is the image of the corresponding component
of the cut fiber. Monodromy preserves this component, so the
mapping-torus relation, with a basepoint path in the component,
gives an element $s_j$ normalizing $H_j$ and projecting to a
nonzero power of $\bar\gamma$. Some positive power of $s_j$ must
then lie in $H_j$, contradicting $f_*(H_j)=1$. This excludes equality
and proves \eqref{eq:handle-strict}.

If $\mu(\gamma)$ were reducible, some positive power would preserve
an essential curve. Cutting and capping along this single curve
gives $\sum_j\max\{g_j-1,0\}=g-2$. For $2\leq g\leq4$ this is
at most $\frac23(g-1)$, contrary to \eqref{eq:handle-strict}.
\end{proof}

We expect the pseudo-Anosov conclusion to hold more generally:

\begin{conjecture}\label{conj:all-monodromy-pa}
Let $f:X\to\mathcal B$ be a holomorphic Lefschetz fibration on a compact
complex ball quotient, with connected fibers of genus $g\geq2$ and closed
base $\mathcal B$ of positive genus. Let $\mathcal B^\circ$ be the complement
of the critical values, and let
$\mu:\pi_1(\mathcal B^\circ)\to\Mod(\Sigma_g)$ be the monodromy.
If $\gamma\in\pi_1(\mathcal B^\circ)$ has nontrivial image in
$\pi_1(\mathcal B)$, then $\mu(\gamma)$ is pseudo-Anosov.
\end{conjecture}

The pseudo-Anosov conclusion of Corollary~\ref{cor:all-handle-pa4}
carries over to finite unramified base covers and compatible unramified fiber
covers, since the invariant foliations lift. In particular it applies to the lifted
monodromy of Chen's genus-16 fibration in Example~\ref{rem:Chen},
before separating critical values.

For a holomorphic genus-19 fibration on a compact ball quotient,
\eqref{eq:handle-strict} only gives
\[\sum_j\max\{g_j-1,0\}\geq13,\]
so it does not exclude a single
reduction curve. For $X$, however,
Proposition~\ref{prop:genus19group} supplies the group hypotheses of
Proposition~\ref{prop:handle-obstruction}. The fiber image is
nontrivial: otherwise the inclusion would lift to the contractible
universal cover, contradicting its nonzero symplectic homology
class. Thus every monodromy along a loop nontrivial in $E$ has
no multitwist power, while Proposition~\ref{prop:pa19} proves
pseudo-Anosovness for the handle maps $A,B$ directly.

\subsection{Finite covers}\label{sec:covers}

Finite unramified covers preserve the equality $\eu=3\sigma$.
For a degree-$m$ cover of the base, the pullback of a genus-$g$
fibration over $\Sigma_h$ has the same fiber genus and
\[
 h'=1+m(h-1),\qquad (\eu',\sigma',n')=m(\eu,\sigma,n),
\]
where $n$ counts critical points. A degree-$d$ unramified cover
of the total space, connected on a regular fiber, instead has
fiber genus $1+d(g-1)$ over the same base and $d$ lifts of each
node. A small perturbation separates their critical values
without changing the total space.

\begin{example}[Z. Chen's genus-16 fibration]\label{rem:Chen}
The genus-16 fibration of Z.~Chen \cite{Chen}, described as $H(1)$
in \cite[Section 2]{StipsiczBMY}, is a degree-5 unramified cover
of $X_1\to G$. To see the covering explicitly, use the equations
\[
 w^5=x(x-1)^3,\qquad a^5=y(y-1)^3,\qquad b^5=y-x.
\]
This is the cyclic cover of the blown-up product of two genus-2
curves along the five graphs of the order-$5$ action. The subgroup
\[
 (w,a,b)\longmapsto(w,\zeta a,\zeta^3b)
\]
has order $5$ and fixes
$\eta=a^2b/(y-1)$, for which
$\eta^5=y^2(y-1)(y-x)$. With
\[
 Y=\frac{y-x}{y-1},\qquad
 z=\frac{w\eta(Y-1)}{x-1},
\]
the quotient is given by
\[
 w^5=x(x-1)^3,\qquad z^5=xY(Y-1)(Y-x)^2.
\]
These are precisely the equations for $X_1$, over the same base.
The action is free: at branch crossings the two sheet changes
project to the independent pairs in \eqref{g4:eq:inertia} with
$k=1$, so their span meets this subgroup trivially. The same holds along
smooth branch curves. Thus the compact covering is unramified. On a regular fiber its degree is five, giving
$g=1+5(4-1)=16$ and
\[
 (\eu,\sigma,n)=(75,25,15).
\]
Its monodromy is obtained by lifting the $k=1$ representation through
this covering, with the lifts fixed by the equations.
\end{example}

\begin{example}[A genus-127 fibration over the torus]\label{rem:127}
We can also obtain a $7$--fold cover of our genus-19 example
$\pi:X\to E$ from Borisov--Yeung's covering of the CS surface.
They describe an unramified regular $21$--fold cover with deck
group $G_{21}=\Z_7\rtimes\Z_3$
\cite[Sections~1 and~4]{BY}. Our fundamental-group identification
in Proposition~\ref{prop:genus19group} therefore gives a surjection
$\theta:\pi_1(X)\to G_{21}$.
Choosing a subgroup $K<G_{21}$ of order $3$, we obtain the
unramified $7$--fold cover $p:\widehat X\to X$ corresponding to
$\theta^{-1}(K)$.

Since $H_1(X;\Z)=\Z^2$, the surjection $\pi_*$ is the
abelianization. The fiber image is therefore the commutator
subgroup, mapping onto $[G_{21},G_{21}]=\Z_7$. Its transitive
action on the seven cosets of $K$ shows that a regular fiber
lifts to a connected $7$--fold cover.

The composition $\widehat\pi=\pi\circ p$ therefore has fibers of
genus $1+7(19-1)=127$, and
\[
 (\eu(\widehat X),\sigma(\widehat X))=(21,7).
\]
Each of the three critical fibers has seven nodes. Separating
their critical values by a small perturbation gives a genus-127
Lefschetz fibration over the torus with 21 critical points.
\end{example}

\appendix
\section{Genus-4 monodromy factors}\label{g4:app:words}

We describe the curves in the genus-4 factorizations using standard
generators of the fiber fundamental group. We first give the change from the ribbon
generators of Section~\ref{g4:sec:identification}, then list the eight
Dehn twist curves and the three vanishing cycles. We use
\[
 \pi_1(\Sigma_4)=\langle a,b,c,d,e,f,g,h\mid[a,b][c,d][e,f][g,h]=1\rangle.
\]
A bar denotes an inverse. The words specify unoriented free-homotopy
classes, so cyclic conjugation and inversion do not change the
corresponding Dehn twists.

To compare the two sets of generators, first set
\begin{equation}\label{g4:eq:rawbasis}
\begin{aligned}
 a'&=F_2^{-1}E_1^{-1}F_4^{-1}F_3,&
 b'&=F_3^{-1}F_4E_1E_2^{-1}E_3F_3,\\
 c'&=F_3^{-1}E_2^{-1}E_1F_1F_4,&
 d'&=F_4^{-1}F_1^{-1}E_1^{-1}E_2F_4E_1E_2^{-1}E_4F_4,\\
 e'&=F_4^{-1}E_1^{-1},&
 f'&=E_1F_1^{-1}E_1^{-1}E_2E_1^{-1},\\
 g'&=E_1,&h'&=F_1^{-1}.
\end{aligned}
\end{equation}
Let $q$ be the first twelve letters of $r$ in \eqref{g4:eq:r}, and put
$a=qa'q^{-1},\ldots,h=qh'q^{-1}$. Free reduction gives
\[
 [a,b][c,d][e,f][g,h]=r.
\]
The inverse change of basis is also explicit. Put
\[
 u=e'f'^{-1}e'^{-1},\qquad v=g'h'^{-1}g'^{-1}f'.
\]
Then
\[
\begin{aligned}
 E_1&=g',&F_1&=h'^{-1},\\
 E_2&=vg',&F_2&=g'^{-1}uc'^{-1}a'^{-1},\\
 E_3&=vuc'^{-1}b'c'e'f'g',&F_3&=g'^{-1}f'^{-1}e'^{-1}c'^{-1},\\
 E_4&=vud'e'g',&F_4&=g'^{-1}e'^{-1}.
\end{aligned}
\]
Substituting either set of generators into the other gives the identity.
Since the oriented boundary class is preserved, the
Dehn--Nielsen--Baer theorem realizes this change of basis by a
homeomorphism. The common conjugation does not change the
free-homotopy classes of the curves.

The eight Dehn twist curves are
\begin{equation}\label{g4:eq:literal8}
\begin{aligned}
 \al_1&\sim \overline{f}\,g\,h\,\overline{g},&
 \be_1&\sim \overline{h}\,a\,c\,e\,f\,\overline{e}\,g,\\
 \al_2&\sim \overline{e}\,\overline{c}\,\overline{b}\,c\,e\,f\,\overline{e}\,\overline{f},&
 \be_2&\sim e\,\overline{f}\,\overline{e}\,\overline{c}\,\overline{a}\,c\,e\,f,\\
 \al_3&\sim \overline{c}\,b\,c\,e\,f\,\overline{e}\,\overline{d},&
 \be_3&\sim \overline{f}\,\overline{e}\,\overline{c}\,e,\\
 \al_4&\sim g\,\overline{h}\,\overline{g}\,f\,e\,\overline{f}\,\overline{e}\,d\,e\,g,&
 \be_4&\sim \overline{g}\,\overline{e}.
\end{aligned}
\end{equation}
Figure~\ref{fig:genus4twists} shows these eight curves on the same
fiber model as Figure~\ref{fig:genus4cycles}. Following a curve through
the bands gives its word in \eqref{g4:eq:literal8}. The corresponding
Dehn twists determine $U,V$ by \eqref{g4:eq:UV}, and then $Z,D$ and the
four handle maps by \eqref{g4:eq:center}--\eqref{g4:eq:handles}.
The powers of $Z$ in \eqref{g4:eq:twisthandles} give the variations
associated with the characters $\lambda$.

\begin{figure}[tbp]
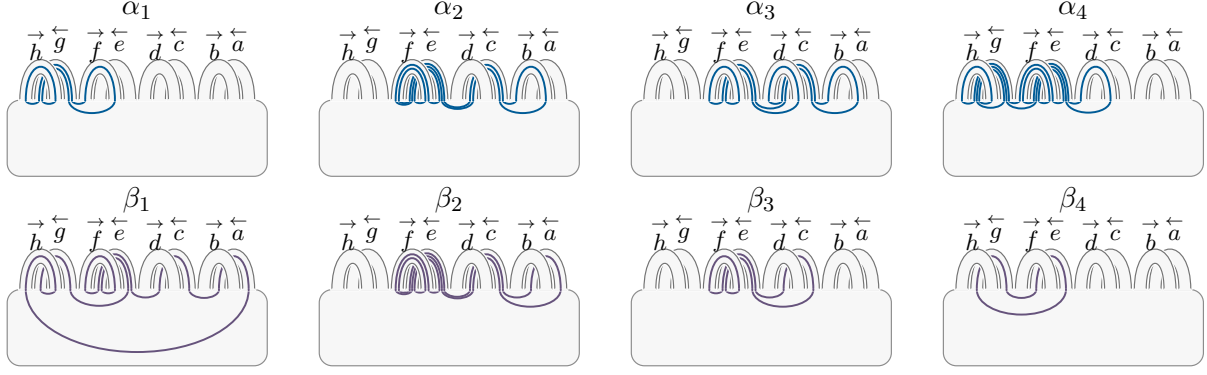

\centering
\resizebox{\textwidth}{!}{


}
\caption{The $8$ Dehn twist curves on $\Sigma_4^1$, with
$\alpha_1,\ldots,\alpha_4$ in the upper row and
$\beta_1,\ldots,\beta_4$ in the lower row. The band directions use
the same convention as Figure~\ref{fig:genus4cycles}.}
\label{fig:genus4twists}
\end{figure}

The three vanishing cycles have cyclically reduced words of lengths
$22,22,60$:
\begin{align}
 d_1\sim{}&\overline{h}\,a\,c\,e\,f\,\overline{e}\,\overline{f}\,\overline{e}\,\overline{c}\,e\,g\,h\,\overline{g}\,e\,\overline{f}\,\overline{e}\,\overline{c}\,\overline{a}\,c\,e\,f\,\overline{e},\label{g4:eq:d1literal}\\
 d_2\sim{}&g\,\overline{h}\,a\,b\,\overline{d}\,e\,f\,\overline{e}\,\overline{f}\,g\,h\,\overline{g}\,\overline{g}\,f\,e\,\overline{f}\,\overline{e}\,\overline{c}\,\overline{a}\,\overline{b}\,c\,d,\label{g4:eq:d2literal}\\
 d_3\sim{}&e\,\overline{f}\,\overline{e}\,\overline{c}\,\overline{a}\,\overline{b}\,c\,d\,e\,g\,g\,\overline{h}\,\overline{g}\,f\,e\,\overline{f}\,\overline{e}\,d\,\overline{b}\,c\nonumber\\
 &\cdot d\,\overline{f}\,g\,g\,\overline{h}\,a\,b\,c\,e\,f\,\overline{e}\,\overline{d}\,\overline{c}\,b\,a\,b\,\overline{d}\,e\,f\,\overline{e}\nonumber\\
 &\cdot \overline{f}\,g\,h\,\overline{g}\,\overline{g}\,\overline{d}\,e\,f\,\overline{e}\,\overline{f}\,g\,h\,\overline{g}\,\overline{g}\,\overline{e}\,\overline{c}\,\overline{b}\,\overline{a}\,\overline{g}\,f.\label{g4:eq:d3literal}
\end{align}
These curves are separating: by \eqref{g4:eq:cycles}, they are images
of the boundary curves of the genus-2 neighborhoods in \eqref{g4:eq:gamma}.

\section{Fundamental groups of the genus-4 real quotients}\label{app:realreduction}

We spell out the eliminations used in \eqref{eq:real4reduction} to
compute the fundamental groups of the three genus-4 real quotients.
The same two generators remain in each case, with no relations,
giving the free group of rank two.
Put $e_{s,j}=p^sE_jp^{-s}$, $f_{s,j}=p^sF_jp^{-s}$ and $a=p^5$.
We number the Schreier generators by
\[
 e_{s,j}=x_{9s+j},\quad f_{s,j}=x_{9s+4+j}
 \quad(0\leq s\leq4,\ 1\leq j\leq4),
\]
with $b_s=x_{9s+9}$ for $0\leq s\leq3$, $a=x_{45}$ and
$b_4=x_{46}$. The generators that remain at the end are $x_{27}=b_2$
and $x_{46}=b_4$.

The initial relators are obtained by rewriting the fiber relator,
the three vanishing-cycle relations, the two conjugation actions,
and $p^5,q^5,(pq)^5$ from each of the five sheets. We then impose
$x_i=r(x_i)$ for all 46 Schreier generators.
Each line below eliminates a generator using a relation $x_i=w$,
where $w$ does not contain $x_i$. This relation is obtained from a
current relator by cyclic conjugation and, when necessary, inversion.
We substitute $w$ for $x_i$ in the other relators, remove the generator
and its defining relation, and freely cancel. Apply the substitutions
from left to right and then top to bottom, including within each
indexed list $x_i=1$.

For all three values of $k$, begin with
\[
\begin{gathered}
 x_{45}=1,\quad x_9=x_{27}^{-1},\quad
 x_{36}=x_{46}^{-1},\quad x_{18}=1,\\
 x_5=x_2^{-1}x_3x_8^{-1}x_3^{-1}x_2x_6,\quad
 x_2=x_3,\quad x_1=x_4.
\end{gathered}
\]
Continue as follows.

\smallskip
\noindent For $k=0$:
\begin{align*}
&x_{6}=x_{8}\\
&x_i=1\ (i=23,32,14,41)\\
&x_{3}=x_{17}^{-1}x_{4}x_{8}x_{7}^{-1},\qquad x_{4}=x_{17}^{-1},\qquad x_{10}=x_{17}^{-1},\qquad x_{13}=x_{17}^{-1},\qquad x_{12}=x_{17}^{-1}x_{16}^{-1}\\
&x_{7}=x_{16}x_{17}^{-1}x_{8},\qquad x_{8}=x_{17},\qquad x_{15}=x_{17},\qquad x_{11}=x_{20}\\
&x_{17}=x_{24},\qquad x_{19}=x_{24}^{-1},\qquad x_{22}=x_{24}^{-1}\\
&x_{16}=x_{24}x_{38},\qquad x_{37}=x_{38},\qquad x_{20}=x_{24}^{-1}x_{25}^{-1},\qquad x_{21}=x_{33}^{-1}x_{25}^{-1},\qquad x_{24}=x_{25}x_{38}^{-1}\\
&x_i=1\ (i=25,30,33,26,29,34,38,28,31)\\
&x_i=1\ (i=35,39,40,42,43,44).
\end{align*}

\smallskip
\noindent For $k=1$:
\begin{align*}
&x_i=1\ (i=12,38)\\
&x_{6}=x_{8}\\
&x_i=1\ (i=23,32,14,41)\\
&x_{4}=x_{16}^{-1},\qquad x_{13}=x_{16},\qquad x_{15}=x_{17}^{-1},\qquad x_{28}=x_{30},\qquad x_{3}=x_{16}^{-1}x_{10},\qquad x_{10}=x_{11}\\
&x_i=1\ (i=19,31)\\
&x_{22}=x_{26}^{-1},\qquad x_{30}=x_{35}^{-1}\\
&x_i=1\ (i=11,20,26,8,7,24,25,33,17)\\
&x_i=1\ (i=16,21,29,34,35,37,39,40,42)\\
&x_i=1\ (i=43,44).
\end{align*}

\smallskip
\noindent For $k=4$:
\begin{align*}
&x_i=1\ (i=11,39)\\
&x_{6}=x_{8}\\
&x_i=1\ (i=23,32,14,41)\\
&x_{29}=x_{31},\qquad x_{4}=x_{16}^{-1}x_{8}^{-1},\qquad x_{3}=x_{16}^{-1}x_{7}^{-1},\qquad x_{7}=x_{8}x_{13}^{-1},\qquad x_{12}=x_{13}\\
&x_i=1\ (i=22,28)\\
&x_{8}=x_{13}^{-1}\\
&x_i=1\ (i=17,31,16,19,13,10,15,20,21)\\
&x_i=1\ (i=24,25,26,30,33,34,35,37,38)\\
&x_i=1\ (i=40,42,43,44).
\end{align*}

After these substitutions every fiber generator maps to $1$,
$x_9$ maps to $x_{27}^{-1}$, $x_{36}$ to $x_{46}^{-1}$, and
$x_{18},x_{45}$ to $1$. Every relator freely reduces to the empty
word. The substitution map and the inclusions of $x_{27},x_{46}$ are
inverse to one another. We therefore obtain
$\langle x_{27}\rangle\star\langle x_{46}\rangle\cong\Z\star\Z$,
with both free generators represented by loops from the base.

The supplementary code described in Appendix~\ref{g19:app:group} carries out
these eliminations and specifies the relator used in each step.

\section{Integral homology of the genus-19 fibration}\label{ssec:finite}\label{g19:app:homology}
Here we give the details of the computation of $H_1(X;\Z)$ in
Proposition~\ref{prop:invariants}. The handle actions and vanishing cycles
leave a quotient $\Z_2$ of the fiber's first homology. The relation around the
base torus supplies one further relation, which kills this quotient.
To find it, we keep track of the basepoint paths and work with
automorphisms of $\pi_1(\Sigma,*)$, rather than outer automorphisms.

We use the fiber basis $[g_1],\ldots,[g_{38}]$ and the handle loops of
Lemma~\ref{lem:schreier}. The relation matrix
\[
 M_{\mathrm{rel}}=\bigl(A_*-I\mid B_*-I\mid[\eps_0]\mid[\eps_1]\mid[\eps_2]\bigr)
\]
has Smith invariant factors $1$ repeated $37$ times and $2$, so
$\mathcal H=\operatorname{coker}M_{\mathrm{rel}}\cong\Z_2$. A functional detecting
this quotient is
\begin{equation}\label{eq:homology-functional}
 \ell\left(\sum_{j=1}^{38}n_j[g_j]\right)=\sum_{j\in J}n_j\pmod2,
 \qquad
 \begin{aligned}
 J=\{&1,2,4,6,7,8,11,14,17,18,19,\\
     &24,25,27,29,30,32,34,35,36,38\}.
 \end{aligned}
\end{equation}
Indeed, $\ell M_{\mathrm{rel}}=0$ over $\F_2$ and $\ell\ne0$.
The Schreier words of Lemma~\ref{lem:schreier} are verified by direct
rewriting.

Only the homology class of the remaining relation is needed. We
recover it by working in $\Pi/\Gamma_3\Pi$, where
$\Pi=\pi_1(\Sigma,*)$ and $\Gamma_3\Pi$ is the third term of the lower
central series. In the Magnus expansion through degree two, a word
is represented by its exponent-sum vector $a$ and its matrix $M$ of
quadratic coefficients. Multiplication is given by
$(a,M)(b,N)=(a+b,M+N+ab^{\mathsf T})$.
An automorphism with $g_k\mapsto(a_k,M_k)$ sends $(a,M)$ to
$(Aa,\ \sum_k n_kM_k+AMA^{\mathsf T})$, where
$A=(a_1\ \cdots\ a_{38})$ and $a=(n_1,\ldots,n_{38})^{\mathsf T}$.
The degree-two matrices are taken modulo $\Z\,\Omega_{\mathfrak r}$,
the class of the surface relator, which is preserved up to sign:
$A\Omega_{\mathfrak r}A^{\mathsf T}=\pm\Omega_{\mathfrak r}$.

For $\Phi=\Inn(w)$, we have
$M_k=[w]e_k^{\mathsf T}-e_k[w]^{\mathsf T}$ modulo
$\Omega_{\mathfrak r}$, with $e_k$ the $k$th standard basis vector.
Thus the degree-two action determines $[w]$.

We now apply this calculation to the lifted relation on the base.
The automorphism $\widetilde t_c$ representing the Dehn twist fixes
the chosen loop representing $c$; this equality is checked using
Dehn's algorithm. Thus
$\widetilde t_c$ and its conjugates
$\tilde P^i\widetilde t_c\tilde P^{-i}$ are geometric Dehn twists
up to inner automorphisms by powers of the vanishing loops;
these do not affect $H_1(X)$. For
\[
 \Phi=[\tilde A,\tilde B]\,\tilde t_0\tilde t_1\tilde t_2=\Inn(w)
\]
with $\tilde A=\tilde Q\tilde\rho$ and $\tilde B=\tilde\rho\tilde Q$,
this computation gives
\begin{equation}\label{eq:homology-correction}
 \ell([w])=1\in\F_2.
\end{equation}
Since $\ell([w])=1$, the remaining homology relation $[w]=0$ kills
$\mathcal H\cong\Z_2$, the quotient of $H_1(\Sigma;\Z)$ described above. The two base classes remain free, giving
$H_1(X;\Z)\cong\Z^2$. Equivalently, adjoining $[w]$ to
$M_{\mathrm{rel}}$ gives $38$ Smith invariant factors equal to $1$.
Appendix~\ref{app:geometric-pa} describes the program and data used
for this calculation.

\section{Supplementary calculations}\label{app:verification}

The preceding appendices give the genus-4 curves as words in the fiber
fundamental group, the generator eliminations for the genus-4 real
quotients, and the integral homology calculation for the genus-19
fibration. Here we describe the computer calculations used in these and
the other proofs, and explain where each is needed.

We begin with the monodromy relations and the equivalence of the two
genus-19 fibrations. Lemma~\ref{lem:genus19fibercomparison} then identifies
the presentation used in the group calculations with the constructed
fibration. We next treat the real-quotient groups and the isomorphism
$\pi_1(X)\cong\pi_1(\CS)$, and finish with the homology calculations
used to prove that the handle maps are pseudo-Anosov.
The supplement contains the required words and matrices together with
six main Python programs and two GAP programs; \file{README.md} lists their
files and explains how to run them.

\subsection{Monodromy and changes of generators}\label{app:equivalence}
For the genus-19 monodromy, \file{geometry.json} specifies the
$S_3$-surface and the local reflection surface of Section~\ref{sec:genus19}.
The program \file{geometry.py} constructs these surfaces and the actions
of their maps on the fundamental groups. It verifies the local reflection
relation, the common conjugation in Lemma~\ref{lem:h-commutes}, the
curve-orbit and power identities, and the fundamental domain in
Figure~\ref{fig:quotient}. These are the calculations used in the
periodic construction.

The comparison with the cyclic-cover description is carried out by
\file{identifications.py}, using the changes of generators in
\file{identifications.json}. For Proposition~\ref{prop:models}, the
program verifies that these changes are inverse surface-group
isomorphisms, preserve the intersection form, and carry the periodic
maps, reflection, and vanishing cycle to the stated ones. The
Dehn--Nielsen--Baer theorem then gives the fiber homeomorphism. Together
with the Hurwitz moves of Proposition~\ref{prop:second-solution}, this
proves the equivalence of fibrations in Section~\ref{ssec:models}.
The comparison of their total-space presentations is the subject of
Lemma~\ref{lem:genus19fibercomparison} below.

The word equalities are verified by free cancellation and substitutions
using the surface relator. Each substitution replaces 39 consecutive
letters of a cyclic 76-letter relator by the inverse of its remaining
37 letters, so every successful reduction gives a finite derivation
of the equality.

For genus 4, \file{genus4.py} constructs the curves and their Dehn
twist actions. It verifies the monodromy formulas of
Section~\ref{sec:genus4}, computes the five determinants used there
for integral homology, and expresses the curve words in the standard
generators of Appendix~\ref{g4:app:words}.

\subsection{Identifying the group presentation with the fibration}\label{g19:app:identification}

The group calculations use the original fiber generators
$x_1,\ldots,x_{38}$, rather than the cyclic-cover generators $y_i$ or
the geometric generators $g_i$. Proposition~\ref{prop:models} identifies
the latter two systems on the fiber. We must also check that the change
from $x_i$ to $y_i$ respects the relations involving the base loops,
including their conjugating words. The following lemma supplies this
step. It shows that the presentations used for the CS group comparison
in Proposition~\ref{prop:genus19group} and the real quotient in
Section~\ref{g19:sec:realquotient} belong to the fibration constructed
in Section~\ref{sec:genus19}.

\begin{lemma}[Compatibility of the total-space presentations]\label{lem:genus19fibercomparison}
The change of fiber generators below extends to an isomorphism from
\eqref{g19:eq:orbifoldgroup} to the corresponding presentation in
cyclic-cover coordinates. It preserves the character to $\Z_3$ and
carries the fiber involution with its basepoint corrections.
Together with Proposition~\ref{prop:models}, it therefore identifies
$\Gamma=\ker\chi$ with $\pi_1(X)$.
\end{lemma}

\begin{proof}
Write $\Pi_x$ and $\Pi_y$ for the original fiber group and the
surface group $\Pi_B$ of Section~\ref{ssec:chain-model}.
The section \texttt{fiber\_dictionary} of \file{identifications.json}
gives the images of the generators under inverse isomorphisms
$\Phi:\Pi_x\to\Pi_y$ and $\Psi:\Pi_y\to\Pi_x$, both preserving the
intersection form.
For the original curve $c_x=x_{18}x_{32}x_{28}$ and
$\hat c=y_{38}y_{37}^{-1}$ they satisfy
\[
 \Phi(c_x)=s\hat c s^{-1},\qquad
 \Phi f_x\Phi^{-1}=\Inn(h_f)f_y\quad(f=\mathsf P,\mathsf Q,\mathsf R,\mathsf t_c,\mathsf C),
\]
where $(\mathsf P_y,\mathsf Q_y,\mathsf R_y,\mathsf t_{c,y},\mathsf C_y)$ represent
$(\hat P,\hat Q,\hat R,t_{\hat c},\hat C)$. The same file gives $s$ and the conjugating words $h_f$.
Substitution verifies the two surface relators and both inverse
compositions. For each map $f$, the displayed identity is checked on
all 38 generators with one common word $h_f$.

The change of generators
\[
 x_i\mapsto\Phi(x_i),\qquad
 p_x\mapsto h_Pp_y,\quad q_x\mapsto h_Qq_y,\quad r_x\mapsto h_Rr_y
\]
extends the identification to the total-space presentation
\eqref{g19:eq:orbifoldgroup}. The words in its cube and product
relations transform as follows:
\[
 \Phi(w_f^x)=h_ff_y(h_f)f_y^2(h_f)w_f^y,\qquad
 \Phi(w_0^x)=h_{t_c}\mathsf t_{c,y}(L)w_0^y,
\]
where the first identity is for $f=\mathsf P,\mathsf Q,\mathsf R$ and
$L=h_P\mathsf P_y(h_Q)(\mathsf P_y\mathsf Q_y)(h_R)$. After imposing
$\hat c=1$, the relations $h_{t_c}=1$ and $\mathsf t_{c,y}=\id$ follow
by cancellation with $y_{38}=y_{37}$.
The inverse substitution uses $\Psi$ and the corrections
$\Psi(h_f^{-1})$. Since these corrections are fiber words, the
character to $\Z_3$ is preserved.

For the change of generators on the base, the Hurwitz move
\[
 t_{\hat c}\hat P\hat Q\hat R
 =t_{\hat c}\hat P(\hat Q\hat R\hat Q^{-1})\hat Q
\]
preserves the cyclic pullback. Put $W=\hat Q t_{\hat c}$.
The resulting handle pair is $(\hat P W^{-1},W^2\hat P)$; after conjugating
by $\hat Q^{-1}$, the pair is
\[
 \hat Q^{-1}\hat P t_{\hat c}^{-1},\qquad
 t_{\hat c}\hat Q t_{\hat c}\hat P\hat Q.
\]
Applying $\Theta$ gives the handle pair $(A',B')$ in
\eqref{eq:chain-handles}. The character defining the triple cover is
preserved, and the real structure is carried along with these changes.
Its based action in the $y$-coordinates is $\Inn(h_C)\mathsf C_y$.
We retain this inner correction: the transported involution need not
have the same based action as $\mathsf C_y$, nor does this comparison
identify it with the original reflection $C$ over fixed base paths.
\end{proof}

\subsection{Fundamental groups of the real quotients}
\label{g19:app:group}

For the genus-19 real quotient, we start from
\eqref{g19:eq:orbifoldgroup} and the based involution $\kappa$.
The periodic actions and conjugating words are given in \file{groups.json}.
The program \file{quotients.py} verifies the involution relations
\eqref{g19:eq:basedreal} and \eqref{g19:eq:realextension}, then adds the
orbit-space relations and eliminates generators. This computes the
quotient group directly, without using the CS group identification.

Rewriting with $1,r,r^2$ gives 121 generators and 360 relators;
adding $a=\varrho (a)$ gives 481 relators. In the program, the fiber generators $r^sx_ir^{-s}$ are numbered
$40s+i$, and the seven base generators in \eqref{g19:eq:realreduction}
are numbered $39,40,79,80,119,120,121$. The eliminations leave only
the generator $r^2p$, numbered $119$, with no relator. The sequence of
substitutions is given under \texttt{real\_reduction} in \file{groups.json}
and is checked against all 481 original relators. The quotient group
is therefore infinite cyclic.

For the three genus-4 quotients, the same program carries out the
eliminations displayed in Appendix~\ref{app:realreduction}, starting
with the involutions in \file{groups.json} and the monodromies
constructed by \file{genus4.py}. It also verifies the covering data
for the genus-16 example.

\subsection{The Cartwright--Steger group}\label{g19:app:CSgroup}
The proof of $\pi_1(X)\cong\pi_1(\CS)$ in
Proposition~\ref{prop:genus19group} proceeds through two presentation
reductions. First, GAP reduces the $41$-generator presentation of
$\mathcal E$ to $3$ generators and $10$ relators. Reidemeister--Schreier
rewriting then gives a $7$-generator presentation of the index-$3$
kernel $\Gamma=\pi_1(X)$. The second GAP calculation reduces this to
$3$ generators and $13$ relators. Both reductions give the changes
of generators in both directions; these are recorded in \file{groups.json}.

The program \file{groups.py} constructs the presentations supplied to
GAP and verifies the seven kernel-generator eliminations displayed
in the proof. It then checks the final isomorphism with the CS
presentation, using the 918 elementary substitutions in \file{groups.json}.
The commands in \file{README.md} run the two GAP calculations and
compare their full results with the presentations used here.

\subsection{Homology of the genus-19 monodromy}\label{app:geometric-pa}
The homology calculations serve two purposes: they complete the
computation of $H_1(X;\Z)$ in Appendix~\ref{ssec:finite}, and they give
the information used in Proposition~\ref{prop:pa19} to rule out
reducibility of the handle maps. The program \file{homology.py} derives
the homology actions from the surface-group automorphisms constructed
by \file{geometry.py}; \file{homology.json} contains the matrices and
vectors. We describe first the action on the fiber and then the action
on the anti-invariant homology of a double cover.

For $A=QR^{-1}$, let $S$ be its matrix in the basis
$[g_1],\ldots,[g_{38}]$. Then
\[
 \det(tI-S)=(t^2+1)^2h(t)h^*(t),\qquad h^*(t)=t^{17}h(t^{-1}),
\]
where
\[
\begin{split}
 h(t)={}&t^{17}-2t^{16}+2t^{15}-2t^{14}+t^{13}+2t^{11}-t^{10}
          -2t^8+2t^7\\
        &-t^6-t^5+3t^4+t^3-2t^2+1.
\end{split}
\]
In particular, $h(1)=h^*(1)=2$, so $\det(I-A_*)=16$.
Modulo $13$, the irreducible factor degrees of $h$ are $4,13$;
modulo $19$, they are $2,15$. No proper factor degree is compatible
with both, so $h$ is irreducible over $\Q$. Neither $h$ nor $h^*$
is cyclotomic, since their degree is odd and greater than one.
Consequently, every rational homology class with finite $S$-orbit
lies in $K=\ker(S^2+I)$. Every $S$-invariant rational $2$--plane
also lies in $K$, since the characteristic polynomial of the
restriction must be $t^2+1$.

The kernel $K$ has dimension two. To describe its integral lattice,
we use the matrix $N$ in \file{homology.json}. Its entries are integers,
its last two rows form $I_2$, and
\[
 SN=N\begin{pmatrix}0&1\\-1&0\end{pmatrix},\qquad
 N^{\mathsf T}\Omega N=0,
\]
where $\Omega$ is the intersection form. Thus its columns form a basis
of $L=K\cap H_1(\Sigma_{19};\Z)$ and $K$ is isotropic.
The character $\eta$ takes value $1$ on precisely the generators
with indices
\[
\begin{split}
 &3,4,7,8,9,12,14,15,16,17,18,19,20,21,\\
 &23,25,26,27,29,30,31,32,37,38.
\end{split}
\]
Direct multiplication gives $\eta S=\eta$ and $\eta N=0$ modulo $2$.
The two rows of $N^{\mathsf T}\Omega$, together with $\eta$, have
rank $3$ modulo $2$, proving $\eta\notin\operatorname{PD}(\bar L)$.

We next compute the action on the double cover determined by $\eta$.
On anti-invariant homology, the deck involution acts by $-1$, so the
corresponding character sends $g_i$ to $(-1)^{\eta(g_i)}$. Evaluating
the Fox derivatives of the automorphism representing $A$ at these
signs computes the induced action.
This gives a $38\times38$ matrix $F$ on the twisted cellular
$1$--chains. Let $b_i=(-1)^{\eta(g_i)}-1$, and let $v$ be the column
obtained by evaluating the Fox derivatives of the surface relator
with the same signs. Then
\[
 bF=b,\qquad Fv=v,\qquad V^-\cong\ker b/\langle v\rangle.
\]
The resulting characteristic polynomial on $V^-$ is
\begin{equation}\label{eq:cover-pa-poly}
 r(t)=\sum_{j=0}^{17}c_j(t^{36-j}+t^j)-14t^{18},
\end{equation}
with coefficient list
\[
\begin{split}
 (c_0,\ldots,c_{17})={}&(1,0,-1,-1,-2,4,5,1,-4,\\
                       &-9,0,7,1,15,-7,-31,12,6).
\end{split}
\]
It has value $-20$ at $1$. Its irreducible factor degrees modulo
$17$ are $16,20$, and those modulo $37$ are $3,3,30$. These admit no
common proper factor degree, so $r$ is irreducible over $\Q$.

The program \file{homology.py} verifies the finite-field
factorizations used above.

\bigskip

\bibliographystyle{amsplain}
\bibliography{refs}
\enlargethispage{\baselineskip}
\end{document}